\documentclass[12pt]{article}
\usepackage[vmargin=1.18in,hmargin=1.1in]{geometry}				
\usepackage{amssymb,amsfonts,amsthm,mathtools,bm,bbm,mathrsfs} 	
\usepackage{pifont} 											
\usepackage[dvipsnames]{xcolor} 								
\usepackage[hidelinks]{hyperref} 								
\usepackage{graphicx} 											
\usepackage{caption} 											
\usepackage{verbatim}
\usepackage[misc]{ifsym} 										
\usepackage[ruled, linesnumbered]{algorithm2e} 					
\usepackage{layouts} 
\usepackage[uniquename=false, 
			uniquelist = false, 
			url = false, 
			doi = false, 
			isbn = false, 
			natbib = true, 
			backend = bibtex, 
			style = authoryear, 
			maxcitenames = 2,
			maxbibnames = 5]{biblatex} 							
\usepackage{booktabs} 
\usepackage{multirow} 
\usepackage[noend]{algpseudocode}
\usepackage{subcaption}
\usepackage{multirow}

\makeatother

\newtheorem{assumption}{Assumption}
\newtheorem{theorem}{Theorem}
\newtheorem{corollary}{Corollary}
\newtheorem{lemma}{Lemma}
\theoremstyle{definition}
\newtheorem{remark}{Remark}
\newtheorem{proposition}{Proposition}
\newtheorem{definition}{Definition}

\def\P{\mathbb{P}}

\def\sgn{\mathrm{sgn}}
\def\P{\mathbb{P}}

\newcommand{\E}{\mathbb E}								
\newcommand{\V}{\mathrm{Var}}							
\renewcommand{\P}{\mathbb{P}}							
\newcommand{\R}{\mathbb{R}}								
\newcommand{\indicator}{\mathbbm 1}						
\newcommand{\iidsim}{\stackrel{\mathrm{i.i.d.}}{\sim}} 	
\newcommand{\indsim}{\stackrel{\mathrm{ind}}{\sim}}		
\newcommand{\convp}{\overset {\mathbb P} \rightarrow}   

\let\oldnl\nl
\newcommand{\nonl}{\renewcommand{\nl}{\let\nl\oldnl}} 

\title{The saddlepoint approximation for averages of \\ conditionally independent random variables\footnote{This work has been incorporated into the manuscript ``The conditional saddlepoint approximation for fast and accurate large-scale hypothesis testing'', which is available at \href{https://arxiv.org/abs/2407.08911}{https://arxiv.org/abs/2407.08911}. We preserve this manuscript on arXiv for those who desire a shorter, self-contained exposition of the conditional saddlepoint approximation.} }

\begin{document}

\author{Ziang Niu, Jyotishka Ray Choudhury, Eugene Katsevich}
\maketitle

\begin{abstract}
	Motivated by the application of saddlepoint approximations to resampling-based statistical tests, we prove that the Lugannani-Rice formula has vanishing relative error when applied to approximate conditional tail probabilities of averages of conditionally independent random variables. In a departure from existing work, this result is valid under only sub-exponential assumptions on the summands, and does not require any assumptions on their smoothness or lattice structure.  	
	The derived saddlepoint approximation result can be directly applied to resampling-based hypothesis tests, including bootstrap, sign-flipping and conditional randomization tests. We exemplify this by providing the first rigorous justification of a saddlepoint approximation for the sign-flipping test of symmetry about the origin, initially proposed in 1955. On the way to our main result, we establish a conditional Berry-Esseen inequality for sums of conditionally independent random variables, which may be of independent interest.
\end{abstract}

\section{Introduction} \label{sec:introduction-2}

\subsection{Saddlepoint approximations to resampling-based tests} \label{sec:saddlepoint-approximations}

Resampling-based hypothesis testing procedures often have superior finite-sample performance to those of their asymptotic counterparts. Examples of these include the sign-flipping test for symmetry around the origin in the one-sample problem \citep{Fisher1935}, permutation tests for equality of distributions in the two-sample problem \citep{Fisher1935,Pitman1937a}, bootstrap-based tests \citep{Efron1994}, and the conditional randomization test for conditional independence \citep{CetL16}. However, such procedures entail a much greater computational burden. This is particularly the case in modern applications where large numbers of hypotheses are tested simultaneously and multiplicity corrections substantially decrease the $p$-value threshold for significance. If $M$ hypotheses are tested, then $O(M)$ resamples are required per test for sufficiently fine-grained $p$-values to meet the significance threshold, leading to a total computational cost of $O(M^2)$ resamples. In genome-wide association studies, for example, $M$ can be in the millions, making resampling-based tests computationally infeasible.

A promising direction for accelerating resampling-based hypothesis testing is via closed-form approximations to tail probabilities of the resampling distribution, circumventing resampling while maintaining the finite-sample accuracy of the test. The most common approach to obtaining such approximations is via the \textit{saddlepoint approximation} (SPA), introduced by \citet{Daniels1954} for approximating densities of sample averages and extended to tail probabilities by \citet{Lugannani1980}. These approximations are accurate for small sample sizes, even for approximating extreme tail probabilities. They are derived by exponentially tilting the distribution of the summands to match the observed value of the test statistic (requiring the solution of a \textit{saddlepoint equation}), and then applying a normal approximation to the tilted distribution. In the context of resampling, the SPA has been applied to sign-flipping tests \citep{Daniels1955}, permutation tests \citep{Robinson1982,Abd2007,Abd-Elfattah2009,Zhou2013}, and the bootstrap \citep{Hinkley1988,Jing1994,Diciccio1994}.

Despite the excellent empirical performance of SPAs for resampling-based tests, their theoretical justification has three key weaknesses. \textbf{First,} existing theory ignores the conditioning on the data inherent in resampling-based tests. Instead, the data are treated as fixed, and results for unconditional SPAs are invoked. Furthermore, existing results on saddlepoint approximations to $p$-values for resampling-based tests do not account for the fact that the tail probability cutoff is a function of the data, rather than fixed. \textbf{Second,} existing SPA theory requires assumptions on either the smoothness \citep{Daniels1954,Lugannani1980,jensen1995saddlepoint,Jing1994} or lattice structure (\cite[Chapter 6.4]{jensen1995saddlepoint} and \cite[Chapter 5]{Kolassa2006}) of the summands. These assumptions need not be satisfied by the summands arising from resampling-based tests. Existing works on SPAs for resampling-based procedures assume that resampling distributions are ``close enough'' to being continuous \citep{Hinkley1988,Daniels1991,Diciccio1994}. \textbf{Third,} existing works on SPAs for resampling-based tests assume without proof that the saddlepoint equation has a solution (a requirement for the construction of the SPA). Taken together, these limitations leave the application of SPAs to resampling-based tests on shaky theoretical ground.
\subsection{Our contributions}

In this paper, \textbf{we address all three of these limitations by establishing a saddlepoint approximation for the conditional tail probability of averages of conditionally independent random variables.} In particular, we consider a triangular array $W_{in}$ of random variables that are independent and mean-zero conditionally on a $\sigma$-algebra $\mathcal F_n$ for each $n$. We prove that the Lugannani-Rice approximation $\hat p_n$ to the conditional tail probability 
\begin{equation} \label{eq:tail-probability}
p_n \equiv \P\left[\left.\frac{1}{n}\sum_{i = 1}^n W_{in} \geq w_n\ \right| \ \mathcal F_n\right]
\end{equation}
has vanishing relative error:
\begin{equation} \label{eq:relative-error}
p_n = \hat p_n \cdot (1 + o_{\P}(1)),
\end{equation}
where $w_n \in \mathcal F_n$ is a sequence of cutoff values such that $w_n \convp 0$ at any rate (Theorem~\ref{thm:unified_unnormalized_moment_conditions}). Showing that the \textit{relative} rather than \textit{absolute} error is small is particularly important, giving accurate estimates even for very small $p$-values. \textbf{We avoid assumptions on the smoothness or lattice structure of the $W_{in}$,} assuming just that these random variables have light enough tails. Under these conditions, we also prove that the saddlepoint equation has a unique solution with probability approaching one. 

Consider a hypothesis test based on the statistic $\smash{T_n \equiv \frac{1}{n}\sum_{i = 1}^n X_{in}}$ and the resampling distribution $\smash{\widetilde T_n \equiv \frac{1}{n}\sum_{i = 1}^n \widetilde X_{in}}$ for some resampling mechanism generating $\smash{\widetilde X_{in} \mid \sigma(X_{1n}, \dots, X_{nn})}$. Setting $\smash{W_{in} \equiv \widetilde X_{in}}$, $w_n \equiv T_n$, and $\mathcal F_n \equiv \sigma(X_{1n}, \dots, X_{nn})$, we find that the $p$-value of this test,
\begin{equation} \label{eq:resampling-p-value}
p_n \equiv \P\left[\left.\frac{1}{n}\sum_{i = 1}^n \widetilde X_{in} \geq \frac{1}{n}\sum_{i = 1}^n X_{in}\ \right|\ X_{1n}, \dots, X_{nn}\right],
\end{equation}
matches the definition~\eqref{eq:tail-probability}. Therefore, the result~\eqref{eq:relative-error} lays a solid mathematical foundation for applying the SPA to a range of resampling-based hypothesis tests with independently resampled summands, including the sign-flipping test, bootstrap-based tests, and the conditional randomization test (but not the permutation test). We exemplify this by providing \textbf{the first rigorous justification of an SPA for the sign-flipping test} of the kind originally proposed by \citet{Daniels1955} (Theorem~\ref{thm:example}). In addition to justifying existing SPAs for resampling methods, our result paves the way for SPAs for newer resampling-based methods like the conditional randomization test, a direction we pursue in a parallel work \citep{Niu2024a}. Beyond the context of resampling-based tests, \textbf{our result gives the first set of conditions for validity of the SPA involving only tail assumptions on the summands.} On the way to our result, we prove a conditional Berry-Esseen inequality for the sum of conditionally independent random variables (Lemma \ref{lem:conditional-berry-esseen}), which may be of independent interest. Finally, we prove the equivalence of two tail probability approximations \citep{Lugannani1980, Robinson1982} under general conditions (Proposition~\ref{prop:equivalence_spa_formula}).

\subsection{Related work}

The literature on SPAs is vast, and has several book-length reviews \citep{jensen1995saddlepoint,Kolassa2006,Butler2007}. We do not attempt to review this literature here, but we at least point to several relevant strands of work. We note that SPAs have been employed for conditional probabilities arising not just in resampling-based procedures but also in the context of inference conditional on sufficient statistics for nuisance parameters \citep{Skovgaard1987, JingRobinson1994, Davison1996, Kolassa2007a}. This direction is related to our work, but distinct in the sense that the summands are independent \textit{before} conditioning, whereas in our work, the summands are independent \textit{after} conditioning. There has also been work on applying the SPA to handle the null distributions of rank-based tests \citep{Froda2000,Robinson2003}, whose discreteness is a challenge as in resampling-based tests. In this line of work, however, there is no conditioning involved. Finally, \citet{Daniels1954} discusses conditions under which the saddlepoint equation has a unique solution, including when the summands $W_{in}$ have compact support and the cutoff $w_n$ falls within the interior of that support.

\subsection{Overview of the paper}

The organization of the remaining sections is as follows. In Section \ref{sec:spa_theory}, we state Theorem~\ref{thm:unified_unnormalized_moment_conditions}, our main theoretical result. In Section~\ref{sec:spa_motivating_example}, we apply Theorem~\ref{thm:unified_unnormalized_moment_conditions} to the sign-flipping test. Next, we sketch the proof of Theorem~\ref{thm:unified_unnormalized_moment_conditions} in Section~\ref{sec:spa_proof}, deferring several lemmas to the appendix. Finally, we conclude with a discussion in Section~\ref{sec:discussion}.

\section{SPA for conditionally independent variables}\label{sec:spa_theory}

Let $\{W_{in}\}_{1 \leq i \leq n, n \geq 1}$ be a triangular array of random variables on a probability space $(\Omega, \mathcal F, \P)$ and let $\mathcal{F}_n \subseteq \mathcal F$ be a sequence of $\sigma$-algebras so that $\E[W_{in}|\mathcal{F}_n] = 0$ for each $(i, n)$ and $\{W_{in}\}_{1 \leq i \leq n}$ are independent conditionally on $\mathcal{F}_n$ for each $n$. For a sequence of cutoff values $w_n \in \mathcal F_n$, we will consider SPAs for the conditional tail probability~\eqref{eq:tail-probability}. We will first introduce tail assumptions on $W_{in}$ (Section~\ref{sec:tail-assumptions}), which will prepare us to introduce the saddlepoint approximation (Section~\ref{sec:saddlepoint-approximation}) and state our main result on its accuracy (Section~\ref{sec:main-results}). We close this section with a number of remarks (Section~\ref{sec:remarks}).

\subsection{Tail assumptions on the summands} \label{sec:tail-assumptions}

Let $W$ be a random variable on $(\Omega, \mathcal F, \P)$ and let $\mathcal G \subseteq \mathcal F$ be a $\sigma$-algebra. In the following two definitions, we extend the definitions of sub-exponential and compactly supported random variables to the conditional setting.

\begin{definition}[CSE distribution]\label{def:cse_distribution}
Consider a random variable $\theta \in \mathcal G$ such that $\theta \geq 0$ almost surely and a constant $\beta > 0$. We say $W|\mathcal{G}$ is \textit{conditionally sub-exponential} (CSE) with parameters $(\theta, \beta)$ if, almost surely,
\begin{align*}
\P\left[|W|\geq t \mid \mathcal{G}\right]\leq \theta\exp(-\beta t),\ \text{for all } t>0.
\end{align*}
We denote this property via $W|\mathcal G \sim \text{CSE}(\theta, \beta)$.
\end{definition}

\begin{definition}[CCS distribution]\label{def:ccs_distribution}
	Consider a random variable $\nu \in \mathcal G$ such that $\nu \geq 0$ almost surely. We say $W|\mathcal{G}$ is \textit{conditionally compactly supported} (CCS) on $[-\nu, \nu]$ if
\begin{align*}
W \in [-\nu,\nu] \quad \text{almost surely}.
\end{align*}
We denote this property via $W|\mathcal G \sim \text{CCS}(\nu)$.
\end{definition}

Now, we impose assumptions on the triangular array $\{W_{in}\}_{1 \leq i \leq n, n \geq 1}$ in terms of these definitions. We assume throughout that Assumption~\ref{assu:cse} or Assumption~\ref{assu:ccs} holds.
\begin{assumption}[CSE condition]\label{assu:cse}
There exist $\theta_n \in \mathcal F_n$ and $\beta > 0$ such that 
\begin{equation}
W_{in}|\mathcal F_n \sim \textnormal{CSE}(\theta_n, \beta) \text{ for all } i, n, \quad \theta_n < \infty \text{ almost surely}, \quad \theta_n = O_{\P}(1).
\end{equation}
\end{assumption}
\begin{assumption}[CCS condition]\label{assu:ccs}
There exist $\nu_{in} \in \mathcal F_n$ such that 
\begin{align}
W_{in}|\mathcal F_n \sim \textnormal{CCS}(\nu_{in}), \quad \nu_{in}<\infty\ \text{almost surely}, \quad \text{and} \quad \frac{1}{n}\sum_{i=1}^n \nu_{in}^4=O_{\P}(1).
\end{align}
\end{assumption}

\subsection{The saddlepoint approximation} \label{sec:saddlepoint-approximation}

The saddlepoint approximation is usually based on the cumulant-generating functions (CGFs) of the summands. For our purposes, we use conditional CGFs:
\begin{equation*}
K_{in}(s) \equiv \log \E[\exp(sW_{in})|\mathcal F_n] \quad \text{and} \quad K_n(s) \equiv \frac{1}{n}\sum_{i = 1}^n K_{in}(s).
\end{equation*}
Either of the assumptions in the previous section guarantees the existence of these CGFs and their derivatives in a neighborhood of the origin.
\begin{lemma}\label{lem:finite_cgf}
Suppose Assumption~\ref{assu:cse} or Assumption~\ref{assu:ccs} holds. Then, there exists a probability-one event $\mathcal A$ and an $\varepsilon > 0$ such that, on $\mathcal A$,
\begin{align}
K_{in}(s) < \infty\quad \text{for any } s\in (-\varepsilon,\varepsilon)\ \text{and}\ \text{for all}\ i \leq n,\ n \geq 1 \label{eq:finite_cgf}
\end{align}
and 
\begin{equation}\label{eq:finite_cgf_derivatives}
	|K_{in}^{(r)}(s)| < \infty\quad \text{for any } s\in (-\varepsilon,\varepsilon)\ \text{and}\ \text{for all}\ i \leq n,\ n \geq 1, \ r \geq 1, r\in\mathbb{N},
\end{equation}
where $K_{in}^{(r)}$ denotes the $r$-th derivative of $K_{in}$.
\end{lemma}

The first step of the saddlepoint approximation is to find the solution $\hat s_n$ to the \textit{saddlepoint equation}:
\begin{equation}\label{eq:saddlepoint-equation}
K_n'(s) = w_n.
\end{equation}
In particular, we restrict our attention to solutions in the interval $[-\varepsilon/2, \varepsilon/2]$:
\begin{align*}
S_n\equiv \left\{s\in [-\varepsilon/2,\varepsilon/2]:K_n'(s)=w_n\right\}.
\end{align*}
It is possible, for specific realizations of $K'_n(s)$ and $w_n$, that the set $S_n$ is either empty or contains multiple elements. To make the saddlepoint approximation well-defined in these cases, we define $\hat s_n$ as follows:
\begin{equation}\label{eq:def_s_n}
\hat s_n \equiv 
\begin{cases}
\text{the single element of }S_n & \text{if } |S_n|=1; \\
\frac{\varepsilon}{2}\mathrm{sgn}(w_n) & \text{otherwise}.
\end{cases}
\end{equation}
Note that this definition ensures that $\hat s_n \in [-\varepsilon/2, \varepsilon/2]$.

With the solution to the saddlepoint equation in hand, we define the saddlepoint approximation as follows. If we define
\small
\begin{align}\label{eq:lam_n_r_n_def}
  \lambda_n \equiv \hat s_n\sqrt{nK_n''(\hat s_n)}; \quad r_n \equiv
  \begin{cases}
	\sgn(\hat s_n) \sqrt{2n( \hat s_n w_n - K_n(\hat s_n))} & \text{if } \hat s_n w_n - K_n(\hat s_n)\geq 0;\\
	\mathrm{sgn}(\hat s_n) & \text{otherwise},
  \end{cases}
  \end{align}
\normalsize
then the saddlepoint approximation is given by
\begin{equation}
\widehat{\P}_{\text{LR}}\left[\left.\frac{1}{n}\sum_{i = 1}^n W_{in} \geq w_n\ \right|\ \mathcal{F}_n\right] \equiv 1-\Phi(r_n)+\phi(r_n)\left\{\frac{1}{\lambda_n}-\frac{1}{r_n}\right\}. \label{eq:lugannani-rice}
\end{equation}
When $w_n = 0$, we have $\hat s_n = \lambda_n = r_n = 0$. In this case, we take by convention that $1/0 - 1/0 \equiv 0$ in equation~\eqref{eq:conclusion_saddlepoint_approximation}, so that $\widehat{\P}_{\text{LR}} \equiv 1/2$. We will also use the convention $0/0=1$. The approximation $\widehat{\P}_{\text{LR}}$~\eqref{eq:lugannani-rice} is a direct generalization of the classical Lugannani-Rice formula \citep{Lugannani1980} to the conditional setting.

\subsection{Statement of main results} \label{sec:main-results}

In preparation for the theorem statement, recall the following definitions for a sequence of random variables $W_n$ and probability measures $\P_n$:
\begin{align*}
&W_n = O_{\P_n}(1) &&\text{ if for each } \delta > 0 \text{ there is an } M > 0 \text{ s.t. } \limsup_{n \rightarrow \infty}\P_n[|W_n| > M] < \delta; \\
&W_n = \Omega_{\P_n}(1) &&\text{ if for each } \delta > 0 \text{ there is an } \eta > 0 \text{ s.t. } \limsup_{n \rightarrow \infty}\P_n[|W_n| < \eta] < \delta;\\
&W_n = o_{\P_n}(1) &&\text{ if } \P_n[|W_n| > \eta] \rightarrow 0 \text{ for all } \eta > 0.
\end{align*}

\begin{theorem}\label{thm:unified_unnormalized_moment_conditions}
	Let $W_{in}$ be a triangular array of random variables that are mean-zero and independent for each $n$, conditionally on $\mathcal F_n$. Suppose either Assumption \ref{assu:cse} or Assumption \ref{assu:ccs} holds, and that
	\begin{align}\label{eq:lower_bound_conditional_variance}
		\frac{1}{n}\sum_{i=1}^n \E[W_{in}^2 \mid \mathcal{F}_n]=\Omega_{\P}(1).
	\end{align}
	Let $w_n \in \mathcal F_n$ be a sequence with $w_n \overset{\P} \rightarrow 0$. Then, the saddlepoint equation~\eqref{eq:saddlepoint-equation} has a unique and finite solution $\hat s_n \in [-\varepsilon/2, \varepsilon/2]$ with probability approaching 1 as $n \rightarrow \infty$:
	\begin{equation}
	\lim_{n \rightarrow \infty} \P[|S_n| = 1] = 1.
	\label{eq:unique_solution_in_probability}
	\end{equation}
	Furthermore, the saddlepoint approximation $\widehat{\P}_{\textnormal{LR}}$ to the conditional tail probability~\eqref{eq:tail-probability} defined by equations \eqref{eq:lam_n_r_n_def} and~\eqref{eq:lugannani-rice} has vanishing relative error:
	\begin{align}\label{eq:conclusion_saddlepoint_approximation}
	  \P\left[\left.\frac{1}{n}\sum_{i = 1}^n W_{in} \geq w_n\ \right|\ \mathcal{F}_n\right] = \widehat{\P}_{\textnormal{LR}}\left[\left.\frac{1}{n}\sum_{i = 1}^n W_{in} \geq w_n\ \right|\ \mathcal{F}_n\right](1+o_{\P}(1)).
	\end{align}
\end{theorem}

If we set $\mathcal{F}_n \equiv \{\varnothing,\Omega\}$, Theorem~\ref{thm:unified_unnormalized_moment_conditions} reduces to the following variant of the classical \citet{Lugannani1980} result:

\begin{corollary}\label{cor:unconditional_LR_formula}
	Let $W_{in}$ be a triangular array of random variables that are mean-zero and independent for each $n$. Suppose that each $W_{in}$ is sub-exponential with constants $\theta, \beta > 0$, i.e. 
	\begin{equation} \label{eq:subexponential}
	\P[|W_{in}| \geq t] \leq \theta \exp(-\beta t) \quad \text{for all } t > 0.
	\end{equation} 
	Furthermore, suppose that
	\begin{align}  \label{eq:nondegeneracy-unconditional}
	\smash{\liminf_{n \rightarrow \infty}\ \frac{1}{n}\sum_{i = 1}^n \E[W_{in}^2] > 0.}
	\end{align}
	Given a sequence of cutoffs $w_n \rightarrow 0$ as $n \rightarrow \infty$, the saddlepoint equation~\eqref{eq:saddlepoint-equation} (based on the unconditional CGF $K_n$) has a unique solution $\hat s_n$ on $[-\varepsilon/2, \varepsilon/2]$ for all sufficiently large $n$. Furthermore, the unconditional saddlepoint approximation $\widehat{\P}_{\textnormal{LR}}$ (based on the unconditional CGF $K_n$) has vanishing relative error
	\begin{align} \label{eq:conclusion_unconditional_saddlepoint_approximation}
	\P\left[\frac{1}{n}\sum_{i = 1}^n W_{in} \geq w_n\right] = \widehat{\P}_{\textnormal{LR}}\left[\frac{1}{n}\sum_{i = 1}^n W_{in} \geq w_n\right](1+o(1)).
	\end{align}
\end{corollary}

\subsection{Remarks on Theorem~\ref{thm:unified_unnormalized_moment_conditions} and Corollary~\ref{cor:unconditional_LR_formula}} \label{sec:remarks}

\begin{remark}[Generality and transparency of assumptions]

While most of the existing literature is fragmented based on whether the summands $W_{in}$ are smooth or lattice, Theorem~\ref{thm:unified_unnormalized_moment_conditions} and Corollary~\ref{cor:unconditional_LR_formula} unify these two cases by not making any such assumptions. The price we pay for this generality is that our guarantee~\eqref{eq:conclusion_saddlepoint_approximation} does not come with a rate, as compared to most existing results. Nevertheless, we find it useful to have a single result encompassing a broad range of settings. Our assumptions are not only general but also transparent. Most existing results require complicated and difficult-to-verify conditions, often stated in terms of transforms of $W_{in}$. By contrast, consider for example the assumptions~\eqref{eq:subexponential} (subexponential summands) and~\eqref{eq:nondegeneracy-unconditional} (non-degenerate variance) in Corollary~\ref{cor:unconditional_LR_formula}. These are standard and easy to understand and verify.
\end{remark}

\begin{remark}[Assumptions on the threshold $w_n$]
	The assumption $w_n=o_\P(1)$ in Theorem \ref{thm:unified_unnormalized_moment_conditions} appears to be restrictive at the first glance. However, this assumption allows us to guarantee the existence of a solution the saddlepoint equation beyond the case of compact support \citep{Daniels1954}. On the other hand, the rate of convergence of $w_n$ towards zero can be arbitrary, accommodating for two statistically meaningful regimes: (a) Moderate deviation regime: $w_n=O_{\P}(n^{-\alpha}),\alpha\in (0,1/2)$ and (b) CLT regime: $w_n=O_{\P}(1/\sqrt{n})$. We discuss this further in Remark~\ref{rem:mu_n_convergence} after Theorem \ref{thm:example}.
\end{remark}

\begin{remark}[Relative error guarantee]
Our results~\eqref{eq:conclusion_saddlepoint_approximation} and~\eqref{eq:conclusion_unconditional_saddlepoint_approximation} give explicit relative error guarantees. Even though informal statements of relative error bounds on tail probabilities are common in the literature, bounds are often stated in terms of absolute error instead, omitting rigorous statements of relative error bounds. Relative error bounds are particularly desirable in statistical applications when the true conditional probabilities are small.
\end{remark}

\begin{remark}[Technical challenges]

The proof of Theorem \ref{thm:unified_unnormalized_moment_conditions} presents several technical challenges, requiring us to significantly extend existing proof techniques. One of the most significant challenges is the conditioning, which adds an extra layer of randomness to the problem. For example, the saddlepoint equation~\eqref{eq:saddlepoint-equation} is a random equation, with random solution $\hat s_n$, which we must show exists with probability approaching 1. Furthermore, the cutoff $w_n$ is random, and in particular we must handle cases depending on the realization of the sign of this cutoff. A crucial step in our proof (which follows the general structure of that of \cite{Robinson1982}) is to use the Berry-Esseen inequality, but the extra conditioning requires us to prove a new conditional Berry-Esseen theorem. Another challenge is that we allow $w_n$ to decay to zero at an arbitrary rate, which requires a delicate analysis of the convergence of the SPA formula appearing in the RHS of the result \eqref{eq:conclusion_saddlepoint_approximation}. 
\end{remark}

\begin{remark}[Form of tail probability approximation]

As discussed above, our tail probability approximation~\eqref{eq:lugannani-rice} is a conditional adaptation of the Lugannani-Rice formula \citep{Lugannani1980}, which is the most common tail probability approximation in the literature. Aside from the Lugannani-Rice formula, another tail probability estimate is proposed in \cite{Robinson1982}. In fact, we present an extension of Theorem~\ref{thm:unified_unnormalized_moment_conditions} in Proposition \ref{prop:equivalence_spa_formula} that employs a conditional variant of Robinson's formula:
\begin{equation}
\widehat{\P}_\text{R}\left[\left.\frac{1}{n}\sum_{i = 1}^n W_{in} \geq w_n\ \right|\ \mathcal{F}_n\right] \equiv \exp\left(\frac{\lambda_n^2-r_n^2}{2}\right)(1-\Phi(\lambda_n)).
\label{eq:robinson-formula}
\end{equation}
\begin{proposition}\label{prop:equivalence_spa_formula}
	Under the assumptions of Theorem \ref{thm:unified_unnormalized_moment_conditions},   
	\begin{align}\label{eq:alternative_spa_formula}
		\P\left[\left.\frac{1}{n}\sum_{i = 1}^n W_{in} \geq w_n\ \right|\ \mathcal{F}_n\right]=\widehat{\P}_\text{R}\left[\left.\frac{1}{n}\sum_{i = 1}^n W_{in} \geq w_n\ \right|\ \mathcal{F}_n\right](1+o_\P(1)).
	  \end{align}
\end{proposition}
\noindent In other words, $\widehat{\P}_\text{LR}$~\eqref{eq:lugannani-rice} is equivalent to $\widehat{\P}_\text{R}$~\eqref{eq:robinson-formula} with relative error $o_\P(1)$, linking Robinson's formula to that of Lugannani and Rice. A similar result was proved in the unconditional case by \citet{Kolassa2007}.
\end{remark}

\section{Application: Sign-flipping test}\label{sec:spa_motivating_example}

In this section, we apply Theorem~\ref{thm:unified_unnormalized_moment_conditions} to derive and justify the validity of the Lugannani-Rice SPA for the sign-flipping test. 

\subsection{The sign-flipping test}

Suppose
\begin{equation}
X_{in}=\mu_n+\varepsilon_{in}, \quad \varepsilon_{in} \indsim F_{in}, 
\label{eq:location-model-symmetric-errors}
\end{equation}
where the error distributions $F_{in}$ are symmetric, but potentially distinct and unknown. We are interested in testing 
\begin{equation}
H_{0n}:\mu_n=0 \quad \text{versus} \quad H_{1n}:\mu_n>0
\end{equation}
based on $T_n \equiv \frac{1}{n}\sum_{i=1}^n X_{in}$. Note that the SPA cannot directly be applied to approximate tail probabilities of $T_n$ because the error distributions $F_{in}$ are unknown. Instead, we can approximate the tail probability of $T_n$ by conditioning on the observed data and resampling the signs of the data. In particular, define the resamples
\begin{equation}
\widetilde X_{in} \equiv \pi_{in} X_{in}, \quad \pi_{in} \iidsim \text{Rad}(0.5),
\end{equation}
where $\text{Rad}(0.5)$ denotes the Rademacher distribution placing equal probability mass on $\pm 1$. Due to the assumed symmetry of the distributions $F_{in}$, flipping the signs of $X_{in}$ preserves their distributions under the null hypothesis, guaranteeing finite-sample validity of the resampling-based $p$-value from equation~\eqref{eq:resampling-p-value} \citep{Hemerik2018, Hemerik2019a}. To circumvent the computationally costly resampling inherent in the sign-flipping test, we can obtain an accurate approximation to the $p$-value by applying the SPA to the tail probabilities of the resampling distribution
\begin{equation}
\widetilde T_n \equiv \frac1n \sum_{i=1}^n \widetilde X_{in} \equiv \frac1n \sum_{i=1}^n \pi_{in} X_{in}.
\end{equation}
Such approximations have been proposed before \citep{Daniels1955, Robinson1982, Hinkley1988}, but have not been rigorously justified (see also Section~\ref{sec:saddlepoint-approximations}). We will now apply Theorem~\ref{thm:unified_unnormalized_moment_conditions} to derive and justify the Lugananni-Rice SPA for the sign-flipping test.

\subsection{The SPA for the sign-flipping test}

We derive the saddlepoint approximation $\widehat{\P}_\text{LR}$. Defining $\mathcal F_n \equiv \sigma(X_{1n}, \dots, X_{nn})$, we first calculate the conditional cumulant-generating functions
\begin{align*}
	K_{in}(s)\equiv \log\E\left[\exp(s\widetilde{X}_{in})|\mathcal{F}_n\right]=\log\left(\frac{\exp(sX_{in}) + \exp(-sX_{in})}{2}\right) = \log\cosh(sX_{in})
\end{align*}
and their first two derivatives
\begin{align}\label{eq:tilde_K_in_prime}
	K'_{in}(s)=X_{in}-\frac{2X_{in}}{1+\exp(2sX_{in})} \quad \text{and} \quad K''_{in}(s)=\frac{4X_{in}^2\exp(2sX_{in})}{(1+\exp(2sX_{in}))^2}.
\end{align}
Therefore, the saddlepoint equation \eqref{eq:saddlepoint-equation} reduces to
\begin{align}\label{eq:saddlepoint_equation_example_simplified}
\frac{1}{n}\sum_{i = 1}^n K'_{in}(s) = \frac{1}{n}\sum_{i = 1}^n X_i \quad \Longleftrightarrow \quad \sum_{i=1}^n \frac{X_{in}}{1+\exp(2sX_{in})}=0.
\end{align}
Given a solution $\hat s_n$ to the saddlepoint equation (whose existence and uniqueness is guaranteed by Theorem~\ref{thm:example} below), we can define the quantities $\lambda_n$ and $r_n$ from equation~\eqref{eq:lam_n_r_n_def}:
\begin{equation}
\lambda_n \equiv \hat s_n \sqrt{n K''_{n}(\hat s_n)} = \hat s_n\sqrt{\sum_{i=1}^n\frac{4X_{in}^2\exp(2\hat s_nX_{in})}{(1+\exp(2\hat s_nX_{in}))^2}}
\label{eq:lambda_n_sign_flipping}
\end{equation}
and
\begin{equation}
\begin{split}
r_n \equiv \sgn(\hat s_n) \sqrt{2n( \hat s_n w_n - K_n(\hat s_n))} = \sgn(\hat s_n) \sqrt{2\sum_{i=1}^n (\hat s_n X_{in} - \log \cosh(\hat s_n X_{in}))},
\label{eq:r_n_sign_flipping}
\end{split}
\end{equation}
where we set $r_n \equiv \sgn(\hat s_n)$ when the quantity under the square root is negative. With these definitions, the SPA for the tail probability of interest is
\begin{equation}
\widehat{\P}_{\text{LR}} \left[\left.\frac{1}{n}\sum_{i=1}^n \widetilde{X}_{in} \geq \frac{1}{n}\sum_{i=1}^n X_{in}\ \right|\ \mathcal{F}_n\right] \equiv 1-\Phi(r_n)+\phi(r_n)\left\{\frac{1}{\lambda_n}-\frac{1}{r_n}\right\}.
\label{eq:sign-flipping-spa}
\end{equation}

\subsection{Guarantees for the SPA for the sign-flipping test}

The following theorem (proved in Appendix~\ref{sec:proof_example}) gives sufficient conditions for this saddlepoint approximation to have vanishing relative error.

\begin{theorem}\label{thm:example}
Suppose $X_{in}$ are drawn from the probability model~\eqref{eq:location-model-symmetric-errors}, such that
	\begin{align}
		\liminf_{n\rightarrow\infty}\frac{1}{n}\sum_{i=1}^n \E[\varepsilon_{in}^2]>0;&\label{eq:lower_bound_second_moment} \\
		\text{there exists $\delta>0$ such that }\limsup_{n\rightarrow\infty}\frac{1}{n}\sum_{i=1}^n\E[|\varepsilon_{in}|^{4+\delta}]<\infty;&\label{eq:upper_bound_four_delta_moment} \\
		\mu_n=o(1).&\label{eq:mu_n_convergence}
	\end{align}
	Then, the saddlepoint equation \eqref{eq:saddlepoint_equation_example_simplified} has a unique solution $\hat s_n\in [-1,1]$ with probability approaching 1 as $n\rightarrow \infty$. Furthermore, the tail probability approximation~\eqref{eq:sign-flipping-spa} obtained from equations~\eqref{eq:lambda_n_sign_flipping} and~\eqref{eq:r_n_sign_flipping} has vanishing relative error:
	\small
	\begin{align}\label{eq:conclusion_example}
	\P\left[\frac{1}{n}\sum_{i=1}^n \widetilde{X}_{in}\geq \frac{1}{n}\sum_{i=1}^n X_{in}\mid \mathcal{F}_n\right]=\widehat{\P}_{\textnormal{LR}}\left[\left.\frac{1}{n}\sum_{i=1}^n \widetilde{X}_{in} \geq \frac{1}{n}\sum_{i=1}^n X_{in}\ \right|\ \mathcal{F}_n\right](1+o_{\P}(1)).
	\end{align}
	\normalsize
\end{theorem}

\begin{remark}
To our knowledge, this is the first rigorous result providing conditions under which an SPA for the sign-flipping test has vanishing relative error, even though SPAs for this test were first proposed almost seventy years ago \citep{Daniels1955}.
\end{remark}

\begin{remark}
	Conditions \eqref{eq:lower_bound_second_moment} and \eqref{eq:upper_bound_four_delta_moment} in Theorem \ref{thm:example} are very mild. Condition \eqref{eq:lower_bound_second_moment} requires the non-degeneracy of the average of the second moment of the noise distribution. Condition \eqref{eq:upper_bound_four_delta_moment} imposes average $4+\delta$ moment condition on the noise distribution. Therefore, the SPA for the sign-flipping test handles very general noise distributions; in particular, no assumption is required on the cumulant generating functions of the noise distributions.
\end{remark}

\begin{remark} \label{rem:mu_n_convergence}
	Condition \eqref{eq:mu_n_convergence} implies that the alternative hypothesis must ``converge'' to the null hypothesis. This is mainly required when proving the existence of the unique solution to the saddlepoint equation \eqref{eq:saddlepoint_equation_example_simplified}. Under the moderation deviation regime ($\mu_n=O(n^{-\alpha}),\alpha\in (0,1/2)$), the power of the test with p-value defined by the LHS of \eqref{eq:conclusion_example} will converge to $1$ under mild conditions. For the CLT regime ($\mu_n=h/\sqrt{n},h\neq 0$), it is known as contiguous local alternative in statistical testing literature. This is a particularly interesting regime because the asymptotic power of the resampling-based test is strictly between $\alpha$ and $1$ where $\alpha$ is the prespecified significance level, under mild conditions. 
\end{remark}

\section{Proof of Theorem~\ref{thm:unified_unnormalized_moment_conditions}} \label{sec:spa_proof}

The high-level structure of our proof is inspired by that of \citet{Robinson1982}: Exponentially tilt the summands, then apply the Berry-Esseen inequality to get a normal approximation after tilting, then tilt back. In this section, we sketch the proof our main result with the help of a sequence of lemmas, whose proofs we defer to Appendix~\ref{sec:lemma_proofs}. 

\subsection{Solving the saddlepoint equation}\label{sec:solution_spa}

First, we state a lemma lower-bounding the second derivative $K''_n(s)$, which will help us guarantee the existence and uniqueness of solutions to the saddlepoint equation~\eqref{eq:saddlepoint-equation}.
\begin{lemma} \label{lem:positive_second_derivative}
Under the assumptions in Theorem \ref{thm:unified_unnormalized_moment_conditions}, the function $K''_n(s)$ is nonnegative on $(-\varepsilon, \varepsilon)$:
\begin{align}
K''_n(s) \geq 0 \quad \text{for all } s \in (-\varepsilon, \varepsilon) \text{ almost surely}. \label{eq:positive_second_derivative}
\end{align}
Furthermore, it is uniformly bounded away from zero on a neighborhood of the origin, in the sense that for each $\delta > 0$, there exist $\eta > 0, s^* \in (0, \varepsilon/2)$ and $N \in \mathbb N_+$ such that 
\begin{equation}
\P\left[\inf_{s \in [-s_*, s_*]} K''_n(s) \geq \eta\right] \geq 1-\delta \quad \text{for all } n \geq N. \label{eq:uniform_lower_bound_second_derivative}
\end{equation}
\end{lemma}
\noindent This lemma guarantees that the function $K'_n(s)$ is nondecreasing on $(-\varepsilon, \varepsilon)$ and increasing at a positive rate near the origin. To better illustrate the intuition, we refer the reader to Figure \ref{fig:illustration_spsolution}. Since $w_n \convp 0$, this implies that the saddlepoint equation~\eqref{eq:saddlepoint-equation} will have a solution for large enough $n$. 

\begin{figure}[!ht]
	\centering
	\includegraphics[width=1\textwidth]{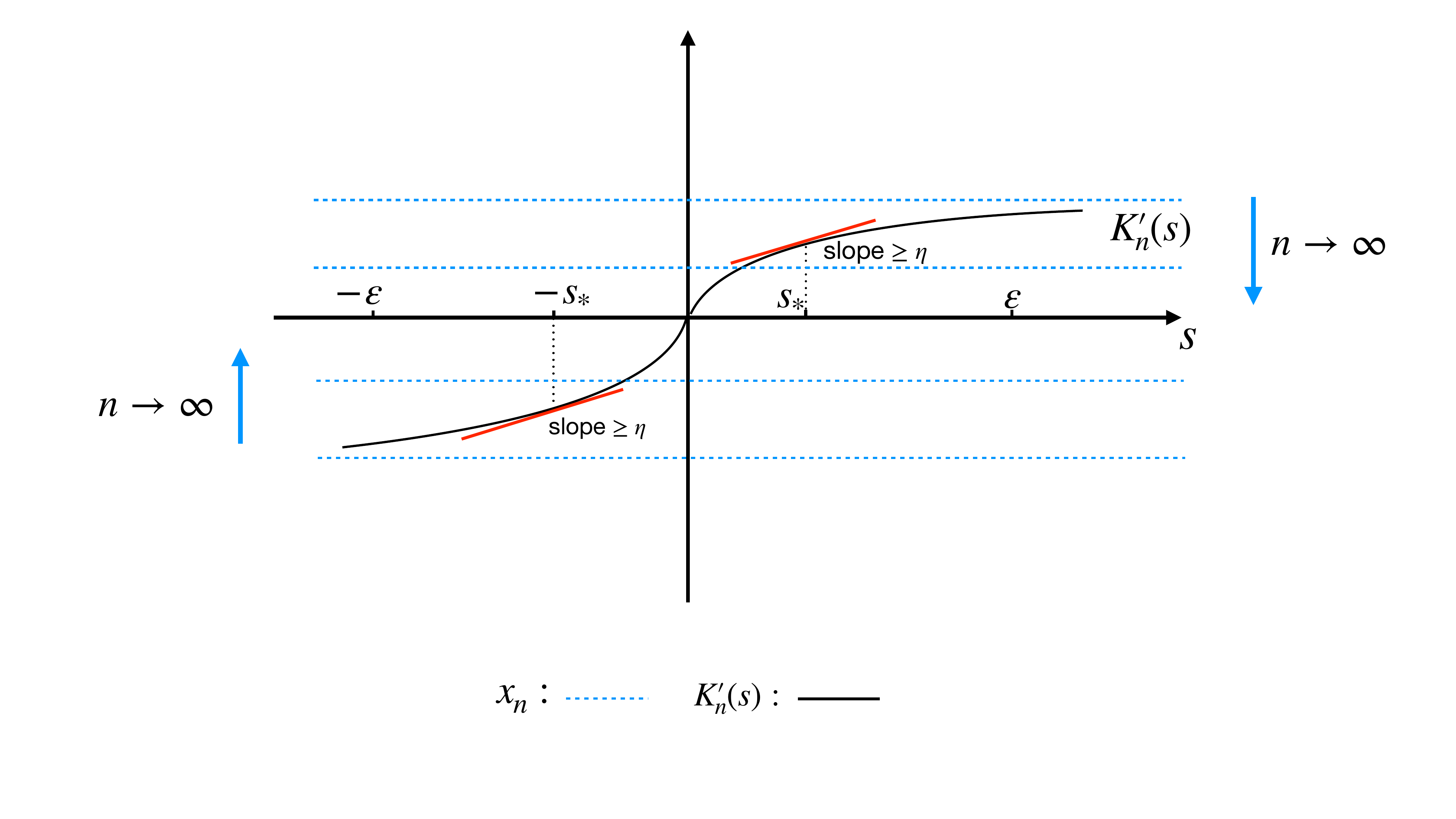}
	\caption{Illustration of the function $K'_n(s)$ for $n$ large. The derivative of the function is nonnegative and strictly positive near the origin.}
	\label{fig:illustration_spsolution}
\end{figure}

To be more precise, fix $\delta > 0$. By Lemma~\ref{lem:positive_second_derivative} and the fact that $w_n \convp 0$, let $\eta > 0$, $s_* \in (0, \varepsilon/2)$ and $N \in \mathbb N_+$ be such that
\begin{equation}
\P\left[\mathcal E_n\right] \geq 1-\delta \quad \text{for all } n \geq N, \quad \text{where} \quad \mathcal E_n \equiv \left\{\inf_{s \in [-s_*, s_*]} K''_n(s) \geq \eta, |w_n| < s_* \eta\right\}.
\end{equation}
On the event $\mathcal E_n \cap \mathcal A$, we can Taylor expand $K'_n(s)$ around $s = 0$ to obtain
\begin{equation}
K'_n(s) =  K'_n(0) + s K''_n(\bar s) = s K''_n(\bar s) \quad \text{for} \quad |\bar s| \in (0, |s|),
\label{eq:K_prime_taylor_expansion}
\end{equation}
where we have used the fact that
\begin{equation}
K'_n(0) = \frac{1}{n}\sum_{i = 1}^n K'_{in}(0) = \frac{1}{n}\sum_{i = 1}^n \E[W_{in} \mid \mathcal F_n] = 0. 
\label{eq:k_n_prime_0_equals_0}
\end{equation}
It follows from the Taylor expansion~\eqref{eq:K_prime_taylor_expansion} that, for $n \geq N$, we have
\begin{equation*}
K'_n(-s_*) \leq -s_* \eta < w_n < s_* \eta \leq K'_n(s_*).
\end{equation*}
By the continuity of $K'_n(s)$ on the event $\mathcal A$ (Lemma~\ref{lem:finite_cgf}), the intermediate value theorem implies that for each $n \geq N$, there exists a solution $\hat s_n \in (-s_*, s_*)$ to the saddlepoint equation~\eqref{eq:saddlepoint-equation} on the event $\mathcal E_n \cap \mathcal A$. Furthermore, for each $n \geq N$, this solution is unique on $\mathcal E_n \cap \mathcal A$ because $K'_n(s)$ is strictly increasing on $[-s_*, s_*]$ and nondecreasing on the entire interval $[-\varepsilon/2, \varepsilon/2]$. Hence, we have shown that, for arbitrary $\delta > 0$, we have
\begin{equation}
\liminf_{n \rightarrow \infty} \P[|S_n| = 1] \geq \liminf_{n \rightarrow \infty} \P[\mathcal E_n \cap \mathcal A] \geq 1 - \delta.
\end{equation}
Letting $\delta \rightarrow 0$ implies the first claim~\eqref{eq:unique_solution_in_probability} of Theorem~\ref{thm:unified_unnormalized_moment_conditions}. 

The following lemma records three properties of the saddlepoint $\hat s_n$, which will be useful in the remainder of the proof:
\begin{lemma} \label{lem:saddlepoint_properties}
The saddlepoint $\hat s_n$ satisfies the following properties:
\begin{align}
\sgn(\hat s_n)=\sgn(w_n)\ \text{almost surely} ;\label{eq:sign_property_s_n}  \\
\hat s_n \convp 0; \label{eq:hat_s_n_convergence} \\
K''_n(\hat s_n) = \Omega_{\P}(1). \label{eq:hat_s_n_second_derivative}
\end{align}
\end{lemma}

\subsection{Decomposing based on the sign of $w_n$}\label{sec:sign_decomposition}

Since $w_n$ is random, it can have uncertainty on the sign. This is a technical challenge since the uncertain sign of $w_n$ will also make the signs of $\lambda_n,r_n$ uncertain. We observe that the desired result~\eqref{eq:conclusion_saddlepoint_approximation} is implied by the following three statements, which decompose the problem based on the sign of $w_n$:
\begin{align}
\indicator(w_n > 0)\left(\frac{\P\left[\frac1n \sum_{i = 1}^n W_{in} \geq w_n \mid \mathcal F_n\right]}{1-\Phi(r_n)+\phi(r_n)\left\{\frac{1}{\lambda_n}-\frac{1}{r_n}\right\}}-1\right) \convp 0; \label{eq:positive_w_n} \\
\indicator(w_n < 0)\left(\frac{\P\left[\frac1n \sum_{i = 1}^n W_{in} \geq w_n \mid \mathcal F_n\right]}{1-\Phi(r_n)+\phi(r_n)\left\{\frac{1}{\lambda_n}-\frac{1}{r_n}\right\}}-1\right) \convp 0; \label{eq:negative_w_n} \\
\P\left[\frac1n \sum_{i = 1}^n W_{in} \geq 0 \mid \mathcal F_n \right] \convp \frac12. \label{eq:zero_w_n} 
\end{align}
In the next two subsections, we verify statements~\eqref{eq:negative_w_n} and~\eqref{eq:zero_w_n}, respectively. This will leave just the statement~\eqref{eq:positive_w_n}. 

\subsubsection{Verifying statement~\eqref{eq:negative_w_n}}

Before verifying statement~\eqref{eq:negative_w_n}, we state a lemma on the properties of the quantities $\lambda_n$ and $r_n$ that are necessary for proving the statement:
\begin{lemma}\label{lem:additional_properties_r_n_lambda_n}
Under the assumptions of Theorem \ref{thm:unified_unnormalized_moment_conditions}, $r_n$ and $\lambda_n$ are almost surely finite:
	\begin{align}
		|r_n|<\infty, |\lambda_n|<\infty,\text{ a.s.} \tag{Finite} \label{eq:finitness_r_n_lambda_n}
	\end{align}
Furthermore, the signs of $w_n$, $r_n$, and $\lambda_n$ have the following relationships:
	\begin{align}
		w_n > 0 \Rightarrow \lambda_n \geq 0, r_n \geq 0,\text{ a.s.}; \tag{Sign1} \label{eq:sign_1} \\
		\P[w_n>0\text{ and }\lambda_nr_n = 0]\rightarrow 0; \tag{Sign2} \label{eq:same_sign_condition} \\
		\P[\hat s_n\neq 0\text{ and }\lambda_nr_n = 0]\rightarrow 0; \tag{Sign3} \label{eq:same_sign_condition_sn} \\
		r_n < 0\Rightarrow \lambda_n\leq 0,\ \lambda_n < 0 \Rightarrow r_n \leq 0,\text{ a.s.} \tag{Sign4} \label{eq:sign_condition_r_lambda}
	\end{align}
Finally, $r_n$ and $\lambda_n$ satisfy the following convergence statements:
	\begin{align}
		\frac{1}{\lambda_n}-\frac{1}{r_n}=o_{\P}(1); \tag{Rate1} \label{eq:rate_1} \\ 
		\frac{\lambda_n}{r_n}-1=o_{\P}(1); \tag{Rate2} \label{eq:rate_2} \\ 
		\indicator(r_n>0,\lambda_n>0)\frac{1}{r_n}\left(\frac{\lambda_n}{r_n}-1\right)=o_{\P}(1); \tag{Rate3} \label{eq:rate_3} \\ 
		\indicator(\lambda_n\neq 0)\frac{1}{\lambda_n}\left(\frac{r_n}{\lambda_n}-1\right)=o_{\P}(1); \tag{Rate4} \label{eq:rate_4} \\ 
		\frac{r_n}{\sqrt{n}}=o_{\P}(1)\tag{Rate5}. \label{eq:rate_r}
	\end{align}
\end{lemma}
\noindent Now we claim that if the statement~\eqref{eq:positive_w_n} holds, then we can derive the statement~\eqref{eq:negative_w_n} by symmetry:

\begin{lemma} \label{lem:symmetry}
Suppose the assumptions of Theorem~\ref{thm:unified_unnormalized_moment_conditions} hold and imply statement~\eqref{eq:positive_w_n}. We can apply the theorem to the triangular array $\widetilde W_{in} \equiv -W_{in}$ and set of cutoffs $\widetilde w_n \equiv -w_n$ to obtain that 
\begin{equation*}
\indicator(\widetilde w_n > 0)\left(\frac{\P\left[\frac1n \sum_{i = 1}^n \widetilde W_{in} \geq \widetilde w_n \mid \mathcal F_n\right]}{1-\Phi(\widetilde r_n)+\phi(\widetilde r_n)\left\{\frac{1}{\widetilde \lambda_n}-\frac{1}{\widetilde r_n}\right\}}-1\right) \convp 0,
\end{equation*}
where $\widetilde r_n = -r_n$ and $\widetilde \lambda_n = -\lambda_n$. Then under conditions \eqref{eq:finitness_r_n_lambda_n}, \eqref{eq:sign_1}, \eqref{eq:sign_condition_r_lambda} and \eqref{eq:rate_1}, the following convergence statement holds:
\begin{align}
\indicator(w_n < 0)\left(\frac{\P\left[\frac1n \sum_{i = 1}^n W_{in} \geq w_n \mid \mathcal F_n\right]}{1-\Phi(r_n)+\phi(r_n)\left\{\frac{1}{\lambda_n}-\frac{1}{r_n}\right\}}-1\right) \convp 0.
\end{align}
\end{lemma}

\subsubsection{Verifying statement~\eqref{eq:zero_w_n}}

To prove the statement~\eqref{eq:zero_w_n}, we first state a conditional central limit theorem:
\begin{lemma}[\cite{Niu2022a}] \label{lem:conditional-clt}
	Consider a sequence of $\sigma$-algebras $\mathcal F_n$ and probability measures $\P_n$. Let $W_{in}$ be a triangular array of random variables, such that for each $n$, $W_{in}$ are independent conditionally on $\mathcal F_n$ under $\P_n$. Let
	\begin{equation*}
		S_n^2 \equiv \frac1n\sum_{i = 1}^n \V_{\P_n}[W_{in} \mid \mathcal F_n].
	\end{equation*} 
	If $\V_{\P_n}[W_{in} \mid \mathcal F_n] < \infty$ almost surely for each $i$ and $n$, and for some $\delta > 0$ we have
	\begin{equation}\label{eq:conditional-lyapunov}
		n^{-\delta/2}\frac{1}{S_n^{2+\delta}} \frac{1}{n}\sum_{i = 1}^n \E_{\P_n}[|W_{in}-\E_{\P_n}[W_{in}|\mathcal{F}_n]|^{2+\delta} \mid \mathcal{F}_n] \overset{\P_n} \rightarrow 0,
	\end{equation}
	then 
	\begin{equation*}
		\sup_{z\in\mathbb{R}}\left|\P_n\left[\sqrt{\frac{n}{S^2_n}} \frac1n\sum_{i = 1}^n (W_{in} - \E_{\P_n}[W_{in} \mid \mathcal{F}_n])\leq z \mid \mathcal F_n\right]-\Phi(z)\right| \overset{\P_n} \rightarrow 0.
	\end{equation*}
\end{lemma}
\noindent We can apply this result to the variables $W_{in}$ to get the following convergence statements:
\begin{lemma}\label{lem:conditional_CLT_W_n}
	Suppose the conditions in Theorem \ref{thm:unified_unnormalized_moment_conditions} hold. Then we have 
	\begin{align}\label{eq:conditional_uniform_CLT}
		\sup_{t\in\mathbb{R}}\left|\P\left[\sqrt{\frac{n}{K_n''(0)}}\frac1n\sum_{i=1}^n W_{in}\leq t|\mathcal{F}_n\right]-\Phi(t)\right|\convp 0.
	\end{align}
	Moreover, for any sequence $y_n\in\mathcal{F}_n$, we know 
	\begin{align}\label{eq:nondegeneracy}
		\P\left[\sqrt{\frac{n}{K_n''(0)}}\frac1n\sum_{i=1}^n W_{in}=y_n|\mathcal{F}_n\right]\convp 0.
	\end{align}
\end{lemma}
\noindent Setting $t=0$ and $y_n=0$ in Lemma \ref{lem:conditional_CLT_W_n}, we obtain 
\begin{equation*}
\begin{split}
	\P\left[\frac{1}{n}\sum_{i=1}^n W_{in}\geq  0|\mathcal{F}_n\right] &= 1 - \P\left[\frac{1}{n}\sum_{i=1}^n W_{in} \leq 0|\mathcal{F}_n\right] + \P\left[\frac{1}{n}\sum_{i=1}^n W_{in} = 0|\mathcal{F}_n\right] \\
	&\convp 1 - \frac{1}{2} + 0 = \frac12,
\end{split}
\end{equation*}
which verifies \eqref{eq:zero_w_n}. 

\subsection{Conditional Berry-Esseen bound on tilted summands}

It remains to prove the statement~\eqref{eq:positive_w_n} regarding the conditional tail probability $\P\left[\frac1n \sum_{i = 1}^n W_{in} \geq w_n \mid \mathcal F_n\right]$. This tail probability can be approximated using the conditional central limit theorem (Lemma~\ref{lem:conditional-clt}). However, the central limit theorem is insufficiently accurate in the tails of the distribution. To overcome this challenge, we apply a normal approximation after exponential tilting, as is common in saddlepoint approximations \citep{Robinson1982, Reid1988a}. The idea is to consider a probability distribution $\P_n$ over the space such that 
\begin{equation}
\frac{1}{n}\sum_{i = 1}^n \E_{\P_n}[W_{in} \mid \mathcal F_n] = w_n.  
\label{eq:tilted-mean}
\end{equation}
Under such $\P_n$, the distribution of $\frac1n \sum_{i = 1}^n W_{in}$ can be approximated as a normal with conditional mean $w_n$, allowing us to avoid approximating extreme tail probabilities. We can then undo the exponential tilting to approximate the desired tail probability under the original measure $\P$.

\subsubsection{Exponential tilting}

Given tilting parameter $s$, define a new probability measure $\P_{n,s}$ on the measurable space $(\Omega, \mathcal F)$ via 
\begin{align}\label{eq:tilted_measure}
    \frac{\mathrm{d}\P_{n,s}}{\mathrm{d}\P}\equiv \prod_{i = 1}^n\frac{\exp(s W_{in})}{\E[\exp(s W_{in})|\mathcal F_n]}.
\end{align}
We employ a variant of tilting measure \eqref{eq:tilted_measure} based on a random tilting parameter $s_n \in \mathcal F_n$ that satisfies the criterion $\P[s_n \in (-\varepsilon, \varepsilon)] = 1$. The following lemma presents some properties of the tilted measure $\P_{n,s_n}$:

\begin{lemma} \label{lem:tilted_measure_properties}
	First, events in $\mathcal F_n$ are preserved under $\P_{n,s_n}$:
	\begin{equation}
	\P_{n,s_n}[A_n] = \P[A_n] \quad \text{for all} \quad A_n \in \mathcal F_n.
	\label{eq:preserving_measurable_events}
	\end{equation}
	It follows that any random variable measurable with respect to $\mathcal F_n$ has the same distribution under $\P_{n,s_n}$ as under $\P$. Second, the random variables $\{W_{in}\}_{1 \leq i \leq n}$ are independent conditionally on $\mathcal F_n$ under $\P_{n,s_n}$. Third, on the event $\mathcal A$, the conditional mean and variance of $W_{in}$ under $\P_{n,s_n}$ are given by the first two derivatives of the conditional cumulant generating function $K_{in}$:
	\begin{align}
	\E_{n, s_n}[W_{in} \mid \mathcal{F}_n]=K'_{in}(s_n) \text{ and } \V_{n, s_n}[W_{in} \mid \mathcal{F}_n]=K''_{in}(s_n) \text{ for all }i\leq n, n\geq 1
	\label{eq:conditional_moments}
	\end{align}
	almost surely.
\end{lemma} 

It follows from equation~\eqref{eq:conditional_moments} that, almost surely,
\begin{equation*}
	\frac{1}{n}\sum_{i = 1}^n \E_{n, s_n}[W_{in} \mid \mathcal F_n] = K'_n(s_n) \quad \text{and} \quad \frac{1}{n}\sum_{i = 1}^n \V_{n, s_n}[W_{in} \mid \mathcal F_n] = K''_n(s_n).
\end{equation*} 
To ensure the property~\eqref{eq:tilted-mean}, it suffices to take $\P_n \equiv \P_{n, \hat s_n}$, where $\hat s_n$ is the solution to the saddlepoint equation~\eqref{eq:saddlepoint-equation}. Therefore, our next step is to construct a normal approximation for the average $\frac1n \sum_{i = 1}^n W_{in}$ under the sequence of tilted probability measures $\P_{n, \hat s_n}$.

\subsubsection{Conditional Berry-Esseen}

It turns out that rate of the normal approximation is important to obtain a relative error bound, so we use the conditional Berry-Esseen theorem rather than the central limit theorem on the tilted summands.
\begin{lemma}[Conditional Berry-Esseen theorem]\label{lem:conditional-berry-esseen}
	Suppose $W_{1n},\ldots,W_{nn}$ are independent random variables conditional on $\mathcal{F}_n$, under $\P_n$. If
	\begin{align}
	S_n^2 \equiv \frac{1}{n}\sum_{i=1}^n \V_{\P_n}[W_{in} \mid \mathcal{F}_n] = \Omega_{\P_n}(1) \label{eq:variance-bounded-below}
	\end{align}
	and
	\begin{align}
	\frac{1}{n}\sum_{i=1}^n \E_{\P_n}[|W_{in}-\E_{\P_n}[W_{in}|\mathcal{F}_n]|^3|\mathcal{F}_n]=O_{\P_n}(1), \label{eq:third-moment-bound}
	\end{align}
	then
	\begin{align*}
		\sqrt{n}\sup_{t\in\mathbb{R}}\left|\P_n\left[\sqrt{\frac{n}{S_n^2}}\frac1n\sum_{i=1}^n (W_{in}-\E_{\P_n}[W_{in}|\mathcal{F}_n])\leq t|\mathcal{F}_n\right]-\Phi(t)\right|=O_{\P_n}(1).
	\end{align*}
\end{lemma}
Now, we wish to apply the Lemma~\ref{lem:conditional-berry-esseen} to the triangular array $\{W_{in}\}_{1 \leq i \leq n, n \geq 1}$ under the sequence of tilted probability measures $\P_{n, \hat s_n}$. The following lemma shows that the requisite conditions are satisfied.
\begin{lemma} \label{lem:conditional_clt_assumptions}
	Under the assumptions of Theorem~\ref{thm:unified_unnormalized_moment_conditions}, the conditions~\eqref{eq:variance-bounded-below} and~\eqref{eq:third-moment-bound} are satisfied by the sequence of probability measures $\P_n \equiv \P_{n, \hat s_n}$.
\end{lemma}
\noindent Noting from equation~\eqref{eq:conditional_moments} that $S_n^2 = K''_n(\hat s_n)$, we conclude from the conditional Berry-Esseen theorem that
\begin{equation}
\begin{split}
&\sqrt{n}\sup_{t\in\mathbb{R}}\left|\P_n\left[\sqrt{\frac{n}{K''_n(\hat s_n)}}\left(\frac1n\sum_{i=1}^n W_{in}-K'_n(\hat s_n)\right)\leq t|\mathcal{F}_n\right]-\Phi(t)\right| \\
&\quad \equiv \sqrt{n}\sup_{t\in\mathbb{R}}\left|\P_n\left[\widetilde Z_n \leq t \mid \mathcal{F}_n\right]-\Phi(t)\right| \\
&\quad= O_{\P_n}(1),
\end{split}
\end{equation}
where we have denoted by $\widetilde Z_n$ the quantity converging to the standard normal distribution. Note that $\widetilde Z_n$ is not exactly the same as 
\begin{equation}
Z_n\equiv \sqrt{\frac{n}{K_n''(\hat s_n)}}\left(\frac{1}{n}\sum_{i=1}^n W_{in}-w_n\right),
\label{eq:Z_n}
\end{equation}
since it is possible that $K'_n(\hat s_n) \neq w_n$. Since the probability of this event is tending to zero~\eqref{eq:unique_solution_in_probability}, we find that
\begin{equation*}
\begin{split}
&\sqrt{n}\sup_{t\in\mathbb{R}}\left|\P_{n}[Z_n\leq t|\mathcal{F}_n]-\Phi(t)\right| \\
&\quad= \indicator(K'_n(\hat s_n) = w_n)\sqrt{n}\sup_{t\in\mathbb{R}}\left|\P_{n}[\widetilde Z_n\leq t|\mathcal{F}_n]-\Phi(t)\right| \\
&\quad \quad + \indicator(K'_n(\hat s_n) \neq w_n)\sqrt{n}\sup_{t\in\mathbb{R}}\left|\P_{n}[Z_n\leq t|\mathcal{F}_n]-\Phi(t)\right| \\
&\quad \leq \sqrt{n}\sup_{t\in\mathbb{R}}\left|\P_{n}[\widetilde Z_n\leq t|\mathcal{F}_n]-\Phi(t)\right| + \indicator(K'_n(\hat s_n) \neq w_n)\sqrt{n} \\
&\quad = O_{\P_n}(1) + o_{\P_n}(1) = O_{\P_n}(1).
\end{split}
\end{equation*}
By conclusion~\eqref{eq:preserving_measurable_events} from Lemma~\ref{lem:tilted_measure_properties} and the measurability with respect to $\mathcal F_n$ of the quantity $\sqrt{n}\sup_{t\in\mathbb{R}}\left|\P_{n}[Z_n\leq t|\mathcal{F}_n]-\Phi(t)\right|$, it follows that
\begin{equation}
\sqrt{n}\sup_{t\in\mathbb{R}}\left|\P_{n}[Z_n\leq t|\mathcal{F}_n]-\Phi(t)\right| = O_{\P}(1).
\label{eq:Z_n_convergence}
\end{equation}
Therefore, we have provided a normal approximation for the average $\frac1n \sum_{i = 1}^n W_{in}$ under the sequence of tilted probability measures $\P_{n, \hat s_n}$. Next, we undo the exponential tilting to approximate the desired tail probability under the original measure $\P$. 

\subsection{Gaussian integral approximation after tilting back}\label{sec:reduction_to_Gaussian_integral}

\subsubsection{Tilting back to the original measure}

The following lemma helps connect the tilted measure to the original one, allowing us to interchange the order of the tilting and the conditioning:
\begin{lemma} \label{lem:tilting_back}
\begin{equation}
\P\left[\left.\frac{1}{n}\sum_{i = 1}^n W_{in} \geq w_n\ \right|\ \mathcal{F}_n\right] = \E_{n, \hat s_n}\left[\indicator\left(\frac{1}{n}\sum_{i = 1}^n W_{in} \geq w_n\right)\frac{d\P}{d\P_{n, \hat s_n}} \mid \mathcal{F}_n\right], \label{eq:tilting_back}
\end{equation} 
where $\hat s_n$ is the solution to the saddlepoint equation~\eqref{eq:def_s_n} for each $n$.
\end{lemma}
\noindent To evaluate the right-hand side of equation~\eqref{eq:tilting_back}, we first note that
\begin{equation*}
\begin{split}
\frac{d\P}{d\P_{n, \hat s_n}} &= \prod_{i = 1}^n \frac{\E[\exp(\hat s_n W_{in})|\mathcal F_n]}{\exp(\hat s_n W_{in})} \\
&= \exp\left(n\left(K_n(\hat s_n) - \hat s_n \frac{1}{n}\sum_{i = 1}^n W_{in}\right)\right) \\
&= \exp\left(n(K_n(\hat s_n) - \hat s_n w_n) - \hat s_n \sqrt{nK''_n(\hat s_n)} \sqrt{\frac{n}{K''_n(\hat s_n)}} \left(\frac{1}{n}\sum_{i = 1}^n W_{in} - w_n\right)\right) \\
&\equiv \exp\left(-\frac12 r_n^2 - \lambda_n Z_n \right),
\end{split}
\end{equation*}
recalling $\lambda_n$ and $r_n$ from equation~\eqref{eq:lam_n_r_n_def} and $Z_n$~\eqref{eq:Z_n} the quantity converging to normality~\eqref{eq:Z_n_convergence}. This allows us to rewrite the probability of interest as
\begin{equation}
\begin{split}
\P\left[\frac1n \sum_{i = 1}^n W_{in} \geq w_n \mid \mathcal F_n\right] &= \E_{n, \hat s_n}\left[\indicator(Z_n \geq 0)\exp\left(-\frac12 r_n^2 - \lambda_n Z_n \right) \mid \mathcal F_n\right] \\
&= \exp\left(-\frac12 r_n^2\right)\E_{n, \hat s_n}\left[\indicator(Z_n \geq 0)\exp\left(-\lambda_n Z_n \right) \mid \mathcal F_n\right].
\end{split}
\end{equation}
Therefore, we have
\begin{equation}\label{eq:Dn_Un_Def}
\frac{\P\left[\frac1n \sum_{i = 1}^n W_{in} \geq w_n \mid \mathcal F_n\right]}{1-\Phi(r_n)+\phi(r_n)\left\{\frac{1}{\lambda_n}-\frac{1}{r_n}\right\}} = \frac{\E_{n, \hat s_n}\left[\indicator(Z_n \geq 0)\exp\left(-\lambda_n Z_n \right) \mid \mathcal F_n\right]}{\exp\left(\frac12 r_n^2\right)\left(1-\Phi(r_n)+\phi(r_n)\left\{\frac{1}{\lambda_n}-\frac{1}{r_n}\right\}\right)} \equiv \frac{D_n}{U_n}.
\end{equation}
Hence, we have simplified the desired statement~\eqref{eq:positive_w_n} to 
\begin{equation}
\indicator(w_n > 0) \left(\frac{D_n}{U_n} - 1\right) \convp 0.
\label{eq:simplified_desired_statement}
\end{equation}

\subsubsection{Reduction to a Gaussian integral approximation}

Next, we exploit the convergence of $Z_n$ to normality~\eqref{eq:Z_n_convergence} to replace the numerator $D_n$ with a Gaussian integral:
\begin{lemma}\label{lem:Gaussian_integral_approximation_additive_error}
	For sequences $Z_n$ and $\lambda_n$ of random variables, we have
	\begin{equation}\label{eq:additive_bound_Gaussian_integral}
	\begin{split}
		&\indicator(\lambda_n \geq 0)\left|\E_{\P_n}\left[\indicator(Z_n \geq 0)\exp\left(- \lambda_n Z_n \right) \mid \mathcal{F}_n\right] - \int_0^\infty \exp(-\lambda_n z)\phi(z)dz\right|\\
		&\quad \leq 2\indicator(\lambda_n \geq 0)\sup_{t\in\mathbb{R}}\left|\P_n\left[Z_n\leq t|\mathcal{F}_n \right]-\Phi(t)\right|.
	\end{split}
	\end{equation}
	almost surely.
\end{lemma}
\noindent We would like to combine the result of this lemma with the convergence of $Z_n$ to normality~\eqref{eq:Z_n_convergence} to reduce the desired statement~\eqref{eq:simplified_desired_statement} to a Gaussian integral approximation. Before doing so, we first state a result that we will use to show that the difference between $D_n$ and the Gaussian integral $\int_0^\infty \exp(-\lambda_n z)\phi(z)dz$ is negligible, even after dividing by $U_n$:
\begin{lemma}\label{lem:relative_error_Berry_Esseen_bound}
Under conditions \eqref{eq:finitness_r_n_lambda_n}, \eqref{eq:sign_condition_r_lambda} and \eqref{eq:rate_1}, we have 
\begin{align}
r_n\geq 0 \Rightarrow U_n \neq 0 \text{ almost surely};\ \P[r_n<0\text{ and }U_n=0]\rightarrow0. \label{eq:U_n_r_n}
\end{align}
Under conditions \eqref{eq:rate_1}, \eqref{eq:rate_2} and \eqref{eq:rate_r}, we have
\begin{align}
	\frac{\indicator(r_n\geq 0)}{\sqrt{n}U_n} =o_{\P}(1). \label{eq:U_n_rate}
\end{align}
\end{lemma}
\noindent Therefore, we have
\begin{align*}
&\indicator(w_n > 0) \left|\frac{D_n}{U_n} - \frac{\int_0^\infty \exp(-\lambda_n z)\phi(z)dz}{U_n}\right| \\
&\quad \leq \indicator(r_n \geq 0, \lambda_n \geq 0) \left|\frac{D_n}{U_n} - \frac{\int_0^\infty \exp(-\lambda_n z)\phi(z)dz}{U_n}\right| && \text{by}~\eqref{eq:sign_1} \\
&\quad\leq \frac{2\indicator(r_n \geq 0, \lambda_n \geq 0) \sup_{t\in\mathbb{R}}\left|\P_n\left[Z_n \leq t \mid \mathcal F_n\right] - \Phi(t)\right|}{|U_n|} && \text{by Lemma}~\ref{lem:Gaussian_integral_approximation_additive_error} \\
&\quad\leq \frac{\indicator(r_n \geq 0)}{\sqrt{n}|U_n|}O_{\P}(1) && \text{by}~\eqref{eq:Z_n_convergence} \\
&\quad= o_{\P}(1)O_{\P}(1) && \text{by}~\eqref{eq:U_n_rate}\\
&\quad= o_{\P}(1),
\end{align*}
Therefore, it suffices to show that
\begin{equation}
\indicator(w_n > 0)\left(\frac{\int_0^\infty \exp(-\lambda_n z)\phi(z)dz}{U_n}-1\right) = o_{\P}(1). \label{eq:final_step_Gaussian_integral}
\end{equation}
By statements~\eqref{eq:sign_1}, \eqref{eq:same_sign_condition}, and~\eqref{eq:U_n_r_n}, it follows that it suffices to show the Gaussian integral approximation
\begin{align}
	\indicator(r_n> 0,\lambda_n> 0)\left(\frac{\int_0^\infty \exp(-\lambda_n z)\phi(z)dz}{U_n}-1\right)=o_{\P}(1),
\label{eq:gaussian_integral_approximation}
\end{align}
which is stated in the following lemma:
\begin{lemma}\label{lem:final_result_except_lam_0}
Under conditions~\eqref{eq:finitness_r_n_lambda_n}, \eqref{eq:rate_1}, \eqref{eq:rate_2}, and~\eqref{eq:rate_3}, the Gaussian integral approximation~\eqref{eq:gaussian_integral_approximation} holds.
\end{lemma}
\noindent This completes the proof Theorem~\ref{thm:unified_unnormalized_moment_conditions}.

\section{Discussion} \label{sec:discussion}

In this paper, we presented a general sufficient condition to guarantee the saddlepoint equation to have a unique solution with probability tending $1$, and for the Lugannani-Rice saddlepoint approximation to be valid for conditionally independent but not necessarily identically distributed random variables. This result is applicable for general (conditional) distributions, placing no assumptions on their smoothness. Our results lay a solid mathematical foundation for existing applications of the saddlepoint approximation to resampling-based tests, including those based on the bootstrap and sign-flipping. Furthermore, they pave the way for applications of the saddlepoint approximation to new resampling-based tests, such as the conditional randomization test. We pursue the latter application in a parallel work \citep{Niu2024a}.

The current work has several limitations, pointing the way to future work. First, our results require conditionally independent summands, which excludes permutation tests. Second, the requirement of our results for the cutoff $w_n$ to converge to $0$ excludes the interesting case when the cutoff is bounded away from zero, i.e. $w_n=\Omega_p(1)$, which is known as the large deviation regime. Third, our results apply only to the sample average statistic but not to more complex statistics, such as studentized averages, for which saddlepoint approximations have been explored \citep{Daniels1991,Jing1994,Jing2004}. Fourth, our results give approximations for tail probabilities (corresponding to hypothesis tests) but not for quantiles (corresponding to confidence interval construction). In future work, it would be interesting to extend our results in these directions.

\section{Acknowledgments}

We are very grateful to John Kolassa, who provided valuable feedback on an earlier version of this paper. This work was partially support by NSF DMS-2113072 and NSF DMS-2310654.

\printbibliography

\clearpage

\appendix







\section{Preliminaries on regular conditional distribution}\label{sec:RCD_preliminary}

To better understand the argument involving conditional distribution, we briefly discuss the basic definition of regular conditional distribution. Let $\mathcal{B}(\mathbb{R}^n)$ be the Borel $\sigma$-algebra on $\mathbb{R}^n$ and $\Omega,\mathcal{F}_n$ be the sample space and a sequence of $\sigma$-algebras. For any $n\in\mathbb{N}_+,\kappa_n:\Omega\times \mathcal{B}(\mathbb{R}^n)$ is a regular conditional distribution of $W_n\equiv (W_{1n},\ldots,W_{nn})^\top$ given $\mathcal{F}_n$ if 
\begin{align*}
	\omega\mapsto \kappa_n(\omega,B) \text{ is measurable with respect to $\mathcal{F}_n$ for any fixed $B\in\mathcal{B}(\mathbb{R}^n)$};
	&
	\\
	B\mapsto \kappa_n(\omega,B) \text{ is a probability measure on }(\mathbb{R}^n,\mathcal{B}(\mathbb{R}^n));
	&
	\\
	\kappa_n(\omega,B)=\P[(W_{1n},\ldots,W_{nn})\in B|\mathcal{F}_n](\omega),\text{ for almost all }\omega\in\Omega\text{ and all }B\in\mathcal{B}(\mathbb{R}^n).&
\end{align*}
The following lemma from \cite[Theorem 8.37]{Lista2017} ensures that the general existence of $\kappa_{n}$.
\begin{lemma}[Theorem 8.37 in \cite{Lista2017}]\label{lem:Klenke_Thm_8.37}
  Suppose $(\Omega,\mathcal{G},\P)$ is the Probability triple. Let $\mathcal{F}\subset \mathcal{G}$ be a sub-$\sigma$-algebra. Let $Y$ be a random variable with values in a Borel space $(E,\mathcal{E})$ (for example, $E$ is Polish, $E=\mathbb{R}^d$). Then there exists a regular conditional distribution $\kappa_{Y,\mathcal{F}}$ of $Y$ given $\mathcal{F}$.
\end{lemma}
\noindent Result from \cite[Theorem 8.38]{Lista2017} guarantees that the conditional expectation and the integral of measurable function with respect to regular conditional distribution are almost surely same.

\begin{lemma}[Modified version of Theorem 8.38 in \cite{Lista2017}]\label{lem:Klenke_Thm_8.38}
  Let $X$ be a random variable $(\Omega,\mathcal{G},\mathbb{P})$ with values in a Borel space $(E,\mathcal{E})$. Let $\mathcal{F}\subset \mathcal{G}$ be a $\sigma$-algebra and let $\kappa_{X,\mathcal{F}}$ be a version of regular conditional distribution of $X$ given $\mathcal{F}$. Further, let $f:E\rightarrow\mathbb{R}$ be measurable and $\E[|f(X)|]<\infty$. Then we can define a version of the conditional expectation of $f(X)$ given $\mathcal{F}$ as:
  \begin{align*}
    \E[f(X)|\mathcal{F}](\omega)=\int f(x)\mathrm{d}\kappa_{X,\mathcal{F}}(\omega,x),\ \forall \omega \in\Omega.
  \end{align*}
\end{lemma}

\section{A version of conditional expectation}

In this paper, we will use regular conditional distribution (RCD) to fix a version of the conditional expectations used in the paper. Suppose $\kappa_{in}:\Omega\times \mathcal{B}(\mathbb{R})$ is a version of the RCD of $W_{in}$ given $\mathcal{F}_n$ and $\kappa_{n}:\Omega\times \mathcal{B}(\mathbb{R}^n)$ is a version of the RCD of $W_n\equiv (W_{1n},\ldots,W_{nn})^\top$ given $\mathcal{F}_n$. The conditional independent assumption can be formulated as 
\begin{align*}
	\kappa_{n}(\omega,B)=\prod_{i=1}^n \kappa_{in}(\omega,B_i),\ B=B_1\times\ldots\times B_n,\ \forall B_1,\ldots,B_n\in \mathcal{B}(\mathbb{R}),\ \forall \omega\in \Omega.
\end{align*}
Define the tilted RCD:
\begin{align*}
	\frac{\mathrm{d}\kappa_{in,s}(\omega,x)}{\mathrm{d}\kappa_{in}(\omega,x)}\equiv \frac{\exp(sx)}{\int \exp(sx)\mathrm{d}\kappa_{in}(\omega,x)},\ \frac{\mathrm{d}\kappa_{n,s}(\omega,x)}{\mathrm{d}\kappa_{n}(\omega,x)}\equiv\prod_{i=1}^n \frac{\exp(sx)}{\int \exp(sx)\mathrm{d}\kappa_{in}(\omega,x)},\ \forall \omega\in\Omega.
\end{align*}
Consider the measure $\P_{n,s}$ defined in \eqref{eq:tilted_measure} and define measure $\P_{in,s}$:
\begin{align*}
	\frac{\mathrm{d}\P_{in,s}}{\mathrm{d}\P} \equiv \frac{\exp(s W_{in})}{\E[\exp(s W_{in})|\mathcal F_n]}.
\end{align*}
For any measurable function $f:\mathbb{R}\mapsto\mathbb{R}$ and $g:\mathbb{R}^n\mapsto\mathbb{R}$, we define the conditional expectation under original measure $\P$, tilted measure $\P_{in,s}$ and $\P_{n,s}$ respectively as
\begin{align}
	\E[f(W_{in})|\mathcal{F}_n](\omega)
	&\label{eq:def_conditional_expectation}
	\equiv \int f(x)\mathrm{d}\kappa_{in}(\omega,x),\ \forall \omega\in\Omega,\\ 
	\E_{in,s}[f(W_{in})|\mathcal{F}_n](\omega)
	&\label{eq:def_conditional_expectation_tilted}
	\equiv\int f(x)\frac{\exp(sx)}{\int \exp(sx)\mathrm{d}\kappa_{in}(\omega,x)}\mathrm{d}\kappa_{in}(\omega,x),\ \forall \omega\in\Omega,\\
	\E_{n,s}[g(W_{n})|\mathcal{F}_n](\omega)
	&\label{eq:def_conditional_expectation_tilted_product}
	\equiv\int g(y)\left(\prod_{i=1}^n \frac{\exp(sy_i)}{\int \exp(sy_i)\mathrm{d}\kappa_{in}(\omega,y_i)}\right)\mathrm{d}\kappa_{n}(\omega,y),\ \forall \omega\in\Omega
\end{align}
where $x\in\mathbb{R},y\in\mathbb{R}^n$. We refer the guarantee of existence of RCD and the validity of the above definition for conditional expectation to Theorem \ref{lem:Klenke_Thm_8.37} and \ref{lem:Klenke_Thm_8.38}.

\section{Equivalent definition of CSE distribution}\label{sec:equivalence_tail_probability}

In this section, we prove the equivalence of the definition of the CSE distribution (Definition \ref{def:cse_distribution}) and a bound on CGF of the conditional distribution of $W_{in}|\mathcal{F}_n$.

\begin{lemma}[Equivalence of the definition of CSE distribution]\label{lem:equivalence_CSE}
	The following two statements are equivalent:
	\begin{enumerate}
		\item there exists positive parameters $(\lambda_{in},\gamma)$ with $\lambda_{in}\in\mathcal{F}_n$ and constant $\lambda$ such that 
		\begin{align}\label{eq:cse_cgf}
			\P\left[\mathcal{B}_1\right]=1,\ \mathcal{B}_1\equiv \left\{\E[\exp(sW_{in})|\mathcal{F}_n]\leq \exp(\lambda_{in}s^2),\ \forall s\in \left(-\frac{1}{\gamma},\frac{1}{\gamma}\right)\right\}.
		\end{align}
		\item there exists positive parameters $(\theta_{in},\beta)$ with $\theta_{in}\in\mathcal{F}_n$ and constant $\beta$ such that 
		\begin{align}\label{eq:cse_tail}
			\P\left[\mathcal{B}_2\right]=1,\ \mathcal{B}_2\equiv  \left\{\P\left[|W_{in}|\geq t|\mathcal{F}_n\right]\leq \theta_{in}\exp(-\beta t),\ \forall t>0\right\}.
		\end{align}
	\end{enumerate}
	In particular, the suppose condition \eqref{eq:cse_tail} holds, then we can choose $(\lambda_{in},\gamma)$ in \eqref{eq:cse_cgf} as
	\begin{align*}
		\lambda_{in}= \frac{\sqrt{6!4^6}(1+\theta_{in})}{24\beta^2}+\frac{16(1+\theta_{in})}{\beta^2},\ \gamma=\frac{4}{\beta}.
	\end{align*}
\end{lemma}

\subsection{Proof of Lemma \ref{lem:equivalence_CSE}}

\begin{proof}[Proof of Lemma \ref{lem:equivalence_CSE}]
	We prove two directions separately. 
	\paragraph{$1\Rightarrow 2$:}
	For any $\omega\in \mathcal{B}_1$, we know by definition \eqref{eq:def_conditional_expectation} that
	\begin{align*}
		\E[\exp(sW_{in})|\mathcal{F}_n](\omega)=\int \exp(sx)\mathrm{d}\kappa_{in}(\omega,x)<\infty,\ \forall s\in \left(-\frac{1}{\gamma},\frac{1}{\gamma}\right).
	\end{align*}
	Then the Chernoff bound gives 
	\begin{align*}
		\int \indicator(x\geq t)\mathrm{d}\kappa_{in}(\omega,x)\leq \int \exp\left(\frac{x}{2\gamma}\right)\mathrm{d}\kappa_{in}(\omega,x)\exp\left(-\frac{t}{2\gamma}\right).
	\end{align*}
	Applying a similar argument to $\indicator(-x\geq t)$, we conclude 
	\begin{align*}
		\int \indicator(|x|\geq t)\mathrm{d}\kappa_{in}(\omega,x)\leq \left(\int \exp\left(\frac{x}{2\gamma}\right)+\exp\left(\frac{-x}{2\gamma}\right)\mathrm{d}\kappa_{in}(\omega,x)\right)\cdot \exp\left(-\frac{t}{2\gamma}\right).
	\end{align*}
	Thus we know 
	\begin{align*}
		\P\left[|W_{in}|\geq t|\mathcal{F}_n\right]\leq \left(\E\left[\exp(W_{in}/(2\gamma))|\mathcal{F}_n\right]+\E\left[\exp(-W_{in}/(2\gamma))|\mathcal{F}_n\right]\right)\cdot \exp(-t/(2\gamma)).
	\end{align*}
	Thus $\P[\mathcal{B}_2]=1$ with
	\begin{align*}
		\theta_{in}=\E\left[\exp(W_{in}/(2\gamma))|\mathcal{F}_n\right]+\E\left[\exp(-W_{in}/(2\gamma))|\mathcal{F}_n\right],\ \beta=\frac{1}{2\gamma}.
	\end{align*}

	\paragraph{$2\Rightarrow 1$:}

	Fix a constant $a>0$ and $T>0$. For any $\omega\in \mathcal{B}_2$, we have 
	\begin{align*}
		&
		\E\left[\exp(a|W_{in}|)\indicator(\exp(a|W_{in}|)\leq \exp(aT))|\mathcal{F}_n\right](\omega)\\
		&
		=\int \exp(a|x|)\indicator(\exp(a|x|)\leq \exp(aT))\mathrm{d}\kappa_{in}(\omega,x)\\
		&
		\leq \int \min \{\exp(a|x|),\exp(aT)\}\mathrm{d}\kappa_{in}(\omega,x)\\
		&
		=\int \int_0^{e^{aT}}\indicator(\exp(a|x|)\geq t)\mathrm{d}t\mathrm{d}\kappa_{in}(\omega,x)\\
		&
		=\int_0^{e^{aT}}\int \indicator(\exp(a|x|)\geq t)\mathrm{d}\kappa_{in}(\omega,x)\mathrm{d}t\\
		&
		\leq 1+ \int_1^{e^{aT}}\int \indicator(|x|\geq \log(t)/a)\mathrm{d}\kappa_{in}(\omega,x)\mathrm{d}t
	\end{align*}
	where the second last equality is due to Fubini's theorem. Then by the definition of $\mathcal{B}_2$, we obtain 
	\begin{align*}
		\E\left[\exp(a|W_{in}|)\indicator(\exp(a|W_{in}|)\leq \exp(aT))|\mathcal{F}_n\right](\omega)
		&
		\leq 1+\theta_{in}(\omega)\int_1^{e^{aT}}e^{-(\beta\log(t))/a}\mathrm{d}t\\
		&
		=1+\theta_{in}(\omega)\int_{1}^{e^{aT}}t^{-\beta/a}\mathrm{d}t.
	\end{align*}
	For $a\in [0,\beta/2]$, we have 
	\begin{align*}
		\E\left[\exp(a|W_{in}|)\indicator(\exp(a|W_{in}|)\leq \exp(aT))|\mathcal{F}_n\right](\omega)
		&
		=1+\theta_{in}(\omega)\frac{1}{1-\beta/a}t^{1-\beta/a}\bigg|_{1}^{e^{aT}}\\
		&
		\leq 1+\theta_{in}(\omega)\frac{1}{\beta/a - 1}(1-e^{(aT-\beta T)})\\
		&
		\leq 1+\theta_{in}(\omega).
	\end{align*}
	Then by Fatou's lemma, we have for any $a\in [0,\beta/2]$,
	\begin{align}
		\E\left[\exp(a|W_{in}|)|\mathcal{F}_n\right](\omega)
		&\nonumber
		=\int \exp(a|x|)\mathrm{d}\kappa_{in}(\omega,x)\\
		&\label{eq:upper_bound_cgf}
		\leq \liminf_{T\rightarrow\infty}\int \exp(a|x|)\indicator(e^{a|x|}\leq e^{aT})\mathrm{d}\kappa_{in}(\omega,x)\leq 1+\theta_{in}(\omega).
	\end{align}
	Then by Taylor's expansion, for any $|s|\leq \beta/4$
	\begin{align*}
		\E[\exp(sW_{in})|\mathcal{F}_n](\omega)
		&
		=\int \exp(sx)\mathrm{d}\kappa_{in}(\omega,x)\\
		&
		=\int 1+sx+\frac{s^2x^2}{2}+\frac{s^3x^3}{6}\exp(y(x))\mathrm{d}\kappa_{in}(\omega,x),\ |y(x)|\leq  |sx|.
	\end{align*}
	Then the assumption $\E[W_{in}|\mathcal{F}_n]=0$ implies for $|s|\leq \beta/4$,
	\begin{align}
		\E[\exp(sW_{in})|\mathcal{F}_n](\omega)
		&\nonumber
		\leq 1+\frac{s^2\E[W_{in}^2|\mathcal{F}_n](\omega)}{2}+\frac{s^2\beta}{24}\int |x|^3\exp(y(x))\mathrm{d}\kappa_{in}(\omega,x)\\
		&\nonumber
		\leq 1+\frac{s^2\E[W_{in}^2|\mathcal{F}_n](\omega)}{2}+\frac{s^2\beta}{24}\int |x|^3\exp(|sx|)\mathrm{d}\kappa_{in}(\omega,x)\\
		&\nonumber
		\leq 1+\frac{s^2\E[W_{in}^2|\mathcal{F}_n](\omega)}{2}+\frac{s^2\beta}{24}\sqrt{\E[W_{in}^6|\mathcal{F}_n](\omega)\E[\exp(2|sW_{in}|)|\mathcal{F}_n](\omega)}\\
		&\nonumber
		\leq 1+\left(\frac{\beta \sqrt{1+\theta_{in}(\omega)}\sqrt{\E[W_{in}^6|\mathcal{F}_n](\omega)}}{12}+\E[W_{in}^2|\mathcal{F}_n](\omega)\right)\frac{s^2}{2}\\
		&\label{eq:cse_bound}
		\leq \exp\left(\left(\frac{\beta \sqrt{1+\theta_{in}(\omega)}\sqrt{\E[W_{in}^6|\mathcal{F}_n](\omega)}}{12}+\E[W_{in}^2|\mathcal{F}_n](\omega)\right)\frac{s^2}{2}\right).
	\end{align}
	where the second inequality is due to $|y(x)|\leq|sx|$, the third inequality is due to Cauchy-Schwarz inequality, the fourth inequality is due to conclusion \eqref{eq:upper_bound_cgf} and the last inequality is due to the inequality $\exp(x)\geq 1+x$ for any $x\in\mathbb{R}$. Now by Fubini's theorem, we have 
	\begin{align*}
		\E[\exp(s|W_{in}|)|\mathcal{F}_n]=1+\sum_{k=1}^{\infty}\frac{s^k}{k!}\E[|W_{in}|^k|\mathcal{F}_n],\ \forall s\in [0,\beta/2].
	\end{align*}
	Then by setting $s=\beta/4$ and conclusion \eqref{eq:upper_bound_cgf}, we have
	\begin{align*}
		\E[W_{in}^2|\mathcal{F}_n]
		&
		\leq \frac{2!4^2}{\beta^2}\E[\exp(\beta|W_{in}|/4)|\mathcal{F}_n]\leq \frac{2!4^2}{\beta^2}(1+\theta_{in}),\\
		\E[W_{in}^6|\mathcal{F}_n]
		&
		\leq \frac{6!4^6}{\beta^6}\E[\exp(\beta|W_{in}|/4)|\mathcal{F}_n]\leq \frac{6!4^6}{\beta^6}(1+\theta_{in})
	\end{align*}
	Then choosing 
	\begin{align*}
		\lambda_{in}= \frac{\sqrt{6!4^6}(1+\theta_{in})}{24\beta^2}+\frac{16(1+\theta_{in})}{\beta^2}\geq \frac{\beta\sqrt{1+\theta_{in}}\sqrt{\E[W_{in}^6|\mathcal{F}_n]}}{24} +\frac{\E[W_{in}^2|\mathcal{F}_n]}{2},\ \gamma=\frac{4}{\beta},
	\end{align*}
	so that by bound \eqref{eq:cse_bound}, we obtain
	\begin{align*}
		\E[\exp(sW_{in})|\mathcal{F}_n](\omega)\leq \exp(\lambda_{in}(\omega)s^2),\ \forall s\in \left(-\frac{1}{\gamma},\frac{1}{\gamma}\right).
	\end{align*}
\end{proof}

\section{Auxiliary lemmas}\label{sec:auxiliary-lemmas}

\begin{lemma}[\cite{Niu2022a}, Corollary 6] \label{lem:wlln} 
	Let $W_{in}$ be a triangular array of random variables, such that $W_{in}$ are independent for each $n$. If for some $\kappa > 0$ we have
	\begin{equation}
		\frac{1}{n^{1+\delta}}\sum_{i=1}^n\E[|W_{in}|^{1+\kappa}] \rightarrow 0,
		\label{eq:one-plus-kappa-moment-condition}
	\end{equation}
	then 
	\begin{equation*}
		\frac{1}{n} \sum_{i = 1}^n (W_{in} - \E[W_{in}]) \convp 0.
	\end{equation*}
\end{lemma}

\begin{lemma}[Gaussian tail probability estimate]\label{lem:Gaussian_tail_estimate}
	For $x>0$, we have 
	\begin{align*}
		1-\Phi(x)-\frac{1}{\sqrt{2\pi}x}\exp(-x^2/2)
	  \left(1-\frac{1}{x^2}\right)=-\int_{x}^{\infty}\frac{\phi(t)}{3t^4}\mathrm{d}t.
	\end{align*}
	Consequently, we have 
	\begin{align*}
	  \left|1-\Phi(x)-\left(\frac{1}{\sqrt{2\pi}x}\exp(-x^2/2)
	  \left(1-\frac{1}{x^2}\right)\right)\right|\leq \frac{\phi(x)}{x^3}
	\end{align*}
	and 
	\begin{align*}
		\left|x\exp\left(\frac{x^2}{2}\right)(1-\Phi(x))-\frac{1}{\sqrt{2\pi}}\right|\leq \frac{2}{\sqrt{2\pi}}\frac{1}{x^2}.
	\end{align*}
\end{lemma}

\begin{lemma}[Lower bound on the Gaussian tail probability]\label{lem:lower_bound_Gaussian}
	For any $x\geq 0$, we have 
	\begin{align*}
	  1-\Phi(x)> \frac{1}{\sqrt{2\pi}}\frac{x}{x^2+1}\exp(-x^2/2).
	\end{align*}
\end{lemma}

Before stating the next lemma, we will keep the way to define a version of conditional expecation using regular conditional distribution as in \eqref{eq:def_conditional_expectation}. 

\begin{lemma}\label{lem:existence_derivative_CGF}
	Consider the probability space $(\P,\Omega,\mathcal{F})$ and the $\sigma$-algebras $\mathcal{F}_n\subset \mathcal{F}$. Suppose the sequence of random variable $W_n$ satisfies there exists $\varepsilon>0$ such that 
	\begin{align*}
		\P\left[\E[|W_n|^p\exp(sW_n)|\mathcal{F}_n]<\infty,\ \forall s\in(-\varepsilon,\varepsilon),\ \forall n,p\in\mathbb{N}\right]=1
	\end{align*}
	Then defining $H_n(s)\equiv \E[\exp(sW_n)|\mathcal{F}_n]$, we have
	\begin{align*}
		\P\left[H_n(s) \text{ has }p\text{-th order derivative at the open neighborhood }(-\varepsilon,\varepsilon),\ \forall n,p\in\mathbb{N}\right]=1,
	\end{align*}
	and
	\begin{align*}
		\P\left[H_n^{(p)}(s)=\E[W_n^p\exp(sW_n)|\mathcal{F}_n],\forall s\in (-\varepsilon,\varepsilon),\ \forall n,p \in\mathbb{N}\right]=1.
	\end{align*}
\end{lemma}

\begin{lemma}[\cite{Chen2011}, Theorem 3.6]\label{lem:berry-esseen}
    Suppose $n\in\mathbb{N}$ and $\xi_{1n},\ldots,\xi_{nn}$ are independent random variables, satisfying for any $1\leq i\leq n$
    \begin{align*}
        \E[\xi_{in}]=0,\ \sum_{i=1}^n\E[\xi_{in}^2]=1.
    \end{align*}
	Then 
    \begin{align*}
        \sup_{t\in\mathbb{R}}\left|\P\left[\sum_{i=1}^n \xi_{in}\leq t\right]-\Phi(t)\right|\leq 9.4\sum_{i=1}^n \E[|\xi_{in}|^3].
    \end{align*}
\end{lemma}

\section{Proofs of lemmas in Appendix \ref{sec:auxiliary-lemmas}}

\subsection{Proof of the Lemma \ref{lem:Gaussian_tail_estimate}}

  \begin{proof}[Proof of Lemma \ref{lem:Gaussian_tail_estimate}]
    Applying integration by parts, we can write 
    \begin{align*}
      1-\Phi(x)
      &
      =\int_{x}^{\infty}\phi(t)\mathrm{d}t\\
      &
      =\int_{x}^{\infty}\frac{1}{t}\frac{t}{\sqrt{2\pi}}\exp(-t^2/2)\mathrm{d}t\\
      &
      =-\frac{1}{t}\frac{1}{\sqrt{2\pi}}\exp(-t^2/2)\bigg|_{x}^{\infty}-\int_{x}^{\infty}\frac{\phi(t)}{t^2}\mathrm{d}t\\
      &
      =\frac{\phi(x)}{x}+\frac{1}{t^3}\frac{1}{\sqrt{2\pi}}\exp(-t^2/2)\bigg|_x^{\infty}-\int_x^{\infty}\frac{\phi(t)}{3t^4}\mathrm{d}t\\
      &
      =\frac{\phi(x)}{x}-\frac{\phi(x)}{x^3}-\int_x^{\infty}\frac{\phi(t)}{3t^4}\mathrm{d}t.
    \end{align*}
    Then we can bound for $x>0$
    \begin{align*}
      \left|\int_{x}^{\infty}\frac{\phi(t)}{3t^4}\mathrm{d}t\right|\leq \phi(x)\int_{x}^{\infty}\frac{1}{3t^4}\mathrm{d}t\leq \frac{\phi(x)}{x^3}.
    \end{align*}
  \end{proof}

\subsection{Proof of Lemma \ref{lem:lower_bound_Gaussian}}

\begin{proof}[Proof of Lemma \ref{lem:lower_bound_Gaussian}]
  Define 
  \begin{align*}
	g(x)\equiv 1-\Phi(x)-\frac{1}{\sqrt{2\pi}}\frac{x}{x^2+1}\exp(-x^2/2).
  \end{align*}
  Computing the derivative we obtain 
  \begin{align*}
	g'(x)=-\frac{2}{\sqrt{2\pi}}\frac{e^{-x^2/2}}{(x^2+1)^2}<0.
  \end{align*}
  Also notice $g(0)=1/2>0$ and $\lim_{x\rightarrow\infty}g(x)=0$. This completes the proof.
\end{proof}

\subsection{Proof of Lemma \ref{lem:existence_derivative_CGF}}

\begin{proof}[Proof of Lemma \ref{lem:existence_derivative_CGF}]
	Consider the regular conditional distribution $W_n|\mathcal{F}_n$ to be $\kappa_{W_n}$. We use induction to prove the existence of the $p$-th derivative of $H_n(s)$. Suppose 
	\begin{align}\label{eq:induction_derivative}
		\P\left[H_n^{(p)}(s)=\E[W_n^p\exp(sW_n)|\mathcal{F}_n],\ \forall s\in (-\varepsilon,\varepsilon)\right]=1
	\end{align} 
	and 
	\begin{align}\label{eq:induction_assumption}
		\P\left[\E[|W_n|^{p+1}\exp(sW_n)|\mathcal{F}_n]<\infty,\ \forall s\in(-\varepsilon,\varepsilon)\right]=1.
	\end{align}
	According to the definition of derivative, we write 
	\begin{align}\label{eq:derivative_definition}
		H_n^{(p+1)}(s)\equiv \lim_{h\rightarrow0}\frac{H_n^{(p)}(s+h)-H_n^{(p)}(s)}{h}.
	\end{align} 
	Then on the event in hypothesis \eqref{eq:induction_derivative}, we have for any $s\in (-\varepsilon,\varepsilon)$
	\begin{align*}
		H_n^{(p)}(s)=\E[W_n^p\exp(sW_n)|\mathcal{F}_n]=\int x^p\exp(sx)\mathrm{d}\kappa_{W_n}(\cdot,x).
	\end{align*}
	Fix any $s_0\in (-\varepsilon,\varepsilon)$ and find $r_0\in (s_0,\varepsilon)$ such that $|r_0|>|s_0|$. We find small enough $h$ such that $|h|<\min\{|r_0|-s_0,|r_0|+s_0\}$. Thus we have
	\begin{align}\label{eq:interval_s_plus_h}
		-\varepsilon <-|r_0|=s_0-|r_0|-s_0<s_0+h< s_0+|r_0|-s_0 = |r_0|<\varepsilon.
	\end{align} 
	Also we notice, $|s_0|<|r_0|$ so that $s_0x\in (-|r_0x|,|r_0x|)$. Then the derivation in \eqref{eq:interval_s_plus_h} informs $(s_0+h)x$ belong to the interval $(-|r_0x|,|r_0x|)$. Therefore, both $s_0x$ and $(s_0+h)x$ belong to $(-|r_0x|,|r_0x|)$. Then we have
	\begin{align}
		|\exp((s_0+h)x)-\exp(s_0x)|
		&\nonumber
		=\left|\int_{s_0x}^{(s_0+h)x}e^{y}\mathrm{d}y\right|\\
		&\nonumber
		\leq |hx|\sup_{y\in [-|r_0x|,|r_0x|]}\exp(y)\\
		&\label{eq:finite_diff_bound}
		\leq |hx|\{\exp(r_0 x)+\exp(-r_0 x)\}.
	\end{align}
	By the definition \eqref{eq:derivative_definition}, we have
	\begin{align*}
		H^{(p+1)}_n(s_0)
		&
		=\lim_{h\rightarrow 0}\E\left[W_n^p \frac{\exp((s_0+h)W_n)-\exp(s_0W_n)}{h}|\mathcal{F}_n\right]\\
		&
		=\lim_{h\rightarrow 0}\int x^p\frac{\exp((s_0+h)x)-\exp(s_0x)}{h}\mathrm{d}\kappa_{W_n}(\cdot,x).
	\end{align*}
	Then we can bound by \eqref{eq:finite_diff_bound}
	\begin{align*}
		\left|x^p \frac{\exp((s_0+h)x)-\exp(s_0x)}{h}\right|\leq |x|^{p+1}\exp(r_0 x)+|x|^{p+1}\exp(-r_0 x).
	\end{align*}
	Notice the RHS is independent of $h$ and integrable with respect to measure $\kappa_{W_n}(\omega,\cdot)$ for almost every $\omega\in\Omega$, by the induction hypothesis \eqref{eq:induction_assumption} since $r_0\in(-\varepsilon,\varepsilon)$. By dominated convergence theorem, we know on the event in hypothesis \eqref{eq:induction_assumption},
	\begin{align*}
		H^{(p+1)}_n(s_0)
		&
		=\int \lim_{h\rightarrow0}x^p\frac{\exp((s_0+h)x)-\exp(s_0x)}{h}\mathrm{d}\kappa_{W_n}(\cdot,x)\\
		&
		=\int x^{p+1}\exp(s_0x)\mathrm{d}\kappa_{W_n}(\cdot,x)\\
		&
		=\E[ W_n^{p+1}\exp(s_0W_n)|\mathcal{F}_n].
	\end{align*}
	Then, by the arbitrary choice of $s_0\in (-\varepsilon,\varepsilon)$, we know on the event in hypothesis \eqref{eq:induction_assumption}, $H^{(p+1)}(s)$ is well-defined on the interval $(-\varepsilon,\varepsilon)$ and takes the form 
	\begin{align*}
		H^{(p+1)}(s)=\E[ W_n^{p+1}\exp(s_0W_n)|\mathcal{F}_n].
	\end{align*}
	Thus we have proved for the case $p+1$ so that we complete the induction and conclude the proof.
\end{proof}

\section{Proof of Proposition \ref{prop:equivalence_spa_formula}}

From \eqref{eq:conclusion_saddlepoint_approximation}, it suffices to prove 
\begin{align*}
	\frac{1-\Phi(r_n)+\phi(r_n)\left\{\frac{1}{\lambda_n}-\frac{1}{r_n}\right\}}{\exp\left(\frac{\lambda_n^2-r_n^2}{2}\right)(1-\Phi(\lambda_n))}=1+o_{\P}(1).
\end{align*}
On the event $\hat s_n=0$, we know the claim is correct. Therefore, we only need to consider the event $\hat s_n\neq 0$. Equivalently, it suffices to show 
\begin{align}
	\indicator(\hat s_n\neq 0)\left(\frac{\exp\left(\frac{r_n^2}{2}\right)(1-\Phi(r_n))}{\exp\left(\frac{\lambda_n^2}{2}\right)(1-\Phi(\lambda_n))}-1\right)\equiv \indicator(\hat s_n\neq 0)\left(\frac{h(r_n)}{h(\lambda_n)}-1\right)=o_{\P}(1)\label{eq:ratio_convergnece}
\end{align}
and 
\begin{align}
	\indicator(\hat s_n\neq 0)\frac{\frac{1}{\lambda_n}-\frac{1}{r_n}}{\exp(\lambda_n^2/2)(1-\Phi(\lambda_n))}=\indicator(\hat s_n\neq 0)\frac{1-\frac{\lambda_n}{r_n}}{\lambda_nh(\lambda_n)}=o_{\P}(1)\label{eq:ratio_vanish}.
\end{align}
We prove the statements \eqref{eq:ratio_convergnece}-\eqref{eq:ratio_vanish} subsequently.

\subsection{Proof of statement \eqref{eq:ratio_convergnece}}

Since $h(x)=\exp(x^2/2)(1-\Phi(x))$ is smooth, then by Taylor's expansion, we have 
\begin{align*}
	\frac{h(r_n)}{h(\lambda_n)}=\frac{h(\lambda_n)+h'(\tilde r_n)(r_n-\lambda_n)}{h(\lambda_n)}=1+\frac{h'(\tilde r_n)(r_n-\lambda_n)}{h(\lambda_n)}
\end{align*}
where $\tilde r_n$ is the point between  $r_n$ and $\lambda_n$. Now we intestigate $h'(x)$. We compute 
\begin{align*}
	h'(x)=x\exp(x^2/2)(1-\Phi(x))-\frac{1}{\sqrt{2\pi}}.
\end{align*}
By Lemma \ref{lem:Gaussian_tail_estimate}, we know 
\begin{align*}
	|h'(x)|\leq \frac{2}{\sqrt{2\pi}}\frac{1}{x^2}\leq \frac{1}{x^2}.
\end{align*}
Then since both event $r_n<0,\lambda_n>0$ and event $r_n>0,\lambda_n<0$ happen with probability zero, we have $1/\tilde r_n^2\in [\min\{1/r_n^2,1/\lambda_n^2\},\max\{1/r_n^2,1/\lambda_n^2\}]$. Therefore, we have 
\begin{align*}
	\left|\frac{h'(\tilde{r}_n)(r_n-\lambda_n)}{h(\lambda_n)}\right|\leq \frac{1}{\tilde{r}_n^2}\left|\frac{r_n-\lambda_n}{h(\lambda_n)}\right|\leq \left(\frac{1}{r_n^2}+\frac{1}{\lambda_n^2}\right)\left|\frac{r_n-\lambda_n}{h(\lambda_n)}\right|=\left(1+\frac{\lambda_n^2}{r_n^2}\right)\left|\frac{1-\frac{r_n}{\lambda_n}}{\lambda_nh(\lambda_n)}\right|.
\end{align*}
Thus in order to prove \eqref{eq:ratio_convergnece}, it suffices to show, by the sign condition \eqref{eq:same_sign_condition_sn},
\begin{align*}
	\indicator(\lambda_n\neq 0)\left(1+\frac{\lambda_n^2}{r_n^2}\right)\frac{1-\frac{r_n}{\lambda_n}}{\lambda_nh(\lambda_n)}=o_{\P}(1).
\end{align*}
The following lemma shows that the above statement is correct:
\begin{lemma}\label{lem:ratio_convergence}
	Under conditions \eqref{eq:rate_2} and \eqref{eq:rate_4}, we have 
	\begin{align*}
		\indicator(\lambda_n\neq 0)\left(1+\frac{\lambda_n^2}{r_n^2}\right)\frac{1-\frac{r_n}{\lambda_n}}{\lambda_nh(\lambda_n)}=o_{\P}(1).
	\end{align*}
\end{lemma}

\subsection{Proof of statement \eqref{eq:ratio_vanish}}

By the sign condition \eqref{eq:same_sign_condition_sn}, it suffices to prove 
\begin{align*}
	\indicator(\lambda_n\neq 0)\frac{1-\frac{\lambda_n}{r_n}}{\lambda_nh(\lambda_n)}=o_{\P}(1).
\end{align*}
The following lemma shows that the above statement is correct:
\begin{lemma}\label{lem:ratio_vanish}
	Under conditions \eqref{eq:rate_1} and \eqref{eq:rate_2}, we have 
	\begin{align*}
		\indicator(\lambda_n\neq 0)\frac{1-\frac{\lambda_n}{r_n}}{\lambda_nh(\lambda_n)}=o_{\P}(1).
	\end{align*}
\end{lemma}

\section{Proof of Theorem \ref{thm:example}}\label{sec:proof_example}

The proof follows by applying Theorem \ref{thm:unified_unnormalized_moment_conditions} with 
\begin{align*}
	W_{in}=\widetilde{X}_{in},\ w_n=\frac{1}{n}\sum_{i=1}^n X_{in},\ \mathcal{F}_n=\mathcal{G}_n,\ \varepsilon=2.
\end{align*}
Thus it suffices to verify the assumptions required in Theorem \ref{thm:unified_unnormalized_moment_conditions} with these realizations. In particular, we will verify $W_{in}|\mathcal{F}_n$ in this case satisfies \textbf{CCS condition}. Notice 
\begin{align*}
	\P[W_{in}\in [-|X_{in}|,|X_{in}|]|\mathcal{F}_n]=1,\ \text{almost surely}.
\end{align*}
Then by Theorem \ref{thm:unified_unnormalized_moment_conditions}, it suffices to verify the following lemma:
\begin{lemma}\label{lem:moment_verification}
	Suppose the assumptions of Theorem \ref{thm:example} hold. Then 
	\begin{align*}
		\frac{1}{n}\E[W_{in}^2|\mathcal{F}_n]=\frac{1}{n}\sum_{i=1}X_{in}^2=\Omega_{\P}(1),\ \frac{1}{n}\sum_{i=1}^n X_{in}^4=O_{\P}(1).
	\end{align*}
\end{lemma}
\noindent The proof is postponed to Appendix \ref{sec:proof_moment_verification}.

\subsection{Proof of Lemma \ref{lem:moment_verification}}\label{sec:proof_moment_verification}

\begin{proof}[Proof of Lemma \ref{lem:moment_verification}]
	We will apply Lemma \ref{lem:wlln} to prove the claims. 

	\paragraph{Proof of $\sum_{i=1}^n X_{in}^4/n=O_{\P}(1)$:}
	
	We will apply Lemma \ref{lem:wlln} with $W_{in}=X_{in}^4$ and $\kappa=\delta/4$. If we can verify,
	\begin{align*}
		\frac{1}{n^{1+\delta/4}}\sum_{i=1}^n\E[|X_{in}^4|^{1+\delta/4}]=\frac{1}{n^{1+\delta/4}}\sum_{i=1}^n\E[|X_{in}|^{4+\delta}]\rightarrow 0.
	\end{align*}
	then applying Lemma \ref{lem:wlln}, we have 
	\begin{align*}
		\frac{1}{n}\sum_{i=1}^n (X_{in}^4-\E[X_{in}^4])=o_{\P}(1).	
	\end{align*}
	It suffices to show 
	\begin{align}\label{eq:upper_bound_fourth_moment}
		\limsup_{n\rightarrow\infty}\frac{1}{n}\sum_{i=1}^n \E[X_{in}^{4}]<\infty,\ \limsup_{n\rightarrow\infty}\frac{1}{n}\sum_{i=1}^n \E[|X_{in}|^{4+\delta}]<\infty.
	\end{align}
	By Lemma \eqref{lem:moment_dominance}, it suffices to just show the $4+\delta$ moment condition in \eqref{eq:upper_bound_four_delta_moment}. In fact, using the inequality $(|a|+|b|)^{p}\leq 2^p (|a|^p+|b|^p)$ for $p>0$, we can bound 
	\begin{align*}
		|X_{in}|^{4+\delta}=|\mu_n+\varepsilon_{in}|^{4+\delta}\leq 2^{4+\delta}|\mu_n|^{4+\delta}+2^{4+\delta}|\varepsilon_{in}|^{4+\delta}
	\end{align*}
	so that 
	\begin{align*}
		\frac{1}{n}\sum_{i=1}^n \E[|X_{in}|^{4+\delta}]\leq 2^{4+\delta}|\mu_n|^{4+\delta}+\frac{2^{4+\delta}}{n}\sum_{i=1}^n\E[|\varepsilon_{in}|^{4+\delta}].
	\end{align*}
	By condition \eqref{eq:mu_n_convergence} and condition \eqref{eq:upper_bound_four_delta_moment}, we know the claim is true. Therefore, we have 
	\begin{align*}
		\frac{1}{n}\sum_{i=1}^n X_{in}^4=\frac{1}{n}\sum_{i=1}^n (X_{in}^4-\E[X_{in}^4])+ \frac{1}{n}\sum_{i=1}^n \E[X_{in}^4]=O_{\P}(1).
	\end{align*}
	
	\paragraph{Proof of $\sum_{i=1}^n X_{in}^2/n=\Omega_{\P}(1)$:}
	
	We will apply Lemma \ref{lem:wlln} with $W_{in}=X_{in}^2$ and $\kappa=1$. In other words, we need to verify 
	\begin{align*}
		\frac{1}{n^2}\sum_{i=1}^n \E[X_{in}^4]\rightarrow0.
	\end{align*}
	This is true by cliam \eqref{eq:upper_bound_fourth_moment} that $\limsup_{n\rightarrow\infty}\sum_{i=1}^n \E[X_{in}^4]/n<\infty$. Therefore applying Lemma \ref{lem:wlln}, we have 
	\begin{align*}
		\frac{1}{n}\sum_{i=1}^n (X_{in}^2-\E[X_{in}^2])=o_{\P}(1).
	\end{align*}
	Thus in order to prove $\sum_{i=1}^n X_{in}^2/n=\Omega_{\P}(1)$, it suffices to show $\liminf_{n\rightarrow\infty}\sum_{i=1}^n \E[X_{in}^2]/n>0$. Indeed,
	\begin{align*}
		\frac{1}{n}\sum_{i=1}^n \E[X_{in}^2]=\mu_n^2+\frac{2\mu_n}{n}\sum_{i=1}^n \E[\varepsilon_{in}]+\frac{1}{n}\sum_{i=1}^n \E[\varepsilon_{in}^2]\geq \frac{1}{n}\sum_{i=1}^n \E[\varepsilon_{in}^2].
	\end{align*}
	where the second inequality is true because $\E[\varepsilon_{in}]=0$ due to the symmetric distribution assumption. By condition \eqref{eq:lower_bound_second_moment}, we know 
	\begin{align*}
		\liminf_{n\rightarrow\infty}\frac{1}{n}\sum_{i=1}^n \E[X_{in}^2]\geq \liminf_{n\rightarrow\infty}\frac{1}{n}\sum_{i=1}^n \E[\varepsilon_{in}^2]>0.
	\end{align*}
	Therefore 
	\begin{align*}
		\frac{1}{n}\sum_{i=1}^n X_{in}^2=\frac{1}{n}\sum_{i=1}^n (X_{in}^2-\E[X_{in}^2])+\frac{1}{n}\sum_{i=1}^n\E[X_{in}^2]=\Omega_{\P}(1).
	\end{align*}
\end{proof}

\section{Proofs of supporting lemmas for Theorem \ref{thm:unified_unnormalized_moment_conditions}} \label{sec:lemma_proofs}

In this section, we first state two lemmas that reduce the condition of Theorem \ref{thm:unified_unnormalized_moment_conditions} to several conditions on the CGF. Then we prove the supporting lemmas for Theorem \ref{thm:unified_unnormalized_moment_conditions} based on the reduced conditions.

\begin{lemma}\label{lem:reduced_condition}
	Suppose Assumption \ref{assu:cse} or Assumption \ref{assu:ccs} holds. Then, the following statements hold:
	\begin{align}
		\sup_{s \in (-\varepsilon, \varepsilon)} \frac{1}{n}\sum_{i = 1}^n (K''_{in}(s))^2 = O_{\P}(1); \label{eq:second_cgf_derivative_bound} \\  
		\frac{1}{n}\sum_{i = 1}^n K'''_{in}(0) = O_{\P}(1); \label{eq:third_cgf_derivative_bound} \\  
		\sup_{s \in (-\varepsilon, \varepsilon)}\left|\frac{1}{n}\sum_{i = 1}^n K''''_{in}(s)\right| = O_{\P}(1). \label{eq:fourth_cgf_derivative_bound}
	\end{align}
\end{lemma}

\begin{lemma}\label{lem:reduced_variance_condition}
	Suppose Assumption \ref{assu:cse} or Assumption \ref{assu:ccs} holds. Then condition \eqref{eq:lower_bound_conditional_variance} implies 
	\begin{align}\label{eq:lower_bound_variance}
		\frac{1}{n}\sum_{i=1}^n K_{in}''(0)=\Omega_{\P}(1).
	\end{align}
\end{lemma}

\subsection{Proof of Lemma \ref{lem:finite_cgf}}

We prove claim \eqref{eq:finite_cgf} and \eqref{eq:finite_cgf_derivatives} separately. 

\paragraph{Proof of claim \eqref{eq:finite_cgf}:}

We consider two cases: CSE distribution and CCS distribution.

\paragraph{Case 1: CSE distribution}

By Lemma \ref{lem:equivalence_CSE}, we know 
\begin{align*}
	\P\left[K_{in}(s)\leq \lambda_ns^2,\ \forall s\in \left(-\frac{1}{\gamma},\frac{1}{\gamma}\right)\right]=1
\end{align*}
where 
\begin{align*}
	\lambda_n\equiv \frac{\sqrt{6!4^6}(1+\theta_{n})}{24\beta^2}+\frac{16(1+\theta_{n})}{\beta^2},\ \gamma=\frac{4}{\beta}.
\end{align*}
Since $\theta_n<\infty$ almost surely, we know condition \eqref{eq:finite_cgf} holds with $\varepsilon=1/(2\gamma)=\beta/8$.

\paragraph{Case 2: CCS distribution}

By the definition of CCS distribution and the definition of regular conditional distribution, we have 
\begin{align*}
	\P\left[\mathrm{Supp}(\kappa_{in}(\omega,\cdot))\in [-\nu_{in}(\omega),\nu_{in}(\omega)]\right]=1.
\end{align*}
Then we have for almost every $\omega\in\Omega$,
\begin{align*}
	\E[\exp(sW_{in})|\mathcal{F}_n](\omega)=\int \exp(sx)\mathrm{d}\kappa_{in}(\omega,x)\leq \exp(\nu_{in})<\infty,\ \forall s\in (-1,1)
\end{align*}
where the last inequality is due to the assumption $\nu_{in}<\infty$ almost surely. Therefore, condition \eqref{eq:finite_cgf} holds with $\varepsilon=1$.

\paragraph{Proof of claim \eqref{eq:finite_cgf_derivatives}:}

By Lemma \ref{lem:existence_derivative_CGF}, it suffices to prove the following lemma.
\begin{lemma}\label{lem:finite_cgf_moments}
	On the event $\mathcal{A}$, 
	\begin{align*}
		\E[|W_{in}|^p\exp(sW_{in})|\mathcal{F}_n]<\infty,\ \forall s\in (-\varepsilon,\varepsilon),\ \forall i\in\{1,\ldots,n\}, n\geq 1\text{ and } p\in \mathbb{N}.
	\end{align*}
\end{lemma}
\noindent Proof of Lemma \ref{lem:finite_cgf_moments} is postponed to Appendix \ref{sec:proof_finite_cgf_moments}.

\subsection{Proof of Lemma~\ref{lem:positive_second_derivative}}

The claim~\eqref{eq:positive_second_derivative} holds because by equation~\eqref{eq:conditional_moments}, on the event $\mathcal A$ we have, for each $s \in (-\varepsilon, \varepsilon)$,
\begin{equation*}
K''_n(s) = \frac{1}{n}\sum_{i = 1}^n K''_{in}(s) = \frac{1}{n}\sum_{i = 1}^n \V_{n, s}[W_{in} \mid \mathcal F_n] \geq 0.
\end{equation*}

Next, we verify claim~\eqref{eq:uniform_lower_bound_second_derivative}. To this end, fix $\delta > 0$. By assumptions~\eqref{eq:lower_bound_variance}, \eqref{eq:third_cgf_derivative_bound}, and~\eqref{eq:fourth_cgf_derivative_bound}, there exist $\eta, M > 0$ and $N \geq 1$ be such that for all $n \geq N$,
\begin{equation*}
\P[K''_n(0) < 2 \eta ] < \delta/3, \quad \P[|K'''_{n}(0)| > M] < \delta/3, \quad \P\left[\sup_{s \in (-\varepsilon, \varepsilon)}|K''''_{n}(s)| > M\right] < \delta/3. 
\end{equation*}
Define
\begin{equation*}
s_* \equiv \min(\eta/(2M), \sqrt{\eta/M}, \varepsilon/2).
\end{equation*}
On the event $\mathcal A$, Lemma~\ref{lem:finite_cgf} guarantees that we can we Taylor expand $K_n''(s)$ around $s = 0$ to obtain
\begin{equation*}
K_n''(s) = K_n''(0) + sK_n'''(0) + \frac12 s^2 K_n''''(\bar s)
\end{equation*}
for some $|\bar s|\leq |s|$. Therefore, for all $n \geq N$, we have
\begin{equation*}
\begin{split}
1-\delta &< \P\left[\mathcal A, K''_n(0) \geq 2 \eta, |K'''_{n}(0)| \leq M, \sup_{s \in (-\varepsilon, \varepsilon)}|K''''_{n}(s)| \leq M  \right] \\
&\leq \P\left[\inf_{s \in [-s_*, s_*]}K''_n(s) \geq 2\eta - s_* M - \frac12 s_*^2 M \geq \eta\right],
\end{split}
\end{equation*}
which verifies the claim~\eqref{eq:uniform_lower_bound_second_derivative} and completes the proof.

\subsection{Proof of Lemma~\ref{lem:saddlepoint_properties}}

\paragraph{Proof of \eqref{eq:sign_property_s_n}:}

Suppose $|S_n|=1$. Because $K'_n$ is almost surely nondecreasing on $(-\varepsilon, \varepsilon)$~\eqref{eq:positive_second_derivative} and $K'_n(0) = 0$~\eqref{eq:k_n_prime_0_equals_0}, the identity $K'_n(\hat s_n) = w_n$ implies that $\mathrm{sgn}(\hat s_n)=\mathrm{sgn}(w_n)$. When $|S_n|\neq 1$, by the definition of $\hat s_n$~\eqref{eq:def_s_n}, we have $\mathrm{sgn}(\hat s_n)=\mathrm{sgn}(w_n)$. This completes the proof.

\paragraph{Proof of \eqref{eq:hat_s_n_convergence}:}

Fix $\gamma, \delta > 0$. By Lemma~\ref{lem:positive_second_derivative}, there $\eta > 0$, $s_* \in (0, \varepsilon/2)$, and $N \in \mathbb N_+$ such that
\begin{equation*}
\P\left[\inf_{s \in [-s_*, s_*]} K''_n(s) \geq \eta\right] \geq 1 - \delta/2 \quad \text{for all} \quad n \geq N.
\end{equation*}
By increasing $N$ if necessary, the fact that $w_n \convp 0$ implies that
\begin{equation*}
\P[|w_n| \leq \eta \min(\gamma, s_*)] \geq 1 - \delta/2  \quad \text{for all} \quad n \geq N.
\end{equation*}
Define the event 
\begin{equation*}
\mathcal E'_n \equiv \left\{\inf_{s \in [-s_*, s_*]} K''_n(s) \geq \eta, |w_n| \leq \eta \min(\gamma, s_*)\right\}
\end{equation*}
On the event $\mathcal E'_n \cap \mathcal A$, the Taylor expansion~\eqref{eq:K_prime_taylor_expansion} gives
\begin{equation*}
|K'_n(s)| \geq |s| \eta \quad \text{for all} \quad s \in [-s_*, s_*] \text{ and all } n \geq N.
\end{equation*}  
Hence, $|w_n| \leq \eta s_* \leq \min(-K'_n(-s_*), K'_n(s_*))$, implying $w_n \in [K'_n(-s_*), K'_n(s_*)]$, so the saddlepoint equation has a solution $\hat s_n$ such that $K'_n(\hat s_n) = w_n$ and $|\hat s_n| \leq s_*$. Therefore, on the event $\mathcal E'_n \cap \mathcal A$, we have
\begin{equation*}
|\hat s_n| \eta \leq |K'_n(\hat s_n)| = |w_n| \leq \eta \gamma \quad \Longrightarrow |\hat s_n| \leq \gamma.
\end{equation*}
It follows that
\begin{equation*}
\P[|\hat s_n| \leq \gamma] \geq \P[\mathcal E'_n \cap \mathcal A] \geq 1 - \delta \quad \text{for all} \quad n \geq N,
\end{equation*}
which shows that $\hat s_n \convp 0$, as desired.

\paragraph{Proof of \eqref{eq:hat_s_n_second_derivative}:} By the argument following the statement of Lemma~\ref{lem:positive_second_derivative}, for any $\delta$ there is an $\eta > 0$ and $N \in \mathbb N_+$ such that $\P[K''_n(\hat s_n) \geq \eta] \geq 1-\delta$ for all $n \geq N$. This shows that $K''_n(\hat s_n) = \Omega_{\P}(1)$, as desired.

\subsection{Proof of Lemma \ref{lem:additional_properties_r_n_lambda_n}}

\begin{proof}[Proof of Lemma \ref{lem:additional_properties_r_n_lambda_n}]
We prove the claims separately.

\paragraph{Verification of \eqref{eq:finitness_r_n_lambda_n}.}

Since $w_n\in (-\infty,\infty)$, together with Lemma \ref{lem:finite_cgf} guaranteeing that $K_n(s),K_n'(s),K_n''(s)\in (-\infty,\infty),\forall s\in (-\varepsilon,\varepsilon)$ almost surely and definition of $\hat s_n$ such that $|\hat s_n|<\varepsilon$, we have 
\begin{align*}
	\lambda_n^2=|n\hat s_n K_n''(\hat s_n)|<n\varepsilon|K_n''(\hat s_n)|<\infty,\ r_n^2\leq \max\left\{1,|2n(\hat s_n w_n-K_n(\hat s_n))|\right\}<\infty.
\end{align*}

\paragraph{Verification of \eqref{eq:sign_1}.} This claim follows from conclusion~\eqref{eq:sign_property_s_n} of Lemma~\ref{lem:saddlepoint_properties} and the definitions of $r_n$ and $\lambda_n$ in equation~\eqref{eq:lam_n_r_n_def}.

\paragraph{Verification of \eqref{eq:sign_condition_r_lambda}.}

This is true by definition of $r_n$ and $\lambda_n$. This completes the proof.

\paragraph{Verification of \eqref{eq:same_sign_condition}, \eqref{eq:same_sign_condition_sn}, \eqref{eq:rate_1}, \eqref{eq:rate_2}, \eqref{eq:rate_3}, \eqref{eq:rate_4} and \eqref{eq:rate_r}:}

We present a useful lemma.
\begin{lemma}[Asymptotic estimate of $\lambda_n$ and $r_n$]\label{lem:asym-estimate-lam-r}
	Under the assumptions of Theorem \ref{thm:unified_unnormalized_moment_conditions}, the followings are true
	\begin{align}
		\frac{r_n^2}{n}=o_{\P}(1);& \label{eq:r_n_over_n_rate}\\
		\frac{\lambda_n}{r_n}=1+\hat s_nO_{\P}(1); & \label{eq:asym-estimate-ratio-lam-r}\\
		\frac{r_n}{\lambda_n}=1+\hat s_nO_{\P}(1); & \label{eq:asym-estimate-ratio-r-lam}\\
		\frac{1}{\lambda_n}-\frac{1}{r_n}=o_{\P}(1); & \label{eq:asym-estimate-diff-lam-r}\\
		\indicator(r_n>0,\lambda_n>0)\frac{1}{r_n}\left(\frac{\lambda_n}{r_n}-1\right)=o_{\P}(1); & \label{eq:asym-estimate-diff-lam-r-multiplication}\\
		\indicator(\lambda_n\neq 0)\frac{1}{\lambda_n}\left(\frac{r_n}{\lambda_n}-1\right)=o_{\P}(1); & \label{eq:asym-estimate-diff-r-lam-multiplication}\\
		\P[w_n>0 \text{ and }\lambda_n r_n\leq 0]\rightarrow0. & \label{eq:same_sign_condition_w_n}\\
		\P[\hat s_n\neq 0 \text{ and }\lambda_n r_n\leq 0]\rightarrow0. & \label{eq:same_sign_condition_s_n}
	\end{align}
	\end{lemma}

	\paragraph{Verification of \eqref{eq:same_sign_condition}:}

	\eqref{eq:same_sign_condition_w_n} verifies \eqref{eq:same_sign_condition}.

	\paragraph{Verification of \eqref{eq:same_sign_condition_sn}:}

	\eqref{eq:same_sign_condition_s_n} verifies \eqref{eq:same_sign_condition_sn}.

	\paragraph{Verification of \eqref{eq:rate_1}:}
	
	\eqref{eq:asym-estimate-diff-lam-r} verifies \eqref{eq:rate_1}.

	\paragraph{Verification of \eqref{eq:rate_2}:} 
	
	Since $\hat s_n\convp 0$, we know \eqref{eq:asym-estimate-ratio-lam-r} implies 
	\begin{align*}
		\frac{\lambda_n}{r_n}=1+o_{\P}(1)
	\end{align*}
	which verifies \eqref{eq:rate_2}. 

	\paragraph{Verification of \eqref{eq:rate_3}:}

	\eqref{eq:asym-estimate-diff-lam-r-multiplication} verifies \eqref{eq:rate_3}.

	\paragraph{Verification of \eqref{eq:rate_4}:}

	\eqref{eq:asym-estimate-diff-r-lam-multiplication} verifies \eqref{eq:rate_4}.

	\paragraph{Verification of \eqref{eq:rate_r}:}

	\eqref{eq:r_n_over_n_rate} verifies \eqref{eq:rate_r}.

\end{proof}

\subsection{Proof of Lemma~\ref{lem:symmetry}}

We can apply the theorem to the triangular array $\widetilde W_{in} \equiv -W_{in}$ and set of cutoffs $\widetilde w_n \equiv -w_n$, since the theorem assumptions are invariant to the signs of $W_{in}$ and $x_{in}$. Therefore, we get the result
\begin{equation*}
\indicator(\widetilde w_n > 0)\left(\frac{\P\left[\frac1n \sum_{i = 1}^n \widetilde W_{in} \geq \widetilde w_n \mid \mathcal F_n\right]}{1-\Phi(\widetilde r_n)+\phi(\widetilde r_n)\left\{\frac{1}{\widetilde \lambda_n}-\frac{1}{\widetilde r_n}\right\}}-1\right) \convp 0,
\end{equation*}
where we claim that $\widetilde r_n = -r_n$ and $\widetilde \lambda_n = -\lambda_n$. To see this, we define
\begin{align*}
	\widetilde{K}_{in}(s)\equiv \log\E\left[\exp(s\widetilde W_{in})|\mathcal{F}_n\right],\ \widetilde{K}_n(s)\equiv \frac{1}{n}\sum_{i=1}^n\widetilde{K}_{in}(s) = \frac{1}{n}\sum_{i=1}^n K_{in}(-s)=K_n(-s).
\end{align*}
Then, consider the saddlepoint equation for $\widetilde w_n$:
\begin{align}\label{eq:saddlepoint-equation-negative-xn}
	\widetilde{K}'_n(s)=\widetilde{x}_n.
\end{align}
Furthermore, we define 
\begin{align*}
	\widetilde{S}_n\equiv \{s\in [-\varepsilon/2,\varepsilon/2]:\widetilde{K}_n'(s)=\widetilde{x}_n\}.
\end{align*}
Then we write the solution $\widetilde{s}_n$ to the saddlepoint equation~\eqref{eq:saddlepoint-equation-negative-xn} according to the definition of $\hat s_n$ as in \eqref{eq:def_s_n}
\begin{align*}
	\widetilde{s}_n=
	\begin{cases}
	\text{the single element of }\widetilde{S}_n & \text{if } |\widetilde{S}_n|=1; \\
	\frac{\varepsilon}{2}\mathrm{sgn}(\widetilde{x}_n) & \text{otherwise}.
	\end{cases}
\end{align*}
Then we argue that $\widetilde{s}_n=-\hat s_n$. This is because given $\hat s_n$ uniquely solves \eqref{eq:saddlepoint-equation}, we know $-\hat s_n$ uniquely solves \eqref{eq:saddlepoint-equation-negative-xn}. Similarly, whenever $\widetilde{s}_n$ uniquely solves \eqref{eq:saddlepoint-equation-negative-xn}, we know $-\widetilde{s}_n$ uniquely solves \eqref{eq:saddlepoint-equation}. Therefore, we have $\widetilde{s}_n=-\hat s_n$. Then recall the definition 
\begin{align*}
	\widetilde{\lambda}_n\equiv \sqrt{n}\widetilde{s}_n\widetilde{K}_n''(\widetilde{s}_n),\ \widetilde{r}_n\equiv 
	\begin{cases}
		\sgn(\widetilde s_n) \sqrt{2n( \widetilde s_n \widetilde w_n - \widetilde K_n(\widetilde s_n))} & \text{if } \widetilde s_n \widetilde w_n - \widetilde K_n(\widetilde s_n)\geq 0;\\
		\mathrm{sgn}(\widetilde s_n) & \text{otherwise},
	  \end{cases}.
\end{align*}
Since $\widetilde{K}_n''(-s)=K_n''(s),\widetilde{K}_n(-s)=K_n(s)$ and $\widetilde{x}_n=-w_n$, we know $\widetilde{\lambda}_n=-\lambda_n$ and $\widetilde{r}_n=-r_n$. Therefore, we have
\begin{align}
&\nonumber
\indicator(\widetilde w_n > 0)\left(\frac{\P\left[\frac1n \sum_{i = 1}^n \widetilde W_{in} \geq \widetilde w_n \mid \mathcal F_n\right]}{1-\Phi(\widetilde r_n)+\phi(\widetilde r_n)\left\{\frac{1}{\widetilde \lambda_n}-\frac{1}{\widetilde r_n}\right\}}-1\right) \\
&\nonumber
\quad = \indicator(w_n < 0)\left(\frac{\P\left[\frac1n \sum_{i = 1}^n W_{in} \leq w_n \mid \mathcal F_n\right]}{1-\Phi(-r_n)+\phi(r_n)\left\{\frac{1}{r_n}-\frac{1}{\lambda_n}\right\}}-1\right) \\
&\label{eq:convergence_flip_sign}
\quad = \indicator(w_n < 0)\left(\frac{1 - \Phi(r_n)+\phi(r_n)\left\{\frac{1}{\lambda_n}-\frac{1}{r_n}\right\} - \P\left[\frac1n \sum_{i = 1}^n W_{in} > w_n \mid \mathcal F_n\right]}{\Phi(r_n)+\phi(r_n)\left\{\frac{1}{r_n}-\frac{1}{\lambda_n}\right\}}\right) \convp 0. 
\end{align}
Note the demoninator in \eqref{eq:convergence_flip_sign} is not what we want and we would like to change it to $1-\Phi(r_n)+\phi(r_n)\{\frac{1}{\lambda_n}-\frac{1}{r_n}\}$. Now, we need the following lemma to proceed.
\begin{lemma}\label{lem:upper_bound_ratio_spa}
	Suppose the assumptions of Theorem \ref{thm:unified_unnormalized_moment_conditions} hold. Then \eqref{eq:finitness_r_n_lambda_n}, \eqref{eq:sign_1}, \eqref{eq:sign_condition_r_lambda}, \eqref{eq:rate_1} conditions are true by Lemma \ref{lem:additional_properties_r_n_lambda_n}. Furthermore, we have 
	\begin{enumerate}
		\item 	
		\begin{align}\label{eq:multiplication-flip-sign}
			\indicator(w_n<0)\left|\frac{\Phi(r_n)+\phi(r_n)\left\{\frac{1}{r_n}-\frac{1}{\lambda_n}\right\}}{1 - \Phi(r_n)+\phi(r_n)\left\{\frac{1}{\lambda_n}-\frac{1}{r_n}\right\}}\right|\leq 1+o_\P(1);
		\end{align}
		\item 
		\begin{align}\label{eq:equality-corner-case}
			\indicator(w_n < 0)\frac{\P\left[\frac1n \sum_{i = 1}^n W_{in} = w_n \mid \mathcal F_n\right]}{1-\Phi(r_n)+\phi(r_n)\left\{\frac{1}{\lambda_n}-\frac{1}{r_n}\right\}}=o_\P(1).
		\end{align}
	\end{enumerate}
\end{lemma}

Now, guaranteed by \eqref{eq:multiplication-flip-sign} in Lemma \ref{lem:upper_bound_ratio_spa}, we multiply both sides of the last statement as in \eqref{eq:convergence_flip_sign} by $\indicator(w_n < 0)\frac{\Phi(r_n)+\phi(r_n)\left\{\frac{1}{r_n}-\frac{1}{\lambda_n}\right\}}{1 - \Phi(r_n)+\phi(r_n)\left\{\frac{1}{\lambda_n}-\frac{1}{r_n}\right\}}$ and rearrange to obtain that 
\begin{equation*}
\indicator(w_n < 0)\left(\frac{\P\left[\frac1n \sum_{i = 1}^n W_{in} > w_n \mid \mathcal F_n\right]}{1 - \Phi(r_n)+\phi(r_n)\left\{\frac{1}{\lambda_n}-\frac{1}{r_n}\right\}}-1\right) \convp 0.
\end{equation*}

This is almost what we want~\eqref{eq:negative_w_n}, except the inequality in the numerator is strict. To address this, we note we have proved \eqref{eq:equality-corner-case} in Lemma \ref{lem:upper_bound_ratio_spa} that 
\begin{equation*}
\indicator(w_n < 0)\frac{\P\left[\frac1n \sum_{i = 1}^n W_{in} = w_n \mid \mathcal F_n\right]}{1-\Phi(r_n)+\phi(r_n)\left\{\frac{1}{\lambda_n}-\frac{1}{r_n}\right\}} \convp 0.
\end{equation*}
Putting together the preceding two displays, we conclude that
\begin{equation*}
\begin{split}
&\indicator(w_n < 0)\left(\frac{\P\left[\frac1n \sum_{i = 1}^n W_{in} \geq w_n \mid \mathcal F_n\right]}{1-\Phi(r_n)+\phi(r_n)\left\{\frac{1}{\lambda_n}-\frac{1}{r_n}\right\}}-1\right) \\
&\quad = \indicator(w_n < 0)\left(\frac{\P\left[\frac1n \sum_{i = 1}^n W_{in} > w_n \mid \mathcal F_n\right]}{1-\Phi(r_n)+\phi(r_n)\left\{\frac{1}{\lambda_n}-\frac{1}{r_n}\right\}}-1 + \frac{\P\left[\frac1n \sum_{i = 1}^n W_{in} = w_n \mid \mathcal F_n\right]}{1-\Phi(r_n)+\phi(r_n)\left\{\frac{1}{\lambda_n}-\frac{1}{r_n}\right\}}\right) \\
&\quad \convp 0.
\end{split}
\end{equation*}

\subsection{Proof of Lemma \ref{lem:conditional_CLT_W_n}}

\begin{proof}[Proof of Lemma \ref{lem:conditional_CLT_W_n}]
	We first prove the first claim.  
	\paragraph{Proof of \eqref{eq:conditional_uniform_CLT}:}
	We apply Lemma \ref{lem:conditional-clt} to prove the result. It suffices to show 
	\begin{enumerate}
		\item  
		\begin{align*}
			\mathrm{Var}[W_{in}|\mathcal{F}_n]<\infty;
		\end{align*}
		\item 
		\begin{align*}
			\frac{1}{n}\sum_{i=1}^n \E[|W_{in}-\E[W_{in}|\mathcal{F}_n]|^{3}|\mathcal{F}_n]=O_\P(1),\ K_n''(0)=\Omega_{\P}(1).
		\end{align*}
	\end{enumerate}
	For the first claim, we know 
	\begin{align*}
		\mathrm{Var}[W_{in}|\mathcal{F}_n]=K_{in}''(0)<\infty
	\end{align*}
	almost surely by Lemma \ref{lem:finite_cgf}. For the second claim, we claim it suffices to prove 
	\begin{align*}
		\frac{1}{n}\sum_{i=1}^n \E[|W_{in}-\E[W_{in}|\mathcal{F}_n]|^{4}|\mathcal{F}_n]=O_\P(1),\ K_n''(0)=\Omega_\P(1).
	\end{align*}
	This is because, intuitively, we can upper bound the lower moment by the higher moment. We provide a formal result for such intuition.
	\begin{lemma}[Dominance of higher moment]\label{lem:moment_dominance}
		For any $1<p<q<\infty$, the following inequality is true:
		\begin{align*}
		  \frac{\sum_{i=1}^n \E[|W_{in}|^{p}|\mathcal{F}_n]}{n}\leq \left(\frac{\sum_{i=1}^n \E[|W_{in}|^{q}|\mathcal{F}_n]}{n}\right)^{p/q}.
		\end{align*}
	\end{lemma}
	\noindent Applying Lemma \ref{lem:moment_dominance} with $p=3,q=4$ we know the claim is true. By the expression of the fourth central moment in terms of the second and fourth cumulant, we have 
	\begin{align*}
		\frac{1}{n}\sum_{i=1}^n \E[|W_{in}-\E[W_{in}|\mathcal{F}_n]|^{4}|\mathcal{F}_n]=\frac{1}{n}\sum_{i=1}^n \left\{K_{in}''''(0)+3(K_{in}''(0))^2\right\}=O_{\P}(1)
	\end{align*}
	guaranteed by assumptions \eqref{eq:second_cgf_derivative_bound} and \eqref{eq:fourth_cgf_derivative_bound}. $K_n''(0)=\Omega_\P(1)$ is guaranteed by assumption \eqref{eq:lower_bound_variance}. Thus by Lemma \ref{lem:conditional-clt}, we know 
	\begin{align*}
		\sup_{t\in\mathbb{R}}\left|\P\left[\frac{1}{\sqrt{nK_n''(0)}}\sum_{i=1}^n W_{in}\leq t|\mathcal{F}_n\right]-\Phi(t)\right|\convp 0.
	\end{align*}

	\paragraph{Proof of \eqref{eq:nondegeneracy}:}
		Fix $\delta>0$. Then we can bound 
		\begin{align*}
			\P\left[\frac{1}{\sqrt{nK_n''(0)}}\sum_{i=1}^n W_{in}=y_n|\mathcal{F}_n\right]
			&
			\leq \P\left[\frac{1}{\sqrt{nK_n''(0)}}\sum_{i=1}^n W_{in}\in (y_n-\delta,y_n+\delta]|\mathcal{F}_n\right]\\
			&
			\equiv P((y_n-\delta,y_n+\delta])
		\end{align*}
		where 
		\begin{align*}
			P(A)\equiv \P\left[\frac{1}{\sqrt{nK_n''(0)}}\sum_{i=1}^n W_{in}\in A|\mathcal{F}_n\right],\ A\subset \mathbb{R}.
		\end{align*}
		Furthermore we have 
		\begin{align*}
			P((y_n-\delta,y_n+\delta])
			&
			\leq 
			\left|P((-\infty, y_n+\delta])-\Phi(y_n+\delta)\right|+\left|P((-\infty, y_n-\delta])-\Phi(y_n-\delta)\right|\\
			&
			\quad  + |\Phi(y_n+\delta)-\Phi(y_n-\delta)|.
		\end{align*}
		By \eqref{eq:conditional_uniform_CLT} and the Lipschitz continuity of $\Phi(x)$, we can bound 
		\begin{align*}
			P((y_n-\delta,y_n+\delta])
			&
			\leq 2\sup_{t\in\mathbb{R}}\left|\P\left[\frac{1}{\sqrt{nK_n''(0)}}\sum_{i=1}^n W_{in}\leq t|\mathcal{F}_n\right]-\Phi(t)\right|+\sup_{x\in\mathbb{R}}\phi(x)2\delta\\
			&
			=\sup_{x\in\mathbb{R}}\phi(x)2\delta+o_\P(1).
		\end{align*}
		Since $\sup_{x\in\mathbb{R}}\phi(x)\leq 1/\sqrt{2\pi}$, we know 
		\begin{align*}
			\P\left[\frac{1}{\sqrt{nK_n''(0)}}\sum_{i=1}^n W_{in}=y_n|\mathcal{F}_n\right]\leq P((y_n-\delta,y_n+\delta])=o_\P(1)+\frac{2\delta}{\sqrt{2\pi}}.
		\end{align*}
		We can take $\delta$ arbitrarily small so that we obtain
		\begin{align*}
			\P\left[\frac{1}{\sqrt{nK_n''(0)}}\sum_{i=1}^n W_{in}=y_n|\mathcal{F}_n\right]=o_\P(1).
		\end{align*}
\end{proof}

\subsection{Proof of Lemma \ref{lem:tilted_measure_properties}}
  
\begin{proof}[Proof of Lemma \ref{lem:tilted_measure_properties}]

To prove the statement~\eqref{eq:preserving_measurable_events}, note that
\begin{equation*}
\begin{split}
	\P_{n, s_n}[A_n] &= \E\left[\indicator(A_n)\prod_{i = 1}^n \frac{\exp(s_n W_{in})}{\E[\exp(s_n W_{in}) \mid \mathcal F_n]}\right] \\
	&= \E\left[\E\left[\indicator(A_n)\prod_{i = 1}^n \frac{\exp(s_n W_{in})}{\E[\exp(s_n W_{in}) \mid \mathcal F_n]} \mid \mathcal F_n \right]\right]  \\
	&= \E\left[\indicator(A_n)\prod_{i = 1}^n \frac{\E[\exp(s_n W_{in}) \mid \mathcal F_n]}{\E[\exp(s_n W_{in}) \mid \mathcal F_n]}\right] = \P[A_n].
\end{split}
\end{equation*}
Next, we compute for each $A_n \in \mathcal F_n$ and $B_1, \dots, B_n \subseteq \mathcal B(\R)$ that
\begin{equation}
\begin{split}
&\P_{n, s_n}[W_{1n} \in B_1, \dots, W_{nn} \in B_n, A_n] \\
&= \E\left[\indicator(W_{1n} \in B_1, \dots, W_{nn} \in B_n, A_n)\prod_{i = 1}^n \frac{\exp(s_n W_{in})}{\E[\exp(s_n W_{in}) \mid \mathcal F_n]}\right] \\
&= \E\left[\indicator(A_n)\prod_{i = 1}^n \frac{\E[\indicator(W_{in} \in B_i)\exp(s_n W_{in}) \mid \mathcal F_n]}{\E[\exp(s_n W_{in}) \mid \mathcal F_n]}\right],
\end{split}
\label{eq:conditional_tilting}
\end{equation}
from which it follows that
\begin{equation*}
\P_{n, s_n}[W_{1n} \in B_1, \dots, W_{nn} \in B_n \mid \mathcal F_n] = \prod_{i = 1}^n \frac{\E[\indicator(W_{in} \in B_i)\exp(s_n W_{in}) \mid \mathcal F_n]}{\E[\exp(s_n W_{in}) \mid \mathcal F_n]}
\end{equation*}
This verifies the claim that under $\P_{n, s_n}$, $(W_{1n}, \dots, W_{nn})$ are still independent conditionally on $\mathcal F_n$. Furthermore, this shows that the marginal distribution of each $W_{in}$ is exponentially tilted by $s_n$, conditionally on $\mathcal F_n$. From this, we can derive the conditional mean and variance of $W_{in}$ under the measure $\P_{n,s_n}$. We write 
\begin{align*}
	\E_{n,s_n}[W_{in}\mid \mathcal{F}_n]=\E\left[W_{in}\prod_{i = 1}^n \frac{\exp(s_n W_{in})}{\E[\exp(s_n W_{in}) \mid \mathcal F_n]}\mid \mathcal{F}_n\right]=\frac{\E\left[W_{in}\exp(s_nW_{in})\mid \mathcal{F}_n\right]}{\E[\exp(s_n W_{in}) \mid \mathcal F_n]}.
\end{align*}
Then by Lemma \ref{lem:tilted_moment}, we have,
\begin{align*}
	\P\left[\mathcal{T}\right]=1,\ \mathcal{T}\equiv \left\{K_{in}'(s)=\E_{in,s}[W_{in}|\mathcal{F}_n],\ \forall s\in (-\varepsilon,\varepsilon)\right\}.
\end{align*}
Then we know $\forall \omega\in\mathcal{T}\cap \{|s_n|<\varepsilon\}$,
\begin{align*}
	K_{in}'(s_n)(\omega)=\E_{in,s_n}[W_{in}|\mathcal{F}_n](\omega)
	&
	=\int x\frac{\exp(s_nx)}{\int \exp(s_nx)\mathrm{d}\kappa_{in}(\omega,x)}\mathrm{d}\kappa_{in}(\omega,x)\\
	&
	=\E_{n,s_n}[W_{in}|\mathcal{F}_n](\omega),
\end{align*}
so that by the assumption $\P[s_n\in (-\varepsilon,\varepsilon)]=1$,
\begin{align*}
	\P\left[K_{in}'(s_n)=\E_{n,s_n}[W_{in}|\mathcal{F}_n]\right]=1.
\end{align*}
Similarly, we have 
\begin{align*}
	\P\left[K_{in}''(s_n)=\mathrm{Var}_{n,s_n}[W_{in}|\mathcal{F}_n]\right]=1.
\end{align*}
\end{proof}

	\subsection{Proof of Lemma \ref{lem:conditional-berry-esseen}}

	\begin{proof}[Proof of Lemma \ref{lem:conditional-berry-esseen}]
		Define 
		\begin{align*}
			F_n(t,\omega)\equiv \P\left[\frac{1}{S_n\sqrt{n}}\sum_{i=1}^n (W_{in}-\E[W_{in}|\mathcal{F}_n])\leq t|\mathcal{F}_n\right](\omega)
		\end{align*}
		and 
		\begin{align*}
			S_n(\omega)\equiv \left(\frac{1}{n}\sum_{i=1}^n \E[(W_{in}-\E[W_{in}|\mathcal{F}_n])^2|\mathcal{F}_n](\omega)\right)^{1/2}.
		\end{align*}
		We prove the result by recalling the notion of regular conditional distribution defined in Appendix \ref{sec:RCD_preliminary}. Define $W_n\equiv (W_{1n},\ldots,W_{nn})$. Suppose $\kappa_{W_n,\mathcal{F}_n}$ is a regular conditional distribution of $W_n$ given $\mathcal{F}_n$. Then for every $\omega\in\Omega$, we know $\kappa_{W_n,\mathcal{F}_n}(\omega,\cdot)$ is a probability measure. We draw $(\widetilde W_{1n}(\omega),\ldots,\widetilde{W}_{nn}(\omega))\sim \kappa_{W_n,\mathcal{F}_n}(\omega,\cdot)$. Define 
		\begin{align*}
			\tilde S_n(\omega)\equiv \left(\frac{1}{n}\sum_{i=1}^n \E[(\widetilde{W}_{in}(\omega)-\E[\widetilde W_{in}(\omega)])^2]\right)^{1/2},\ \mathcal{D}_n\equiv \left\{S_n>0\right\}.
		\end{align*}
		In order to apply Lemma \ref{lem:berry-esseen} to almost every $\omega\in \mathcal{D}_n\subset \Omega$, we need to verify that for those $\omega$ it is true that $\forall  i\in \{1,\ldots,n\},$
		\begin{align}\label{eq:RCD_moment_condition}
			\sum_{i=1}^n\E\left[\frac{(\widetilde{W}_{in}(w)-\E[\widetilde{W}_{in}(w)])^2}{\widetilde S_n^2(\omega)}\right]=1,\ \E\left[\frac{\widetilde{W}_{in}(w)-\E[\widetilde{W}_{in}(w)]}{\widetilde S_n(\omega)}\right]=0
		\end{align}
		and $\widetilde S_n(\omega)>0$. Both claims are true by applying Lemma \ref{lem:Klenke_Thm_8.38}, such that for almost every $\omega\in\Omega$, we have for any positive integer $p$,
		\begin{align*}
			\E[W^p_{in}|\mathcal{F}_n](\omega)=\E[\widetilde{W}^p_{in}(\omega)],\ \E[|W_{in}|^p|\mathcal{F}_n](\omega)=\E[|\widetilde{W}_{in}(\omega)|^p],\ \forall i\in \{1,\ldots,n\},\ n\geq 1.
		\end{align*}
		Together with the assumption imposed in the lemma, we know conditions in \eqref{eq:RCD_moment_condition} are statisfied for any $\omega\in \mathcal{D}_n\cap \mathcal{N}^{c}$, where $\mathcal{N}$ is a null set with probability measure $0$. Then we apply Lemma \ref{lem:berry-esseen} to obtain that for any fixed $t\in\mathbb{R}$ there exists a universal constant, that is independent of $\omega$, such that $\forall \omega\in \mathcal{D}_n\cap \mathcal{N}^{c}$
		\begin{align*}
			\left|\P\left[\frac{\sum_{i=1}^n (\widetilde W_{in}(\omega)-\E[\widetilde W_{in}(\omega)])}{\widetilde S_n(\omega)\sqrt{n}}\leq t\right]-\Phi(t)\right|\leq C\frac{\sum_{i=1}^n \E[|\widetilde W_{in}(\omega)-\E[\widetilde W_{in}(\omega)]|^3]}{\widetilde S_n^3(\omega)n^{3/2}}
		\end{align*}
		Then fixing any $t\in\mathbb{R}$, we again appy Lemma \ref{lem:Klenke_Thm_8.38} such that for almost every $\omega\in\mathcal{C}_n\cap \mathcal{D}_n$,
		\begin{align*}
			|F_n(t,\omega)-\Phi(t)|\leq C\frac{\sum_{i=1}^n \E[|W_{in}-\E[W_{in}|\mathcal{F}_n]|^3|\mathcal{F}_n](\omega)}{S_n(\omega)n^{3/2}}.
		\end{align*}
		Fix $k\in\mathbb{N}$. By the continuity of the normal CDF, there exists points $-\infty=x_0<x_1\cdots<x_k=\infty$ with $\Phi(x_i)=i/k$. By monotonicity, we have, for $x_{i-1}\leq t\leq x_i$,
		\begin{align*}
			F_n(t,\omega)-\Phi(t)\leq F_n(x_i,\omega)-\Phi(x_{i-1})=F_n(x_i,\omega)-\Phi(x_i)+\frac{1}{k}
		\end{align*}
		and 
		\begin{align*}
			F_n(t,\omega)-\Phi(t)\geq F_n(x_{i-1},\omega)-\Phi(x_i)=F_n(x_{i-1},\omega)-\Phi(x_{i-1})-\frac{1}{k}.
		\end{align*}
		Thus for fixed $x\in\mathbb{R}$, we can bound for almost every $\omega\in\mathcal{D}_n$ that
		\begin{align*}
			|F_n(t,\omega)-\Phi(t)|
			&
			\leq \sup_{i}|F_n(x_i,\omega)-\Phi(x_{i})|+\frac{1}{k}\\
			&
			\leq C\frac{\sum_{i=1}^n \E[|W_{in}-\E[W_{in}|\mathcal{F}_n]|^3|\mathcal{F}_n](\omega)}{S_n(\omega)n^{3/2}}+\frac{1}{k}.
		\end{align*}
		Then taking the supremum on $t\in\mathbb{R}$, we have almost every $\omega\in\mathcal{D}_n$,
		\begin{align*}
			\sup_{t\in\mathbb{R}}|F_n(t,\omega)-F(t)|\leq C\frac{\sum_{i=1}^n \E[|W_{in}-\E[W_{in}|\mathcal{F}_n]|^3|\mathcal{F}_n](\omega)}{S_n^3(\omega)n^{3/2}}+\frac{1}{k}.
		\end{align*}
		Letting $k$ go to infinity, we have
		\begin{align}\label{eq:Berry_Esseen_bound_on_C_n}
			\indicator(\omega\in\mathcal{D}_n)\sup_{t\in\mathbb{R}}|F_n(t,\omega)-F(t)|\leq C\indicator(\omega\in\mathcal{D}_n)\frac{\sum_{i=1}^n \E[|W_{in}-\E[W_{in}|\mathcal{F}_n]|^3|\mathcal{F}_n](\omega)}{S_n^3(\omega)n^{3/2}}.
		\end{align}
		Now we decompose 
		\begin{align*}
			&
			\sqrt{n}\sup_{t\in\mathbb{R}}|F_n(t,\omega)-F(t)|\\
			&
			=\sqrt{n}\indicator(\omega\notin \mathcal{D}_n)\sup_{t\in\mathbb{R}}|F_n(t,\omega)-F(t)|+\sqrt{n}\indicator(\omega\in\mathcal{D}_n)\sup_{t\in\mathbb{R}}|F_n(t,\omega)-F(t)|.
		\end{align*}
		We first show that $\indicator(\omega\notin\mathcal{D}_n)=o_{\P}(1)$. We only need to prove 
		\begin{align*}
			\P[\mathcal{D}_n^{c}]=\P[S_n\leq 0]\rightarrow0.
		\end{align*}
		This is obvious since $S_n=\Omega_{\P}(1)$. Then we have $\indicator(\omega\notin\mathcal{D}_n)=o_{\P}(1)$ such that 
		\begin{align}\label{eq:bound_on_Cn_Dn_complement}
			\sqrt{n}\indicator(\omega\notin\mathcal{D}_n)\sup_{t\in\mathbb{R}}|F_n(t,\omega)-F(t)|=o_{\P}(1).
		\end{align}
		By \eqref{eq:Berry_Esseen_bound_on_C_n} and 
		\begin{align*}
			\frac{\sum_{i=1}^n \E[|W_{in}-\E[W_{in}|\mathcal{F}_n]|^3|\mathcal{F}_n]}{n}=O_{\P}(1), S_n=\Omega_{\P}(1),
		\end{align*}
		we know there exists $M>0$ such that for any $m\geq M$, we have 
		\begin{align*}
			\P\left[\sqrt{n}\indicator(\omega\in\mathcal{D}_n)\sup_{t\in\mathbb{R}}|F_n(t,\omega)-F(t)|> m\right]\rightarrow0.
		\end{align*}
		This, together with \eqref{eq:bound_on_Cn_Dn_complement}, implies for any $m\geq M$ we have 
		\begin{align*}
			\P\left[\sqrt{n}\sup_{t\in\mathbb{R}}|F_n(t,\omega)-F(t)|> m\right]\rightarrow0.
		\end{align*}
	  \end{proof}

\subsection{Proof of Lemma~\ref{lem:conditional_clt_assumptions}}

The statement~\eqref{eq:variance-bounded-below} can be verified using a Taylor expansion, guaranteed by Lemma \ref{lem:finite_cgf}, of $K''_{n}(s)$ around $s = 0$:
\begin{equation*}
\begin{split}
\smash{\frac{1}{n}\sum_{i = 1}^n} \V_{n, \hat s_n}[W_{in} \mid \mathcal F_n] &= K''_n(\hat s_n) \\
&= K''_n(0) + \hat s_n K'''_n(0) + \frac12 \hat s_n^2 K''''_n(\tilde s_n) \\
&= \Omega_\P(1) + o_\P(1) + o_\P(1) \\
&= \Omega_\P(1),
\end{split}
\end{equation*}
where $|\tilde s_n| \leq |\hat s_n|$. The statement $K''_n(0) = \Omega_\P(1)$ is by assumption~\eqref{eq:lower_bound_variance}, the statement $\hat s_n K'''_n(0)$ is by the convergence $\hat s_n \convp 0$~\eqref{eq:hat_s_n_convergence} and the finiteness of $K'''_n(0)$~\eqref{eq:finite_cgf_derivatives}, and the statement $\hat s_n^2 K''''_n(\tilde s_n) = o_\P(1)$ is by the convergence $\hat s_n \convp 0$ and the assumption~\eqref{eq:fourth_cgf_derivative_bound}. 

To verify the moment condition~\eqref{eq:third-moment-bound}, by Lemma \ref{lem:moment_dominance} with $p=3,q=4$, it suffices to verify a stronger fourth moment statement
\begin{align}
	\frac{1}{n}\sum_{i = 1}^n \E_{n, \hat s_n}[(W_{in} - \E_{n, \hat s_n}[W_{in}])^4 \mid \mathcal F_n] = O_{\P_{n, \hat s_n}}(1). \label{eq:bounded_fourth_moment}
\end{align}
To this end, we combine an expression for the fourth central moment of $W_{in}$ in terms of the second and fourth cumulants and the assumptions~\eqref{eq:second_cgf_derivative_bound} and~\eqref{eq:fourth_cgf_derivative_bound}:
\begin{equation*}
\frac{1}{n}\sum_{i = 1}^n \E_{n, \hat s_n}[(W_{in} - \E_{n, \hat s_n}[W_{in}])^4 \mid \mathcal F_n] = \frac{1}{n}\sum_{i = 1}^n \left\{K''''_{in}(\hat s_n) + 3(K''_{in}(\hat s_n))^2\right\} = O_\P(1).
\end{equation*}

\subsection{Proof of Lemma~\ref{lem:tilting_back}}

For any $A_n \in \mathcal F_n$, we have
\begin{equation*}
\begin{split}
&\E\left[\indicator\left(\frac{1}{n}\sum_{i = 1}^n W_{in} \geq w_n\right)\indicator(A_n)\right] \\
&\quad = \E_{n, \hat s_n}\left[\indicator\left(\frac{1}{n}\sum_{i = 1}^n W_{in} \geq w_n\right)\indicator(A_n)\frac{d\P}{d\P_{n, \hat s_n}}\right] \\
&\quad = \E_{n, \hat s_n}\left[\E_{n, \hat s_n}\left[\indicator\left(\frac{1}{n}\sum_{i = 1}^n W_{in} \geq w_n\right)\indicator(A_n)\frac{d\P}{d\P_{n, \hat s_n}} \mid \mathcal F_n\right] \right] \\
&\quad = \E_{n, \hat s_n}\left[\E_{n, \hat s_n}\left[\indicator\left(\frac{1}{n}\sum_{i = 1}^n W_{in} \geq w_n\right)\frac{d\P}{d\P_{n, \hat s_n}} \mid \mathcal F_n\right] \indicator(A_n) \right] \\
&\quad = \E\left[\E_{n, \hat s_n}\left[\indicator\left(\frac{1}{n}\sum_{i = 1}^n W_{in} \geq w_n\right)\frac{d\P}{d\P_{n, \hat s_n}} \mid \mathcal F_n\right] \indicator(A_n) \right].
\end{split}
\end{equation*}
The equality is due to Lemma~\ref{lem:tilted_measure_properties}, since the random variable inside the expectation is measurable with respect to $\mathcal F_n$.

\subsection{Proof of Lemma \ref{lem:Gaussian_integral_approximation_additive_error}}

\begin{proof}[Proof of Lemma \ref{lem:Gaussian_integral_approximation_additive_error}]
	Recall the notion of regular conditional distribution introduced in Appendix \ref{sec:RCD_preliminary}. We know $Z_n|\mathcal{F}_n$ must admit the regular conditional distribution $\kappa_n(\omega,B)$ for $B\in\mathcal{B}(\mathbb{R}^n)$. We define $F_{n}(\cdot,\omega)$ to be the CDF of $Z_n|\mathcal{F}_n$ for the probability measure $\kappa_n(\omega,\cdot)$. We apply Lemma \ref{lem:Klenke_Thm_8.38} and the integration by parts formula to obtain, for almost every $\omega\in\Omega$,
	\begin{align}
		&\nonumber
		\indicator(\lambda_n\geq 0)\E\left[\indicator(Z_n \geq 0)\exp\left( - \lambda_n Z_n \right) \mid \mathcal{F}_n\right](\omega)\\
		&\nonumber
		=\indicator(\lambda_n\geq 0)\int_0^{\infty}\exp(-\lambda_n z)\mathrm{d}F_{n}(z,\omega)\\
		&\label{eq:G_n_1}
		=-\indicator(\lambda_n\geq 0)F_{n}(0,\omega)+\indicator(\lambda_n\geq 0)\int_0^{\infty}\lambda_n\exp(-\lambda_nz)F_{n}(z,\omega)\mathrm{d}z.
	\end{align}
	Similarly, apply integration by parts so that we have
	\begin{align}\label{eq:G_n_2}
		\indicator(\lambda_n\geq 0)\int_0^{\infty}\exp(-\lambda_n z)\phi(z)\mathrm{d}z=\indicator(\lambda_n\geq 0)\left(-\Phi(0)+\int_0^{\infty}\lambda_n\exp(-\lambda_nz)\Phi(z)\mathrm{d}z\right).
	\end{align}
	Then combining \eqref{eq:G_n_1} and \eqref{eq:G_n_2}, we can bound 
	\begin{align*}
		&
		\indicator(\lambda_n \geq 0)\left|\E\left[\indicator(Z_n \geq 0)\exp\left(- \lambda_n Z_n \right) \mid \mathcal{F}_n\right](\omega) - \int_0^\infty \exp(-\lambda_n z)\phi(z)dz\right|\\
		&
		= \indicator(\lambda_n\geq 0)\left|\Phi(0)-F_{n}(0,\omega)+\int_0^{\infty}\lambda_n\exp(-\lambda_nz)\left(F_{n}(z,\omega)-\Phi(z)\right)\mathrm{d}z\right|\\
		&
		\leq \indicator(\lambda_n\geq 0)\sup_{z\geq 0}|F_{n}(z,\omega)-\Phi(z)|\left(1+\int_{0}^{\infty}\lambda_n\exp(-\lambda_n z)\mathrm{d}z\right)\\
		&
		= 2\indicator(\lambda_n\geq 0)\sup_{z\geq 0}|F_{n}(z,\omega)-\Phi(z)|\\
		&
		\leq 2\indicator(\lambda_n\geq 0)\sup_{z\in\mathbb{R}}|F_{n}(z,\omega)-\Phi(z)|\\
		&
		= 2\indicator(\lambda_n\geq 0)\sup_{z\in\mathbb{R}}\left|\P\left[Z_n\leq z|\mathcal{F}_n \right](\omega)-\Phi(z)\right|
	\end{align*}
	almost surely. For the last equality, we use Lemma \ref{lem:Klenke_Thm_8.38} together with the density argument to prove the equality. Indeed, fixing any $k\in \mathbb{N}$, by the continuity of the normal CDF, there exists points $-\infty=x_0<x_1\cdots<x_k=\infty$ with $\Phi(x_i)=i/k$. By monotonicity, we have for $x_{i-1}\leq t\leq x_i$
	\begin{align*}
		F_n(t,\omega)-\Phi(t)\leq F_n(x_i,\omega)-\Phi(x_{i-1})=\P[Z_n\leq x_i|\mathcal{F}_n](\omega)-\Phi(x_i)+\frac{1}{k}
	\end{align*}
	and 
	\begin{align*}
		F_n(t,\omega)-\Phi(t)\geq F_n(x_{i-1},\omega)-\Phi(x_i)=\P[Z_n\leq x_{i-1}|\mathcal{F}_n](\omega)-\Phi(x_{i-1})-\frac{1}{k}
	\end{align*}
	for almost every $\omega\in\Omega$. Then we have for almost every $\omega\in\Omega$
	\begin{align*}
		|F_n(t,\omega)-\Phi(t)|
		&
		\leq \sup_{i}|\P[Z_n\leq x_{i}|\mathcal{F}_n](\omega)-\Phi(x_{i})|+\frac{1}{k}\\
		&
		\leq \sup_{t\in \mathbb{R}}|\P[Z_n\leq t|\mathcal{F}_n](\omega)-\Phi(t)|+\frac{1}{k}.
	\end{align*}
	Therefore by the arbitrary choice of $k$ so that we have 
	\begin{align*}
		\sup_{t\in\mathbb{R}}|F_n(t,\omega)-\Phi(t)|\leq \sup_{t\in \mathbb{R}}|\P[Z_n\leq t|\mathcal{F}_n](\omega)-\Phi(t)|
	\end{align*}
	almost surely. By interchanging the $F_n(t,\omega)$ and $\P[Z_n\leq t|\mathcal{F}_n](\omega)$, we have shown the desired result.
\end{proof}

\subsection{Proof of Lemma \ref{lem:relative_error_Berry_Esseen_bound}}

\begin{proof}[Proof of statement~\eqref{eq:U_n_r_n}]
	We consider the events $r_n = 0$ and $r_n > 0$ separately. First, define $\mathcal{U}_n\equiv \{U_n \neq 0\}$.

	\paragraph{On the event $r_n=0$:}

	We further divide this case into two cases.
	\begin{enumerate}
		\item When $\lambda_n=0$: this implies $|U_n|=1/2>0$;
		\item When $\lambda_n\neq 0$: this implies $|U_n|=\infty$.
	\end{enumerate}
	This implies $\mathcal{U}_n$ happens.

	\paragraph{On the event $r_n>0$:}
	
	By \eqref{eq:sign_condition_r_lambda} condition, this implies $\lambda_n\geq 0$. We divide the case to $\lambda_n>0$ and $\lambda_n=0$.
	\begin{enumerate}
		\item When $\lambda_n=0$: this implies $|U_n|=\infty$;
		\item When $\lambda_n>0$: we discuss when $r_n-\lambda_n>0,r_n-\lambda_n<0$ and $r_n-\lambda_n=0$. 

		\paragraph{When $r_n- \lambda_n\geq 0$:}
		By the formula of $U_n$, we have 
		\begin{align*}
			U_n
			&
			=\exp\left(\frac{r_n^2}{2}\right)(1-\Phi(r_n))+
			\frac{1}{\sqrt{2\pi}}\left\{\frac{1}{\lambda_n}-\frac{1}{r_n}\right\}\\
			&
			\geq \exp\left(\frac{r_n^2}{2}\right)(1-\Phi(r_n)).
		\end{align*}
		By \eqref{eq:finitness_r_n_lambda_n}, we know $r_n\in (-\infty,\infty)$ almost surely, so that $U_n>0$ almost surely.
		
		\paragraph{When $r_n-\lambda_n<0$:}

		In order to proceed the proof, we present a lemma to relate the $U_n$ with the Gaussian integral estimate via integration by parts.
		
		\begin{lemma}\label{lem:R_n_formula}
			Suppose \eqref{eq:finitness_r_n_lambda_n} condition is true. Define 
			\begin{align}\label{eq:R_n_def}
				R_n\equiv \int_{r_n}^{\lambda_n}
				y\exp(y^2/2)(1-\Phi(y))-\frac{1-y^{-2}}{\sqrt{2\pi}}
				\mathrm{d}y.
			\end{align}
			Recall the definition of $U_n$ as in \eqref{eq:Dn_Un_Def}. If $\lambda_n,r_n\neq 0$ almost surely, then we have
			\begin{align*}
				R_n=\int_{0}^{\infty}\exp(-\lambda_n y)\phi(y)\mathrm{d}y-U_n,\text{ almost surely}.
			\end{align*}
		\end{lemma}
		Since in this case, $\lambda_n>r_n>0$ and \eqref{eq:finitness_r_n_lambda_n} is assumed in the lemma statement, then by Lemma \ref{lem:R_n_formula} we have 
		\begin{align*}
			U_n=\int_{0}^{\infty}\exp(-\lambda_n y)\phi(y)\mathrm{d}y-R_n,\text{ almost surely}.
		\end{align*}
		By Lemma \ref{lem:Gaussian_tail_estimate},
		we can write 
		\begin{align*}
			R_n = \int_{r_n}^{\lambda_n} \left(-y\int_{y}^{\infty}\frac{\phi(t)}{3t^4}\mathrm{d}t\right)\mathrm{d}y<0,
		\end{align*}
		which implies $U_n>0$. Thus $\mathcal{U}_n$ happens.
	\end{enumerate}

	\paragraph{On the event $r_n<0$:}
	By \eqref{eq:sign_condition_r_lambda} condition, we know $\lambda_n\leq 0$. We divide the case to $\lambda_n<0$ and $\lambda_n=0$.

	\begin{enumerate}
		\item When $\lambda_n<0$: we can lower bound 
		\begin{align*}
			U_n=\exp\left(\frac{r_n^2}{2}\right)(1-\Phi(r_n))+
			\frac{1}{\sqrt{2\pi}}\left\{\frac{1}{\lambda_n}-\frac{1}{r_n}\right\}\geq \frac{1}{2} -\frac{1}{\sqrt{2\pi}}\left|\frac{1}{\lambda_n}-\frac{1}{r_n}\right|
		\end{align*}
		By \eqref{eq:rate_1}, we have 
		\begin{align*}
			\frac{1}{\sqrt{2\pi}}\left|\frac{1}{\lambda_n}-\frac{1}{r_n}\right|=o_{\P}(1).
		\end{align*}
		Thus we have
		\begin{align*}
			\P\left[\mathcal{U}_n^c\text{ and } r_n<0\text{ and } \lambda_n<0\right]\leq \P\left[U_n\geq \frac{1}{4}\text{ and } r_n<0\right]\rightarrow0.
		\end{align*}

		\item 
		When $\lambda_n=0$: we know $|U_n|=\infty$. Thus $\mathcal{U}_n$ happens. This completes the proof.
	\end{enumerate}
\end{proof}

\begin{proof}[Proof of statement~\eqref{eq:U_n_rate}]
	We write
	\begin{align*}
		\indicator(r_n\geq 0)\frac{1}{\sqrt{n}U_n}=\indicator(r_n \geq 1)\frac{1}{r_nU_n}\frac{r_n}{\sqrt{n}}+\indicator(r_n\in [0,1))\frac{1}{\sqrt{n}U_n}.
	\end{align*}
	Now we present an auxiliary lemma.
	\begin{lemma}[Convergence rate of $1/U_n$]\label{lem:convergence_rate_denominator_relative_error}
		Suppose \eqref{eq:rate_1} and \eqref{eq:rate_2} hold. Then we have 
		\begin{align}
			\frac{\indicator(r_n\geq 1)}{r_nU_n}
			&\label{eq:relative_error_denominator_1}
			=O_{\P}(1)\\
			\indicator(r_n\in [0,1))\frac{1}{U_n}
			&
			\label{eq:relative_error_denominator_2}
			=O_{\P}(1).
		\end{align}
	\end{lemma}
	\paragraph{Intuition of Lemma \ref{lem:convergence_rate_denominator_relative_error}:}
	The intuition behind this is when $r_n$ is small, we expect $|U_n|$, is lower bounded with high probability since $1/\lambda_n-1/r_n=o_{\P}(1)$ and thus the dominant term is $\exp(r_n^2/2)(1-\Phi(r_n))$, which is lower bounded when $r_n$ is small. When $r_n$ is large, $|U_n|$ will go to zero but with a rate that is slower than $1/r_n$. The latter case needs a finer analysis with \eqref{eq:rate_1} and \eqref{eq:rate_2} conditions involved. 
	
	Then by Lemma \ref{lem:convergence_rate_denominator_relative_error}, we know 
	\begin{align*}
		\indicator(r_n\geq 1)\frac{1}{r_nU_n}=O_{\P}(1),\ \indicator(r_n\in [0,1))\frac{1}{U_n}=O_{\P}(1).
	\end{align*}
	Since $r_n/\sqrt{n}=o_{\P}(1)$, we conclude 
	\begin{align*}
		\indicator(r_n\geq 0)\frac{1}{\sqrt{n}U_n}=o_{\P}(1).
	\end{align*}
	This completes the proof.
  \end{proof}

  \subsection{Proof of Lemma \ref{lem:final_result_except_lam_0}}

	\begin{proof}[Proof of Lemma \ref{lem:final_result_except_lam_0}]
		Define 
		\begin{align*}
			O_n\equiv \frac{|r_n-\lambda_n|}{\sqrt{2\pi}}\left(\frac{1}{r_n^2}+\frac{1}{\lambda_n^2}\right).
		\end{align*}
		We first present an auxiliary lemma.
		\begin{lemma}[Upper bound of Gaussian integral]\label{lem:upper_bound_Gaussian_integral}
			Under conditions \eqref{eq:finitness_r_n_lambda_n} and \eqref{eq:sign_condition_r_lambda}, the following inequality is true almost surely:
			 \begin{align*}
				\indicator(r_n> 0,\lambda_n> 0)\left|\frac{\int_0^\infty \exp(-\lambda_n z)\phi(z)dz}{U_n}-1\right|\leq \indicator(r_n > 0,\lambda_n > 0)\left|\frac{1}{U_n}\right|\cdot O_n.
			\end{align*}
		\end{lemma}
		\noindent By Lemma \ref{lem:upper_bound_Gaussian_integral}, we can bound 
		\begin{align*}
			\indicator(r_n> 0,\lambda_n> 0)\left|\frac{\int_0^\infty \exp(-\lambda_n z)\phi(z)dz}{U_n}-1\right|\leq\left| \indicator(r_n> 0,\lambda_n> 0)\frac{O_n}{U_n} \right|
		\end{align*}
		almost surely. Then we can further decompose 
		\begin{align*}
			\indicator(r_n> 0,\lambda_n> 0)\frac{O_n}{U_n}=\frac{\indicator(r_n\geq 1,\lambda_n> 0)}{r_nU_n}\cdot r_n O_n+ \indicator(r_n\in (0,1),\lambda_n> 0)\frac{O_n}{U_n}
		\end{align*}
		By Lemma \ref{lem:convergence_rate_denominator_relative_error}, we know 
	\begin{align*}
		\indicator(r_n\geq 1)\frac{1}{r_nU_n}=O_{\P}(1),\ \indicator(r_n\in (0,1))\frac{1}{U_n}=O_{\P}(1).
	\end{align*}
	Thus it suffices to show 
	\begin{align}\label{eq:remainder_bound_1}
		r_n O_n =o_{\P}(1)
	\end{align}
	and 
	\begin{align}\label{eq:remainder_bound_2}
		\indicator(r_n>0,\lambda_n>0)O_n=o_{\P}(1).
	\end{align}
	
	\paragraph{Proof of \eqref{eq:remainder_bound_1}:}

	We compute 
	\begin{align*}
		r_n O_n=\frac{1}{\sqrt{2\pi}}\left|1-\frac{\lambda_n}{r_n}\right|\cdot\left|1+\frac{r_n^2}{\lambda_n^2}\right|.
	\end{align*}
	Thus by \eqref{eq:rate_2} we know 
	\begin{align*}
		\left|\frac{\lambda_n}{r_n}-1\right|=o_{\P}(1),\ \frac{r_n^2}{\lambda_n^2}=O_{\P}(1).
	\end{align*}
	Thus we have $\indicator(r_n\geq 1,\lambda_n>0)r_n O_n=o_{\P}(1)$.

	\paragraph{Proof of \eqref{eq:remainder_bound_2}:}

	We can write 
	\begin{align*}
		O_n = \frac{|r_n-\lambda_n|}{\sqrt{2\pi}}\left(\frac{1}{r_n^2}+\frac{1}{\lambda_n^2}\right)=\frac{1}{\sqrt{2\pi}}\left|\left(\frac{\lambda_n}{r_n}-1\right)\frac{1}{r_n}\right|\cdot\left|1+\frac{r_n^2}{\lambda_n^2}\right|.
	\end{align*}
	Then by \eqref{eq:rate_3}, we know 
	\begin{align*}
		\indicator(r_n>0,\lambda_n>0)\left|\left(\frac{\lambda_n}{r_n}-1\right)\frac{1}{r_n}\right|=o_{\P}(1)
	\end{align*}
	and by \eqref{eq:rate_2}, we have $r_n^2/\lambda_n^2=O_{\P}(1)$. Thus we have $O_n=o_{\P}(1)$.
	\end{proof}

\subsection{Proof of Lemma \ref{lem:ratio_convergence}}

\begin{proof}[Proof of Lemma \ref{lem:ratio_convergence}]
	By the rate condition \eqref{eq:rate_2}, we know 
	\begin{align*}
		1+\frac{\lambda_n^2}{r_n^2}=O_{\P}(1).
	\end{align*}
	Thus we only need to show 
	\begin{align*}
		\indicator(\lambda_n\neq 0)\frac{1-\frac{r_n}{\lambda_n}}{\lambda_nh(\lambda_n)}=o_{\P}(1).
	\end{align*}
	We decompose the magnitude of $|\lambda_n|$ to two parts: $|\lambda_n|> 1$ and $|\lambda_n|\in (0,1]$. It suffices to prove 
	\begin{align}\label{eq:ratio_decomposition}
		\indicator(|\lambda_n|\in (0,1])\frac{1-\frac{r_n}{\lambda_n}}{\lambda_nh(\lambda_n)}=o_{\P}(1),\ \indicator(|\lambda_n|>1)\frac{1-\frac{r_n}{\lambda_n}}{\lambda_nh(\lambda_n)}=o_{\P}(1).
	\end{align}
	For the first term in \eqref{eq:ratio_decomposition}, we know $h(x)$ is unformly lower bounded for $x\in[-1,1]$ so that $h(\lambda_n)$ is uniformly lower bounded for $|\lambda_n|\in (0,1]$. Then by the rate condition \eqref{eq:rate_4}, we know 
	\begin{align*}
		\indicator(|\lambda_n|\in (0,1])\frac{1-\frac{r_n}{\lambda_n}}{\lambda_nh(\lambda_n)}=o_{\P}(1).
	\end{align*}
	For the second term in \eqref{eq:ratio_decomposition}, we have by Lemma \ref{lem:lower_bound_Gaussian} that for $|\lambda_n|> 1$,
	\begin{align*}
		|\lambda_nh(\lambda_n)|=|\lambda_n|\exp(\lambda_n^2/2)(1-\Phi(\lambda_n))
		&
		\geq |\lambda_n|\exp(\lambda_n^2/2)(1-\Phi(|\lambda_n|))\\
		&
		\geq \frac{1}{\sqrt{2\pi}}\frac{\lambda_n^2}{\lambda_n^2+1}> \frac{1}{2\sqrt{2\pi}}.
	\end{align*}
	Then by the rate condition \eqref{eq:rate_2}, we know 
	\begin{align*}
		\indicator(|\lambda_n|>1)\frac{|1-\frac{r_n}{\lambda_n}|}{|\lambda_nh(\lambda_n)|}\leq 2\sqrt{2\pi}\indicator(|\lambda_n|>1)\left|1-\frac{r_n}{\lambda_n}\right|=o_{\P}(1).
	\end{align*}
\end{proof}

\subsection{Proof of Lemma \ref{lem:ratio_vanish}}

\begin{proof}[Proof of Lemma \ref{lem:ratio_vanish}]
	We decompose the magnitude of $|\lambda_n|$ to two parts $|\lambda_n|>1$ and $|\lambda_n|\in (0,1]$. It suffices to prove 
	\begin{align}\label{eq:ratio_vanish_decomposition}
		\indicator(|\lambda_n| \in(0,1])\frac{1-\frac{\lambda_n}{r_n}}{\lambda_nh(\lambda_n)}=o_{\P}(1),\ \indicator(|\lambda_n| >1)\frac{1-\frac{\lambda_n}{r_n}}{\lambda_nh(\lambda_n)}=o_{\P}(1).
	\end{align} 
	For the first term in \eqref{eq:ratio_vanish_decomposition}, we know $h(x)$ is uniformly bounded for $x\in[-1,1]$ so that $h(\lambda_n)$ is uniformly bounded for $|\lambda_n|\in (0,1]$. Then by the rate condition \eqref{eq:rate_1}, we know 
	\begin{align*}
		\indicator(|\lambda_n| \in(0,1])\frac{1-\frac{\lambda_n}{r_n}}{\lambda_nh(\lambda_n)}=\indicator(|\lambda_n| \in(0,1])\frac{\frac{1}{\lambda_n}-\frac{1}{r_n}}{h(\lambda_n)}=o_{\P}(1).
	\end{align*}
	For the second term in \eqref{eq:ratio_vanish_decomposition},  we have by Lemma \ref{lem:lower_bound_Gaussian} that for $|\lambda_n|> 1$,
	\begin{align*}
		|\lambda_nh(\lambda_n)|=|\lambda_n|\exp(\lambda_n^2/2)(1-\Phi(\lambda_n))
		&
		\geq |\lambda_n|\exp(\lambda_n^2/2)(1-\Phi(|\lambda_n|))\\
		&
		\geq \frac{1}{\sqrt{2\pi}}\frac{\lambda_n^2}{\lambda_n^2+1}> \frac{1}{2\sqrt{2\pi}}.
	\end{align*}
	Then by the rate condition \eqref{eq:rate_2}, we know
	\begin{align*}
		\indicator(|\lambda_n| >1)\frac{|1-\frac{\lambda_n}{r_n}|}{|\lambda_nh(\lambda_n)|}\leq 2\sqrt{2\pi}\indicator(|\lambda_n| >1)\left|1-\frac{\lambda_n}{r_n}\right|=o_{\P}(1).
	\end{align*}
\end{proof}

\subsection{Proof of Lemma \ref{lem:reduced_condition}}

\begin{proof}[Proof of Lemma \ref{lem:reduced_condition}]
	The following lemma states how the derivatives of $K_{in}(s)$ are related to the conditional moments of $W_{in}|\mathcal{F}_n$ under measure $\kappa_{in,s}$.
	\begin{lemma}\label{lem:tilted_moment}
		On the event $\mathcal{A}$ as in Lemma \ref{lem:finite_cgf}, we have 
		\begin{align}
			K_{in}'(s)=\E_{in,s}[W_{in}|\mathcal{F}_n],\ \forall s\in (-\varepsilon,\varepsilon),
			&\label{eq:first_moment_relationship}\\
			K_{in}''(s)=\V_{in,s}[W_{in}|\mathcal{F}_n],\ \forall s\in (-\varepsilon,\varepsilon),
			&\label{eq:second_moment_relationship}\\
			K_{in}^{(4)}(s)=\E_{in,s}[(W_{in}-\E_{in,s}[W_{in}|\mathcal{F}_n])^4|\mathcal{F}_n]-3\V^2_{in,s}[W_{in}|\mathcal{F}_n],\ \forall s\in (-\varepsilon,\varepsilon)
			&\label{eq:fourth_moment_relationship}.
		\end{align}
	\end{lemma}
	\noindent We first show with Lemma \ref{lem:tilted_moment}, in order to show condition \eqref{eq:second_cgf_derivative_bound}-\eqref{eq:fourth_cgf_derivative_bound}, it suffices to show there exists $\varepsilon>0$ such that $\P[\mathcal{A}]=1$ and for the given $\varepsilon>0$,
	\begin{align}\label{eq:sufficient_condition_cumulant}
		\sup_{s\in (-\varepsilon,\varepsilon)}\frac{1}{n}\sum_{i=1}^n \E_{in,s}[W_{in}^4|\mathcal{F}_n]=O_{\P}(1).
	\end{align}
	Suppose $\P[\mathcal{A}]=1$ and the assumption \eqref{eq:sufficient_condition_cumulant} holds. Now we verify condition \eqref{eq:second_cgf_derivative_bound}-\eqref{eq:fourth_cgf_derivative_bound} subsequently. 

	\paragraph{Verification of condition \eqref{eq:second_cgf_derivative_bound}:}

	By conclusion \eqref{eq:second_moment_relationship}, Jensen's inequality and statement \eqref{eq:sufficient_condition_cumulant}, we have
	\begin{align*}
		\sup_{s\in (-\varepsilon,\varepsilon)}\frac{1}{n}\sum_{i=1}^n (K_{in}''(s))^2
		&
		\leq \sup_{s\in (-\varepsilon,\varepsilon)}\frac{1}{n}\sum_{i=1}^n (\E_{in,s}[W_{in}^2|\mathcal{F}_n])^2\\
		&
		=\sup_{s\in (-\varepsilon,\varepsilon)}\frac{1}{n}\sum_{i=1}^n \left(\int x^2\mathrm{d}\kappa_{in,s}(\omega,x)\right)^2\\
		&
		\leq \sup_{s\in (-\varepsilon,\varepsilon)}\frac{1}{n}\sum_{i=1}^n \int x^4\mathrm{d}\kappa_{in,s}(\omega,x)\\
		&
		= \sup_{s\in (-\varepsilon,\varepsilon)}\frac{1}{n}\sum_{i=1}^n \E_{in,s}[W_{in}^4|\mathcal{F}_n]\\
		&
		=O_{\P}(1).
	\end{align*}

	\paragraph{Verification of condition \eqref{eq:third_cgf_derivative_bound}:}

	It suffices to prove 
	\begin{align*}
		\frac{1}{n}\sum_{i=1}^n K_{in}'''(0)=\frac{1}{n}\sum_{i=1}^n \E[W_{in}^3|\mathcal{F}_n]=O_{\P}(1).
	\end{align*}
	By Lemma \ref{lem:moment_dominance} with $p=3,q=4$, we can bound 
	\begin{align*}
		\left|\frac{1}{n}\sum_{i=1}^n \E[W_{in}^3|\mathcal{F}_n]\right|\leq \frac{1}{n}\sum_{i=1}^n \E[|W_{in}|^3|\mathcal{F}_n]\leq \left(\frac{\sum_{i=1}^n \E[W_{in}^4|\mathcal{F}_n]}{n}\right)^{3/4}.
	\end{align*}
	Then by statement \eqref{eq:sufficient_condition_cumulant}, we have 
	\begin{align*}
		\left|\frac{1}{n}\sum_{i=1}^n K_{in}'''(0)\right|\leq \left(\frac{\sum_{i=1}^n \E[W_{in}^4|\mathcal{F}_n]}{n}\right)^{3/4}=O_{\P}(1).
	\end{align*}

	\paragraph{Verification of condition \eqref{eq:fourth_cgf_derivative_bound}:}

	By conclusion \eqref{eq:fourth_moment_relationship} and \eqref{eq:second_moment_relationship},
	\begin{align*}
		&
		\sup_{s\in (-\varepsilon,\varepsilon)}\left|\frac{1}{n}\sum_{i=1}^n K_{in}''''(s)\right|\\
		&
		\leq \sup_{s\in (-\varepsilon,\varepsilon)}\frac{1}{n}\sum_{i=1}^n\E_{in,s}[(W_{in}-\E_{in,s}[W_{in}|\mathcal{F}_n])^4|\mathcal{F}_n]+\sup_{s\in(-\varepsilon,\varepsilon)}\frac{3}{n}\sum_{i=1}^n (K_{in}''(s))^2\\
		&
		=\sup_{s\in (-\varepsilon,\varepsilon)}\frac{1}{n}\sum_{i=1}^n \int \left(x-\int x\mathrm{d}\kappa_{in,s}(\cdot,x)\right)^4\mathrm{d}\kappa_{in,s}(\cdot,x)+O_{\P}(1)\\
		&
		\leq \sup_{s\in (-\varepsilon,\varepsilon)}\frac{16}{n}\sum_{i=1}^n \left(\int x^4\mathrm{d}\kappa_{in,s}(\cdot,x)+\left(\int x\mathrm{d}\kappa_{in,s}(\cdot,x)\right)^4\right)+O_{\P}(1)\\
		&
		\leq \sup_{s\in (-\varepsilon,\varepsilon)}\frac{32}{n}\sum_{i=1}^n \int x^4\mathrm{d}\kappa_{in,s}(\cdot,x)+O_{\P}(1)\\
		&
		=\sup_{s\in (-\varepsilon,\varepsilon)}\frac{32}{n}\sum_{i=1}^n \E_{in,s}[W_{in}^4|\mathcal{F}_n]+O_{\P}(1)\\
		&
		=O_{\P}(1)
	\end{align*}
	where the third inequality is due to power inequality $(|a|+|b|)^p\leq 2^p (|a|^p+|b|^p)$ and the proved condition \eqref{eq:second_cgf_derivative_bound}, the fourth inequality is due to Jensen's inequality and the last equality is due to statement \eqref{eq:sufficient_condition_cumulant}. Now we show there exists $\varepsilon>0$ such that statement \eqref{eq:sufficient_condition_cumulant} holds for both cases.

	\paragraph{Case 1: CSE distribution}

	Consider the power-series expansion 
	\begin{align}
		\E[\exp(sW_{in})|\mathcal{F}_n]
		&\nonumber
		=\int \exp(sx)\mathrm{d}\kappa_{in}(\omega,x)\\
		&\nonumber
		=\int\left(1+\sum_{k=1}^{\infty}\frac{s^k}{k!}x^k\right)\mathrm{d}\kappa_{in}(\omega,x) \\
		&\nonumber
		=1+\sum_{k=1}^{\infty}\frac{s^k}{k!}\int x^k\mathrm{d}\kappa_{in}(\omega,x)\\ 
		&\label{eq:power_series}
		=1+\sum_{k=2}^{\infty}\frac{s^k}{k!}\E[W_{in}^k|\mathcal{F}_n]
	\end{align}
	where the second last inequality is due to Fubini's theorem and the last inequality is due to the definition of conditional expectation \eqref{eq:def_conditional_expectation}. We first can bound using Lemma \ref{lem:equivalence_CSE} and Assumption \ref{assu:cse} that
	\begin{align}\label{eq:cgf_bound_cse}
		\P\left[\E[\exp(sW_{in})|\mathcal{F}_n]\leq \exp(\lambda_n s^2),\forall s\in (-\beta/4,\beta/4)\right]=1
	\end{align}
	where 
	\begin{align*}
		\lambda_n=\frac{\sqrt{6!4^6}(1+\theta_{n})}{24\beta^2}+\frac{16(1+\theta_{n})}{\beta^2}.
	\end{align*}
	Then we can bound by setting $s=\beta/16$ in the identity \eqref{eq:power_series} and conclusion \eqref{eq:cgf_bound_cse} so that 
	\begin{align}\label{eq:eighth_moment_bound}
		\E[W_{in}^{12}|\mathcal{F}_n]\leq \frac{12!16^{12}}{\beta^{12}}\E\left[\exp\left(\frac{\beta}{16}W_{in}\right)|\mathcal{F}_n\right]\leq \frac{12!16^{12}}{\beta^{12}}\exp\left(\frac{\lambda_n \beta^2}{256}\right)
	\end{align}
	almost surely. Then we can bound for $|s|<\beta/8$:
	\begin{align*}
		\E_{in,s}[W_{in}^4|\mathcal{F}_n]
		&
		=\frac{\E[W_{in}^4\exp(sW_{in})|\mathcal{F}_n]}{\E[\exp(sW_{in})|\mathcal{F}_n]}\\
		&
		=\frac{\int x^4 \exp(sx)\mathrm{d}\kappa_{in}(\cdot,x)}{\int \exp(sx)\mathrm{d}\kappa_{in}(\cdot,x)}\\
		&
		\leq \int x^4 \exp(sx)\mathrm{d}\kappa_{in}(\cdot,x)\\
		&
		\leq \left(\int x^{12} \mathrm{d}\kappa_{in}(\cdot,x)\right)^{1/3}\left(\int \exp\left(\frac{3sx}{2}\right)\mathrm{d}\kappa_{in}(\cdot,x)\right)^{2/3}\\
		&
		= \left(\E[W_{in}^{12}|\mathcal{F}_n]\right)^{1/3}\left(\E[\exp(3sW_{in}/2)|\mathcal{F}_n]\right)^{2/3}\\
		&
		\leq \left(\frac{12!16^{12}}{\beta^{12}}\right)^{1/3}\exp\left(\frac{\lambda_n\beta^2}{768}\right)\cdot \exp\left(\frac{3\lambda_n\beta^2}{128}\right)
	\end{align*}
	where the third and fourth inequality is due to Jensen's inequality and H\"older's inequality, respectively and the last inequality is due to bound \eqref{eq:eighth_moment_bound} and bound \eqref{eq:cgf_bound_cse}. By the assumption $\lambda_n=O_{\P}(1)$, we have
	\begin{align*}
		\sup_{s\in (-\varepsilon,\varepsilon)}\frac{1}{n}\sum_{i=1}^n \E_{in,s}[W_{in}^4|\mathcal{F}_n]\leq \left(\frac{12!16^{12}}{\beta^{12}}\right)^{1/3}\exp\left(\frac{\lambda_n\beta^2}{768}\right)\cdot \exp\left(\frac{3\lambda_n\beta^2}{128}\right)=O_{\P}(1).
	\end{align*}

	\paragraph{Case 2: CCS distribution}

	By Lemma \ref{lem:finite_cgf}, we know $\P[\mathcal{A}]=1$ with $\varepsilon=1$. Since $\P[\mathrm{Supp}(\kappa_{in}(\omega,\cdot))\in [-\nu_{in}(\omega),\nu_{in}(\omega)]]=1$, we can bound, 
	\begin{align*}
		\E_{in,s}\left[W_{in}^4|\mathcal{F}_n\right]
		&
		=\int x^4\frac{\exp(sx)}{\int \exp(sx)\mathrm{d}\kappa_{in}(\cdot,x)}\mathrm{d}\kappa_{in}(\cdot,x)\\
		&
		\leq \frac{\nu_{in}^4\int \exp(sx)\mathrm{d}\kappa_{in}(\cdot,x)}{\int \exp(sx)\mathrm{d}\kappa_{in}(\cdot,x)}= \nu_{in}^4,\ \forall s\in (-\varepsilon,\varepsilon)=(-1,1)
	\end{align*}
	almost surely. Thus we have 
	\begin{align*}
		\sup_{s\in (-\varepsilon,\varepsilon)}\frac{1}{n}\sum_{i=1}^n \E_{in,s}\left[W_{in}^4|\mathcal{F}_n\right]\leq \frac{1}{n}\sum_{i=1}^n \nu_{in}^4=O_{\P}(1).
	\end{align*}
\end{proof}

\subsection{Proof of Lemma \ref{lem:reduced_variance_condition}}

\begin{proof}[Proof of Lemma \ref{lem:reduced_variance_condition}]
	By conclusion \eqref{eq:second_moment_relationship} in Lemma \ref{lem:tilted_moment}, we know on the event $\mathcal{A}$,
	\begin{align*}
		\frac{1}{n}\sum_{i=1}^n K_{in}''(0)=\frac{1}{n}\sum_{i=1}^n \E[W_{in}^2|\mathcal{F}_n].
	\end{align*}
	By Lemma \ref{lem:finite_cgf}, we know $\P[\mathcal{A}]=1$ and together with the condition \eqref{eq:lower_bound_conditional_variance}, the claim is true.
\end{proof}

\subsection{Proof of Lemma \ref{lem:finite_cgf_moments}}\label{sec:proof_finite_cgf_moments}

\begin{proof}[Proof of Lemma \ref{lem:finite_cgf_moments}]
	Define 
	\begin{align*}
		A_{in,s}\equiv \E[|W_{in}|^p\exp(sW_{in})|\mathcal{F}_n].
	\end{align*}
	Fix any $s_0\in (-\varepsilon,\varepsilon)$ and suppose $a_0\in (1, \varepsilon/|s_0|)$. We have by H\"older's inequality
	\begin{align*}
		A_{in,s_0}\leq \left\{\E\left[|W_{in}|^{\frac{pa_0}{a_0-1}}|\mathcal{F}_n\right]\right\}^{(a_0-1)/a_0}\left\{\E\left[\exp\left(a_0s_0W_{in}\right)|\mathcal{F}_n\right]\right\}^{1/a_0}.
	\end{align*}
	We first show that on the event $\mathcal{A}$, all the conditional moments for $W_{in}|\mathcal{F}_n$ are finite almost surely.
	\begin{lemma}\label{lem:finite_moment}
		On the event $\mathcal{A}$, 
		\begin{align*}
			\E[|W_{in}|^m|\mathcal{F}_n]<\infty,\ \forall i\in \{1,\ldots,n\},\ \forall m\in\mathbb{N}.
		\end{align*}
	\end{lemma}
	\noindent By Lemma \ref{lem:finite_moment}, on the event $\mathcal{A}$, we have $\E\left[|W_{in}|^{\lceil pa_0 / (a_0-1)\rceil}|\mathcal{F}_n\right]<\infty$. We can show $a_0s_0<\varepsilon$ so that on the same event, $\E\left[\exp\left(a_0s_0W_{in}\right)|\mathcal{F}_n\right]<\infty$.
	Therefore we have proved on the event $\mathcal{A}$,
	\begin{align*}
		\E[|W_{in}|^p\exp(sW_{in})|\mathcal{F}_n]<\infty,\ \forall s\in (-\varepsilon,\varepsilon),\ \forall i\in\{1,\ldots,n\},\ \forall n,p\in \mathbb{N}.
	\end{align*}
\end{proof}

\subsection{Proof of Lemma \ref{lem:asym-estimate-lam-r}}

\begin{proof}[Proof of Lemma \ref{lem:asym-estimate-lam-r}]

	We first present several auxiliary results.
	\paragraph{Auxiliary results:}
	\begin{align}
		\P\left[\hat s_n w_n - K_n(\hat s_n)\leq  0\text{ and }\hat s_n\neq 0\right]\rightarrow 0;&\label{eq:sqrt_part_r_n}
		\\
		\P[\lambda_nr_n\leq  0\text{ and }\hat s_n\neq 0]\rightarrow0; & \label{eq:sign_property_r_n_lambda_n}
		\\
		r^2_n=2n( \hat s_n w_n - K_n(\hat s_n))=n\hat s_n^2 \left(K_n''(0)+\hat s_nO_{\P}(1)\right); & \label{eq:asym-estimate-r}\\
		\lambda_n^2=n\hat s_n^2K_n''(\hat s_n)
		=n\hat s_n^2 \left(K_n''(0)+\hat s_nO_{\P}(1)\right). &\label{eq:asym-estimate-lam}
	\end{align}

	\paragraph{Proofs of Auxiliary results:}

	\paragraph{Proof of \eqref{eq:sqrt_part_r_n}}

	Guaranteed by Lemma \ref{lem:finite_cgf}, we Taylor expand, for $s\in (-\varepsilon,\varepsilon)$,
	\begin{align}
	  K_n(s)
	  &\nonumber
	  =K_n(0)+ sK_n'(0)+\frac{1}{2}s^2
	  K_n''(0)+\frac{s^3}{6}K_n'''(0)+\frac{s^4}{24}K_n''''(\bar s)\\
	  &\label{eq:Taylor_expansion_K_n}
	  =\frac{1}{2}s^2
	  K_n''(0)+\frac{s^3}{6}K_n'''(0)+\frac{s^4}{24}K_n''''(\bar s(s)).
	\end{align}
	where $\bar s(s)\in (-\varepsilon,\varepsilon)$ and the last equality is due to $K_n(0)=0$ and $K_n'(0)=0$. Similarly, we obtain for $s\in(-\varepsilon,\varepsilon)$,
	\begin{align}
		sK_n'(s)
		&\nonumber
		=sK_n'(0)+K_n''(0)s^2+\frac{s^3}{2}K_n'''(0)+\frac{s^4}{6}K_n''''(\tilde s)\\
		&\label{eq:Taylor_expansion_s_K_n_p}
		=K_n''(0)s^2+\frac{s^3}{2}K_n'''(0)+\frac{s^4}{6}K_n''''(\tilde s(s))
	\end{align}
	where $\tilde s(s)\in [-s,s]\subset (-\varepsilon,\varepsilon)$. Then subtracting the expansion \eqref{eq:Taylor_expansion_s_K_n_p} from the expansion \eqref{eq:Taylor_expansion_K_n} and setting $s=\hat s_n$ since $\hat s_n\in [-\varepsilon/2,\varepsilon/2]$, we get 
	\begin{align}\label{eq:difference_Taylor_expansion}
		\hat s_n K_n'(\hat s_n)-K_n(\hat s_n)=\frac{\hat s_n^2}{2}K_n''(0)+\frac{\hat s_n^3}{3}K_n'''(0)+\frac{\hat s_n^4}{6}\left(K_n''''(\tilde s(\hat s_n))-\frac{1}{4}K_n''''(\bar s(\hat s_n))\right).
	\end{align}
	Notice \eqref{eq:difference_Taylor_expansion} is similar to our target but still differs. To account such difference, we consider
	\begin{align}
		&\nonumber
		\hat s_n w_n -K_n(\hat s_n)\\
		&\nonumber
		=(\hat s_n K_n'(\hat s_n)-K_n(\hat s_n))\indicator(K_n'(\hat s_n)=w_n)+(\hat s_n w_n-K_n(\hat s_n))\indicator(K_n'(\hat s_n)\neq w_n)\\
		&\nonumber
		=\frac{\hat s_n^2}{2}K_n''(0)-\frac{\hat s_n^2}{2}K_n''(0)\indicator(K_n'(\hat s_n)\neq w_n)+\frac{\hat s_n^3}{3}K_n'''(0)\indicator(K_n'(\hat s_n)=w_n)\\
		&\nonumber
		\ +\frac{\hat s_n^4}{6}(K_n''''(\tilde s(\hat s_n))-K_n''''(\bar s(\hat s_n))/4)\indicator(K_n'(\hat s_n)=w_n) +(\hat s_n w_n-K_n(\hat s_n))\indicator(K_n'(\hat s_n)\neq w_n)\\
		&\label{eq:rate_r_n}
		\equiv \frac{\hat s_n^2}{2}K_n''(0)+\hat s_n^3M_n
	\end{align}
	where $M_n$ is a random variable that is $O_{\P}(1)$. This is true because the following claims are true:
	\begin{align}
		\frac{\hat s_n^3}{3}K_n'''(0)\indicator(K_n'(\hat s_n)=w_n)
		&\label{eq:sn_power_3}
		=\hat s_n^3O_{\P}(1)\\
		\frac{\hat s_n^4}{6}(K_n''''(\tilde s(\hat s_n))-K_n''''(\bar s(\hat s_n))/4)\indicator(K_n'(\hat s_n)=w_n)
		&\label{eq:higher_order_estiamte_1}
		=\hat s_n^3O_{\P}(1)\\
		\frac{\hat s_n^2}{2}K_n''(0)\indicator(K_n'(\hat s_n)\neq w_n)
		&\label{eq:higher_order_estiamte_2}
		=\hat s_n^3O_{\P}(1)\\
		(\hat s_n w_n-K_n(\hat s_n))\indicator(K_n'(\hat s_n)\neq w_n)
		&\label{eq:higher_order_estiamte_3}
		=\hat s_n^3O_{\P}(1).
	\end{align}
	Now we prove the claims \eqref{eq:sn_power_3}-\eqref{eq:higher_order_estiamte_3}. For claim \eqref{eq:sn_power_3}, by condition \eqref{eq:third_cgf_derivative_bound}, we know it is true. For claim \eqref{eq:higher_order_estiamte_1}, from condition \eqref{eq:fourth_cgf_derivative_bound} and $\tilde s(\hat s_n),\bar s(\hat  s_n)\in (-\varepsilon,\varepsilon)$, we have
	\begin{align*}
		\left|K_n''''(\tilde s(\hat s_n))\right|=O_{\P}(1), \left|K_n''''(\bar s(\hat s_n))\right|=O_{\P}(1).
	\end{align*}
	Then together with $\hat s_n=o_{\P}(1)$
	For claim \eqref{eq:higher_order_estiamte_2}, we know it is true since $\indicator(K_n'(\hat s_n)\neq w_n)=o_{\P}(1)$. Similar argument applies to \eqref{eq:higher_order_estiamte_3}. Now define the event
	\begin{align}
		\mathcal{P}_n\equiv \{\hat s_n \neq 0 \}\cap \left\{ K_n''(0)\leq  -2\hat s_nM_n \right\}.
	\end{align}
	By condition \eqref{eq:lower_bound_variance}, we know there exists $\eta>0$ such that $\P[K_n''(0)>\eta]\rightarrow1$. For such $\eta$ since $M_n=O_{\P}(1)$ and $\hat s_n=o_{\P}(1)$ by Lemma \ref{lem:saddlepoint_properties}, we have $\P[-2\hat s_nM_n<\eta]\rightarrow 1$. Together we conclude $\P[K_n''(0)\leq  -2\hat s_nM_n ]\rightarrow0$. This implies $\P[\mathcal{P}_n]\rightarrow0$. Moreover, on the event $\mathcal{P}_n$ we know $\hat s_n w_n -K_n(\hat s_n)\leq 0$ and $\hat s_n\neq 0$ happen. Therefore we conclude the proof.

	\paragraph{Proof of \eqref{eq:sign_property_r_n_lambda_n}}

	We compute
	\begin{align*}
		r_n\lambda_n=
		\begin{cases}
			|\hat s_n|\sqrt{n K_n''(\hat s_n)}\sqrt{2n(\hat s_n w_n-K_n(\hat s_n))} & \text{ if }\hat s_n w_n-K_n(\hat s_n)\geq 0\\
			|\hat s_n|\sqrt{n K_n''(\hat s_n)} & \text{ otherwise.}
		\end{cases}
	\end{align*}
	Lemma \ref{lem:conditional_clt_assumptions} implies that $K_n''(\hat s_n)=\Omega_\P(1)$. By \eqref{eq:sqrt_part_r_n}, we know $\P[\hat s_n w_n-K_n(\hat s_n)\leq 0\text{ and }\hat s_n\neq 0]\rightarrow0$. Moreover, $K_n''(\hat s_n)=\Omega_{\P}(1)$ can further imply $\P[K_n''(\hat s_n)=0]\rightarrow0$. Collecting all these, we reach
	\begin{align*}
		\P[r_n\lambda_n\leq 0\text{ and }\hat s_n\neq 0]
		&
		=\P[r_n\lambda_n= 0\text{ and }\hat s_n\neq 0]\\
		&
		= \P[\{K_n''(\hat s_n)=0\text{ or }\hat s_nw_n-K_n(\hat s_n)=0\}\text{ and }\{\hat s_n \neq 0\}]\\
		&
		\leq \P[K_n''(\hat s_n)=0 \text{ and }\hat s_n\neq 0]+\P[\hat s_n w_n-K_n(\hat s_n)=0\text{ and }\hat s_n\neq 0]\\
		&
		\rightarrow0.
	\end{align*}

	\paragraph{Proof of \eqref{eq:asym-estimate-r}}

	For $r^2_n$, we can write, according to \eqref{eq:rate_r_n},
	\begin{align}\label{eq:r_n_formula}
		r_n^2=2n(\hat s_n w_n-K_n(\hat s_n))=2n\left(\frac{1}{2}\hat s_n^2K_n''(0)+\hat s_n^3M_n\right)=n\hat s_n^2K_n''(0)+n\hat s_n^3O_{\P}(1).
	\end{align}

	\paragraph{Proof of \eqref{eq:asym-estimate-lam}}

	We expand $K_n''(s)$ in the neighborhood $(-\varepsilon,\varepsilon)$:
	\begin{align*}
	  K_n''(s)=K_n''(0)+sK_n'''(0)+\frac{s^2}{2}K_n''''(\dot s(s)),\ \dot s(s)\in (-\varepsilon,\varepsilon).
	\end{align*}
	Then plugging $\hat s_n$ into above formula and observing $K_n''''(\dot s(\hat s_n))=O_{\P}(1)$ ensured by condition \eqref{eq:fourth_cgf_derivative_bound} and $\dot s(\hat s_n)\in (-\varepsilon,\varepsilon)$, we obtain 
	\begin{align}
		\lambda_n^2=n\hat s_n^2 K_n''(\hat s_n)
		&\nonumber
		=n\hat s_n^2 K_n''(0)+n\hat s_n^3 K_n'''(0)+n\frac{\hat s_n^4}{2}K_n''''(\dot s(\hat s_n))\\
		&\label{eq:lam_n_formula}
		=n\hat s_n^2 K_n''(0)+n\hat s_n^3 O_{\P}(1).
	\end{align}

	\paragraph{Proof of main results in Lemma \ref{lem:asym-estimate-lam-r}:}

	Now we come to prove the main results in Lemma \ref{lem:asym-estimate-lam-r} using the auxiliary results proved above.

	\paragraph{Proof of \eqref{eq:r_n_over_n_rate}}

	This can be directly obtained by \eqref{eq:asym-estimate-r} that
	\begin{align*}
		\frac{r_n^2}{n}=\hat s_n^2K_n''(0)+\hat s_n^3O_{\P}(1)=o_{\P}(1)
	\end{align*}
	since $\hat s_n=o_{\P}(1)$ and by condition \eqref{eq:second_cgf_derivative_bound} and Cauchy-Schwarz inequality,
	\begin{align*}
		K_n''(0)=\frac{1}{n}\sum_{i=1}^n K_{in}''(0)\leq \left(\frac{1}{n}\sum_{i=1}^n (K_{in}''(0))^2\right)^{1/2}=O_{\P}(1).
	\end{align*}
	
	\paragraph{Proof of \eqref{eq:asym-estimate-ratio-lam-r}}

	Since $\lambda_n,r_n\in (-\infty,\infty)$, we need to divide the proof into several cases. When $\hat s_n=0$, we know $\lambda_n=r_n=0$ so that $\lambda_n/r_n=1$ by convention. Now we consider when $\hat s_n\neq 0$.
	\begin{itemize}
		\item \textbf{When $\lambda_n r_n \leq 0$:} observe that
		\begin{align*}
			\P\left[\frac{\indicator(\hat s_n\neq 0 \text{ and }\lambda_nr_n\leq 0)}{|\hat s_n|}\left|\frac{\lambda_n}{r_n}-1\right|>\delta\right]
			&
			\leq \P[\lambda_nr_n\leq 0\text{ and }\hat s_n\neq 0]\\
			&
			\rightarrow0
		\end{align*}
		so that
		\begin{align*}
			\indicator(\hat s_n\neq 0 \text{ and }\lambda_nr_n\leq 0)\left|\frac{\lambda_n}{r_n}-1\right|=\hat s_n O_{\P}(1).
		\end{align*}
		\item \textbf{When $\lambda_n r_n > 0$:}
		It requires to compute $\lambda_n^2/r_n^2$. By \eqref{eq:r_n_formula} and \eqref{eq:lam_n_formula}, we get
		\begin{align*}
			\frac{\lambda_n^2}{r_n^2}
			&
			=\frac{\hat s_n^2 K_n''(0)+\hat s_n^3 K_n'''(0)+\frac{\hat s_n^4}{2}K_n''''(\dot s(\hat s_n))}{\hat s_n^2K_n''(0)+\hat s_n^3M_n}\\
			&
			=1+\hat s_n\frac{K_n'''(0)-\hat s_nM_n+\frac{\hat s_n}{2}K_n''''(\dot s(\hat s_n))}{K_n''(0)+\hat s_nM_n}\\
			&
			\equiv 1+\hat s_n \cdot F_n.
		\end{align*}
		Thus we know 
		\begin{align*}
			\indicator(r_n\lambda_n>0 \text{ and }\hat s_n\neq 0)\left(\frac{\lambda_n^2}{r_n^2}-1\right)= \indicator(r_n\lambda_n>0\text{ and }\hat s_n\neq 0)\hat s_n \cdot F_n.
		\end{align*}
		To further proceed the proof, we observe
		\begin{align*}
			F_n=O_{\P}(1)
		\end{align*}
		since $\hat s_n=o_{\P}(1),M_n=O_{\P}(1)$ and conditions \eqref{eq:lower_bound_variance}, \eqref{eq:third_cgf_derivative_bound} and \eqref{eq:fourth_cgf_derivative_bound} guarantee respectively $K_n''(0)=\Omega_\P(1),K_n'''(0)=O_{\P}(1)$ and $K_n''''(\dot s(\hat s_n))=O_{\P}(1)$. Thus 
		\begin{align*}
			\indicator(r_n\lambda_n>0 \text{ and }\hat s_n\neq 0)\left|\frac{\lambda_n}{r_n}-1\right|
			&
			\leq \indicator(r_n\lambda_n>0 \text{ and }\hat s_n\neq 0)\left|\frac{\lambda_n}{r_n}-1\right|\left|\frac{\lambda_n}{r_n}+1\right|\\
			&
			=\indicator(r_n\lambda_n>0 \text{ and }\hat s_n\neq 0)\left|\frac{\lambda_n^2}{r_n^2}-1\right|\\
			&
			\leq \indicator(r_n\lambda_n>0 \text{ and }\hat s_n\neq 0)|\hat s_n F_n|\\
			&
			=\hat s_n O_{\P}(1)
		\end{align*}
	\end{itemize}
	Collecting all the results, we have 
	\begin{align*}
		\left|\frac{\lambda_n}{r_n}-1\right|=\hat s_nO_{\P}(1).
	\end{align*}

	\paragraph{Proof of \eqref{eq:asym-estimate-ratio-r-lam}}

	The proof is similar to \eqref{eq:asym-estimate-ratio-lam-r} so we omit the proof.

	\paragraph{Proof of \eqref{eq:asym-estimate-diff-lam-r}}

	When $\hat s_n=0$, we know $\lambda_n=r_n=0$ so that $1/\lambda_n-1/r_n=0$ by the convetion $1/0-1/0=1$. Now we consider the case when $\hat s_n\neq 0$. We divide the proof into serveral cases. By result \eqref{eq:asym-estimate-ratio-lam-r} and claim \eqref{eq:lam_n_formula},
	\begin{align*}
		\left(\frac{1}{r_n}-\frac{1}{\lambda_n}\right)^2=\frac{1}{\lambda_n^2}\left(\frac{\lambda_n}{r_n}-1\right)^2=\frac{\hat s_n^2  O_{\P}(1)}{n\hat s_n^2 K_n''(0)+n\hat s_n^3O_{\P}(1)}=\frac{O_{\P}(1)}{nK_n''(0)+n\hat s_nO_{\P}(1)}=o_{\P}(1)
	\end{align*}
	where the last equality is due to $\hat s_n=o_{\P}(1)$ and $K_n''(0)=\Omega_{\P}(1)$. 

	\paragraph{Proof of \eqref{eq:asym-estimate-diff-lam-r-multiplication}}

	On the event $r_n>0,\lambda_n>0$, we know $\hat s_n>0$. Then by \eqref{eq:asym-estimate-ratio-lam-r} and \eqref{eq:r_n_formula} we can compute 
	\begin{align*}
		\indicator(r_n>0\text{ and }\lambda_n>0)\frac{1}{r_n^2}\left(\frac{\lambda_n}{r_n}-1\right)^2
		&
		=\frac{\indicator(r_n>0\text{ and }\lambda_n>0)\hat s_n^2O_{\P}(1)}{n\hat s_n^2K_n''(0)+n\hat s_n^3O_{\P}(1)}\\
		&
		=\frac{\indicator(r_n>0\text{ and }\lambda_n>0)O_{\P}(1)}{nK_n''(0)+n\hat s_n O_{\P}(1)}.
	\end{align*}
	Then since $\hat s_n=o_{\P}(1)$ and $K_n''(0)=\Omega_{\P}(1)$, we know 
	\begin{align*}
		\indicator(r_n>0\text{ and }\lambda_n>0)\frac{1}{r_n}\left(\frac{\lambda_n}{r_n}-1\right)=o_{\P}(1).
	\end{align*} 

	\paragraph{Proof of \eqref{eq:asym-estimate-diff-r-lam-multiplication}}

	The proof is similar to the proof of \eqref{eq:asym-estimate-diff-lam-r-multiplication} so we omit it.

	\paragraph{Proof of \eqref{eq:same_sign_condition_w_n}}

	Since $\P[\hat s_n>0 \text{ and }\lambda_n r_n\leq 0]\rightarrow0$ by \eqref{eq:sign_property_r_n_lambda_n} and $\mathrm{sgn}(w_n)=\mathrm{sgn}(\hat s_n)$, we have 
	\begin{align*}
		\P[w_n>0 \text{ and }\lambda_n r_n\leq 0]=\P[w_n>0 \text{ and } \hat s_n>0 \text{ and }\lambda_n r_n\leq 0]\rightarrow0
	\end{align*}

	\paragraph{Proof of \eqref{eq:same_sign_condition_s_n}}

	Since $\P[\hat s_n\neq 0 \text{ and }\lambda_n r_n \leq  0]\rightarrow0$ by \eqref{eq:sign_property_r_n_lambda_n}, we have 
	\begin{align*}
		\P[\hat s_n\neq 0 \text{ and }\lambda_n r_n= 0]\leq \P[\hat s_n\neq 0 \text{ and } \lambda_n r_n\leq 0]\rightarrow0.
	\end{align*}

  \end{proof}

\subsection{Proof of Lemma \ref{lem:upper_bound_ratio_spa}}

\begin{proof}[Proof of Lemma \ref{lem:upper_bound_ratio_spa}]
	We prove the statements \eqref{eq:multiplication-flip-sign}-\eqref{eq:equality-corner-case} in order.
	
	\paragraph{Proof of \eqref{eq:multiplication-flip-sign}:}

	It suffices to prove that for any $\kappa>0$, 
	\begin{align}
		&\nonumber
		\P\left[\indicator(w_n<0)\frac{\Phi(r_n)+\phi(r_n)\left\{\frac{1}{r_n}-\frac{1}{\lambda_n}\right\}}{1 - \Phi(r_n)+\phi(r_n)\left\{\frac{1}{\lambda_n}-\frac{1}{r_n}\right\}}\in [0,1 +\kappa)\right]\\
		&\nonumber
		\equiv \P\left[\indicator(w_n<0)A(r_n,\lambda_n)\in [0,1 +\kappa)\right]\\
		&\label{eq:A_rn_lamn_negative_xn}
		\rightarrow 1.
	\end{align}
	\noindent We decompose 
	\begin{align*}
		\indicator(w_n<0)A(r_n,\lambda_n)=\indicator(w_n<0,r_n>0)A(r_n,\lambda_n)+\indicator(w_n<0,r_n \leq 0)A(r_n,\lambda_n).
	\end{align*}
	By condition \eqref{eq:sign_1}, we know $\indicator(w_n<0,r_n > 0)=0$. Moreover, by the statement of \eqref{eq:U_n_r_n} in Lemma \ref{lem:relative_error_Berry_Esseen_bound}, we know
	\begin{align*}
		\P\left[1-\Phi(r_n)+\phi(r_n)\left\{\frac{1}{\lambda_n}-\frac{1}{r_n}\right\}=0\right]\rightarrow0.
	\end{align*}
	Therefore, for any $\delta>0$, 
	\begin{align*}
		\P\left[\left|\indicator(w_n<0,r_n>0)A(r_n,\lambda_n)\right|>\delta\right]\leq \P\left[1-\Phi(r_n)+\phi(r_n)\left\{\frac{1}{\lambda_n}-\frac{1}{r_n}\right\}=0\right]\rightarrow0.
	\end{align*}
	Thus we know $\indicator(w_n<0,r_n>0)A(r_n,\lambda_n)=o_{\P}(1)$. Then we only need to consider behavior of $\indicator(w_n<0,r_n \leq 0)A(r_n,\lambda_n)$. Then we know $1-\Phi(r_n)\geq 1/2$ when $r_n \leq 0$. Then by condition \eqref{eq:rate_1}, we know
	\begin{align*}
		|M_n|\equiv \left|\phi(r_n)\left\{\frac{1}{\lambda_n}-\frac{1}{r_n}\right\}\right|\leq\left|\frac{1}{\lambda_n}-\frac{1}{r_n}\right| =o_\P(1).
	\end{align*}
	Fix $\eta>0,\delta\in (0,0.1)$. Then for large enough $n,\P[|M_n|<\delta]\geq 1-\eta$. Then on the event $|M_n|<\delta$, we have 
	\begin{align}
		0<\frac{\delta}{1-\delta}
		=\frac{1}{1-\delta}-1 <A(r_n,\lambda_n)
		&\label{eq:A_r_lam_decomposition}
		=\frac{1}{1 - \Phi(r_n)+\phi(r_n)\left\{\frac{1}{\lambda_n}-\frac{1}{r_n}\right\}}-1\\
		&\nonumber
		\leq \frac{1}{\frac{1}{2}-\delta}-1=1+\frac{2\delta}{1-2\delta}<1+4\delta.
	\end{align}
	Thus we know 
	\begin{align*}
		\liminf_{n\rightarrow\infty}\P\left[\indicator(w_n<0,r_n \leq 0)A(r_n,\lambda_n)\in [0,1+4\delta)\right]\geq \liminf_{n\rightarrow\infty}\P[|M_n|<\delta]>1-\eta.
	\end{align*}
	Then by the arbitrary choice of $\eta$, we have 
	\begin{align}
		&\nonumber
		\lim_{n\rightarrow\infty}\P\left[\indicator(w_n < 0)A(r_n,\lambda_n)\in [0,1+4\delta)\right]\\
		&\label{eq:A_r_n_lam_n_convergence}
		=\lim_{n\rightarrow\infty}\P\left[\indicator(w_n<0\text{ and }r_n \leq 0)A(r_n,\lambda_n)\in [0, 1+4\delta)\right]=1.
	\end{align}
	Thus we complete the proof for claim \eqref{eq:A_rn_lamn_negative_xn} by choosing $\kappa=4\delta$.

	\paragraph{Proof of \eqref{eq:equality-corner-case}:}
	We have 
	\begin{align*}
		&
		\indicator(w_n < 0)\frac{\P\left[\frac1n \sum_{i = 1}^n W_{in} = w_n \mid \mathcal F_n\right]}{1-\Phi(r_n)+\phi(r_n)\{\frac{1}{\lambda_n}-\frac{1}{r_n}\}}\\
		&
		= \P\left[\frac1n \sum_{i = 1}^n W_{in} = w_n \mid \mathcal F_n\right]\cdot \frac{\indicator(w_n<0)}{1 - \Phi(r_n)+\phi(r_n)\left\{\frac{1}{\lambda_n}-\frac{1}{r_n}\right\}}\\
		&
		=\P\left[\frac1n \sum_{i = 1}^n W_{in} = w_n \mid \mathcal F_n\right]\cdot \left(\indicator(w_n<0)A(r_n,\lambda_n)+\indicator(w_n<0)\right)\\
		&
		=\P\left[\frac1n \sum_{i = 1}^n W_{in} = w_n \mid \mathcal F_n\right]\cdot O_{\P}(1)
	\end{align*}
	where the second equality is due to the decomposition of $A(r_n,\lambda_n)$ in \eqref{eq:A_r_lam_decomposition} and the last equality is due to result \eqref{eq:A_r_n_lam_n_convergence} that $\indicator(w_n<0)A(r_n,\lambda_n)=O_\P(1)$. Now applying claim \eqref{eq:nondegeneracy} with $y_n=w_n \sqrt{n/K_n''(0)}$, we know 
	\begin{align*}
		\P\left[\frac1n \sum_{i = 1}^n W_{in} = w_n \mid \mathcal F_n\right]=o_\P(1).
	\end{align*}
	Therefore we conclude 
	\begin{align*}
		\indicator(w_n < 0)\frac{\P\left[\frac1n \sum_{i = 1}^n W_{in} = w_n \mid \mathcal F_n\right]}{1-\Phi(r_n)+\phi(r_n)\{\frac{1}{\lambda_n}-\frac{1}{r_n}\}}=o_\P(1).
	\end{align*}
\end{proof}

\subsection{Proof of Lemma \ref{lem:moment_dominance}}\label{sec:proof_moment_dominance}

By H\"older's inequality, we have
\begin{align*}
	\frac{\sum_{i=1}^n \E[|W_{in}|^{p}|\mathcal{F}_n]}{n}
	&
	\leq 
	\frac{1}{n}\left(\sum_{i=1}^n \left(\E[|W_{in}|^{p}|\mathcal{F}_n]\right)^{q/p}\right)^{p/q}n^{1-p/q}\\
	&
	=\left(\frac{\sum_{i=1}^n \left(\E[|W_{in}|^{p}|\mathcal{F}_n]\right)^{q/p}}{n}\right)^{p/q}.
\end{align*}
We invoke Jensen's inequality for conditional expectation to conclude the proof.
\begin{lemma}[Conditional Jensen inequality, \cite{Davidson2003}, Theorem 10.18] \label{lem:conditional-jensen}
	Let $W$ be a random variable and let $\phi$ be a convex function, such that $W$ and $\phi(W)$ are integrable. For any $\sigma$-algebra $\mathcal F$, we have the inequality
	\begin{equation*}
		\phi(\E[W \mid \mathcal F]) \leq  \E[\phi(W) \mid \mathcal F].
	\end{equation*}
\end{lemma}
\noindent Applying Lemma \ref{lem:conditional-jensen} with $\phi(x)=x^{q/p}$, we obtain
\begin{align*}
	\frac{\sum_{i=1}^n \E[|W_{in}|^{p}|\mathcal{F}_n]}{n}\leq \left(\frac{\sum_{i=1}^n \left(\E[|W_{in}|^{p}|\mathcal{F}_n]\right)^{q/p}}{n}\right)^{p/q}\leq \left(\frac{\sum_{i=1}^n \E[|W_{in}|^{q}|\mathcal{F}_n]}{n}\right)^{p/q}.
\end{align*}

\subsection{Proof of Lemma \ref{lem:R_n_formula}}

\begin{proof}[Proof of Lemma \ref{lem:R_n_formula}]
	We apply integration by parts to the following integral on the event $\lambda_n,r_n\in (-\infty,\infty)$,
	\begin{align*}
		&
		\int_{0}^{\infty}\exp(-\lambda_n y)\phi(y)\mathrm{d}y\\
		&
		=\exp(\lambda_n^2/2)(1-\Phi(\lambda_n))\\
		&
		=\exp\left(\frac{r^2_n}{2}\right)(1-\Phi(r_n))
		+\int_{r_n}^{\lambda_n}
		y\exp(y^2/2)(1-\Phi(y))-\frac{1}{\sqrt{2\pi}}
		\mathrm{d}y\\
		&
		=\exp\left(\frac{r^2_n}{2}\right)(1-\Phi(r_n))
		-\int_{r_n}^{\lambda_n}\frac{1}{\sqrt{2\pi}
		y^2}\mathrm{d}y+R_n\\
		&
		=\exp\left(\frac{r^2_n}{2}\right)(1-\Phi(r_n))
		+\frac{1}{\sqrt{2\pi}}\left(\frac{1}{\lambda_n}
		-\frac{1}{r_n}\right)+R_n.
	  \end{align*}
	  This completes the proof.
\end{proof}

\subsection{Proof of Lemma \ref{lem:convergence_rate_denominator_relative_error}}

\begin{proof}[Proof of Lemma \ref{lem:convergence_rate_denominator_relative_error}]
	We prove the two claims separately. 

	\paragraph{Proof of \eqref{eq:relative_error_denominator_1}:}
	We can write \eqref{eq:relative_error_denominator_1} as 
	\begin{align*}
		\frac{\indicator(r_n \geq 1)}{r_nU_n}=\frac{\indicator(r_n\geq 1)}{\indicator(r_n\geq 1)r_n\exp(\frac{1}{2}r_n^2)(1-\Phi(r_n))+\frac{1}{\sqrt{2\pi}}\{\frac{r_n}{\lambda_n}-1\}}.
	\end{align*}
	Notice by \eqref{eq:rate_2},
	\begin{align*}
		\frac{1}{\sqrt{2\pi}}\left|\frac{r_n}{\lambda_n}-1\right|=o_{\P}(1).
	\end{align*}
	Then it suffices to prove there exists a universal constant $C>0$ such that
	\begin{align*}
		\indicator(r_n\geq 1)r_n\exp\left(\frac{r_n^2}{2}\right)(1-\Phi(r_n))\geq C\indicator(r_n\geq 1).
	\end{align*}
	To prove this, we apply Lemma \ref{lem:lower_bound_Gaussian} such that 
	\begin{align*}
		r_n\exp\left(\frac{r_n^2}{2}\right)(1-\Phi(r_n))\geq \frac{1}{\sqrt{2\pi}}\frac{r_n^2}{r_n^2+1}\geq \frac{1}{2},\ \text{ when }r_n\geq 1.
	\end{align*}
	Thus we have 
	\begin{align*}
		\indicator(r_n\geq 1) r_n\exp\left(\frac{r_n^2}{2}\right)(1-\Phi(r_n))\geq \frac{\indicator(r_n\geq 1)}{2}.
	\end{align*}
	Therefore we have proved \eqref{eq:relative_error_denominator_1}.

	\paragraph{Proof of \eqref{eq:relative_error_denominator_2}:}

	Similarly, we can write \eqref{eq:relative_error_denominator_2} as 
	\begin{align*}
		\indicator(r_n\in [0,1))\frac{1}{U_n}=\frac{\indicator(r_n\in [0,1))}{\indicator(r_n\in [0,1))\exp(\frac{1}{2}r_n^2)(1-\Phi(r_n))+\frac{1}{\sqrt{2\pi}}\{\frac{1}{\lambda_n}-\frac{1}{r_n}\}}.
	\end{align*}
	By \eqref{eq:rate_1},
	\begin{align*}
		\frac{1}{\sqrt{2\pi}}\left|\frac{1}{\lambda_n}-\frac{1}{r_n}\right|=o_{\P}(1).
	\end{align*}
	We only need to prove there exists a universal constant $C\geq 0$ such that
	\begin{align*}
		\indicator(r_n\in [0,1))\exp\left(\frac{r_n^2}{2}\right)(1-\Phi(r_n))\geq C\indicator(r_n\in [0,1)).
	\end{align*}
	Indeed, we can set $C$ to be 
	\begin{align*}
		\inf_{z\in [0,1]}\exp\left(\frac{z^2}{2}\right)(1-\Phi(z)).
	\end{align*}
	Therefore we proved claim \eqref{eq:relative_error_denominator_2}.	
\end{proof}

	\subsection{Proof of Lemma \ref{lem:upper_bound_Gaussian_integral}}

	\begin{proof}[Proof of Lemma \ref{lem:upper_bound_Gaussian_integral}]
		Using Lemma \ref{lem:R_n_formula}, we obtain for any $\lambda_n,r_n>0$,
		\begin{align}\label{eq:transform_Gaussian_integral}
			\int_{0}^{\infty}\exp(-\lambda_n y)\phi(y)\mathrm{d}y =R_n+U_n
		\end{align}
		almost surely. By statement~\eqref{eq:U_n_r_n} of Lemma \ref{lem:relative_error_Berry_Esseen_bound}, we know $\indicator(r_n>0)/U_n\in (-\infty,\infty)$ almost surely and thus $\indicator(r_n>0,\lambda_n>0)/U_n\in (-\infty,\infty)$ almost surely. Then multiplying both sides in \eqref{eq:transform_Gaussian_integral} with $\indicator(r_n>0,\lambda_n>0)/U_n$, we obtain
		\begin{align*}
			\indicator(r_n>0,\lambda_n>0)\frac{\int_{0}^{\infty}\exp(-\lambda_n y)\phi(y)\mathrm{d}y}{U_n}=
			\indicator(r_n>0,\lambda_n>0)(1+R_n/U_n)
		\end{align*}
		almost surely. This implies 
		\begin{align*}
			\indicator(r_n>0,\lambda_n>0)\left|\frac{\int_{0}^{\infty}\exp(-\lambda_n y)\phi(y)\mathrm{d}y}{U_n}-1\right|
			=\indicator(r_n>0,\lambda_n>0)\left|\frac{1}{U_n}\right|\cdot |R_n|
		\end{align*}
		almost surely. Thus it suffcies to bound $\indicator(r_n>0,\lambda_n>0)|R_n/U_n|$.
		Define 
		\begin{align*}
			R_{\min}\equiv \min\{r_n,\lambda_n\},\ R_{\max}\equiv \max\{r_n,\lambda_n\}.
		\end{align*}
	Therefore the absolute value of $R_n$ can be bounded as, using Lemma \ref{lem:Gaussian_tail_estimate},
	\begin{align*}
		&
		\indicator(r_n>0,\lambda_n>0)|R_n|\\
		&
		=\indicator(\lambda_n>0,r_n>0)|R_n|\\
		&
		\leq\indicator(\lambda_n>0,r_n>0)|r_n-\lambda_n|\sup_{y\in [R_{\min},R_{\max}]}\left|y\exp(y^2/2)(1-\Phi(y))-\frac{1-y^{-2}}{\sqrt{2\pi}}\right|\\
		&
		\leq \indicator(\lambda_n>0,r_n>0)|r_n-\lambda_n|\frac{1}{\sqrt{2\pi}}\sup_{y\in [R_{\min},R_{\max}]}\frac{1}{y^2}\\
		&
		\leq \frac{\indicator(\lambda_n>0,r_n>0)|r_n-\lambda_n|}{\sqrt{2\pi}}\left(\frac{1}{r_n^2}+\frac{1}{\lambda_n^2}\right).
	\end{align*}
	Then we have 
	\begin{align*}
		\indicator(r_n>0,\lambda_n>0)|R_n|\leq \frac{\indicator(r_n>0,\lambda_n>0)|r_n-\lambda_n|}{\sqrt{2\pi}}\left(\frac{1}{r_n^2}+\frac{1}{\lambda_n^2}\right).
	\end{align*}
	Then this implies
	\begin{align*}
		\indicator(r_n>0,\lambda_n>0)\left|\frac{R_n}{U_n}\right|\leq \left|\frac{1}{U_n}\right|\cdot \frac{\indicator(r_n>0,\lambda_n>0)|r_n-\lambda_n|}{\sqrt{2\pi}}\left(\frac{1}{r_n^2}+\frac{1}{\lambda_n^2}\right)
	\end{align*}
	almost surely. Therefore we complete the proof.
	\end{proof}

\subsection{Proof of Lemma \ref{lem:tilted_moment}}

\begin{proof}[Proof of Lemma \ref{lem:tilted_moment}]
	Then by Lemma \ref{lem:existence_derivative_CGF} and Lemma \ref{lem:finite_cgf_moments}, we have,
	\begin{align*}
		\P\left[\mathcal{T}\right]=1,\ \mathcal{T}\equiv \left\{K_{in}'(s)=\frac{\E[W_{in}\exp(sW_{in})|\mathcal{F}_n]}{\E[\exp(sW_{in})|\mathcal{F}_n]},\ \forall s\in (-\varepsilon,\varepsilon)\right\}.
	\end{align*}
	Then we know,
	\begin{align*}
		K_{in}'(s)(\omega)=\frac{\E[W_{in}\exp(sW_{in})|\mathcal{F}_n]}{\E[\exp(sW_{in})|\mathcal{F}_n]}(\omega)=\E_{n,s}[W_{in}|\mathcal{F}_n](\omega),\ \forall \omega\in\mathcal{T}.
	\end{align*}
	so that $\P\left[K_{in}'(s)=\E_{n,s}[W_{in}|\mathcal{F}_n],\ \forall s\in (-\varepsilon,\varepsilon)\right]=1$. The other two claims follow similarly.
\end{proof}

	\subsection{Proof of Lemma \ref{lem:finite_moment}}

	\begin{proof}[Proof of Lemma \ref{lem:finite_moment}]
		We use the definition of conditional expectation \eqref{eq:def_conditional_expectation} to prove the claim:
		\begin{align*}
			\sum_{m=0}^{\infty}\frac{\E[|W_{in}|^m|\mathcal{F}_n]}{m!}|s|^m
			&
			=\sum_{m=0}^{\infty}\frac{s^m}{m!}\int x^m \mathrm{d}\kappa_{in}(\cdot,x)\\
			&
			=\int \sum_{m=0}^{\infty}\frac{s^m}{m!} x^m \mathrm{d}\kappa_{in}(\cdot,x)\\
			&
			=\int \exp(sx)\mathrm{d}\kappa_{in}(\cdot,x)
		\end{align*}
		where the second equality is due to Fubini's theorem. Thus we have proved the claim. Then for the inequality $\E[\exp(|sW_{in}|)|\mathcal{F}_n]<\infty$, we can bound, on the event $\mathcal{A}$,
		\begin{align*}
			\E[\exp(|sW_{in}|)|\mathcal{F}_n]\leq \E[\exp(sW_{in})|\mathcal{F}_n]+\E[\exp(-sW_{in})|\mathcal{F}_n]<\infty, \forall s\in (-\varepsilon,\varepsilon).
		\end{align*}
		Thus we have proved the claim.
	\end{proof}

\end{document}